\documentclass[12pt,a4paper,leqno,oneside]{amschanged}

\usepackage{graphicx}
\usepackage{mathtools}

\usepackage[colorlinks=true, pdfstartview=FitV, linkcolor=blue, citecolor=blue, urlcolor=blue,pagebackref=false]{hyperref}

\usepackage{marginnote}
\usepackage{wasysym}
\usepackage{ragged2e}

\def\Tiny{\fontsize{7pt}{7pt} \selectfont}

\usepackage{xcolor}

\hypersetup{
    colorlinks,%
    citecolor=black,%
    filecolor=black,%
    linkcolor=black,%
    urlcolor=black
}

\usepackage{amsmath}
\usepackage{amssymb}
\usepackage{amsthm}
\newcommand{\mathbbm}{}

\usepackage{atbeginend}

\usepackage{psfrag}

\usepackage{fullpage}

\usepackage{enumerate}

\newtheorem{theorem}{Theorem}[section]
\newtheorem{corollary}[theorem]{Corollary}
\newtheorem{proposition}[theorem]{Proposition}
\newtheorem{lemma}[theorem]{Lemma}

\numberwithin{equation}{section}

\theoremstyle{definition}

\theoremstyle{remark}
\newtheorem{remark}[theorem]{Remark}

\DeclareMathOperator\capacity{cap}

\DeclareMathOperator\dist{dist}

\DeclareMathOperator\Cov{Cov}
\DeclareMathOperator\diam{diam}
\def\d{\mathrm{d}}

\newcommand{\vvviiiggg}{\Big}

\newcommand{\m}{\eta}
\newcommand{\M}{L}
\newcommand{\Exp}{\mathop{\mathrm{Exp}}}

\newcommand{\eps}{\varepsilon}
\newcommand{\Z}{\mathbb{Z}}
\newcommand{\R}{\mathbb{R}}
\newcommand{\N}{\mathbb{N}}

\newcommand{\I}{\mathcal{I}}

\newcommand{\eqd}{\stackrel{\tiny d}{=}}

\newcommand{\V}{V}
\newcommand{\IP}{\mathbb{P}}
\newcommand{\IE}{\mathbb{E}}
\newcommand{\GP}{\mathcal{P}}

\newcommand{\1}[1]{{\mathbbm{1}}_{#1}}
\usepackage{courier}

\usepackage{array}
\newcolumntype{e}{>{\displaystyle}r @{\,} >{\displaystyle}c @{\,} >{\displaystyle}l}
  \newcounter{constant}

  \newcounter{iconst}
  \newcommand{\newiconst}[1]{\refstepcounter{iconst}\label{#1}}
  
\def\clap#1{\hbox to 0pt{\hss#1\hss}}

\usepackage{calc}

\def\arraypar#1{\parbox[c]{\textwidth - 2cm}{\centering #1}}
\usepackage{environ}
\NewEnviron{display}{
\begin{equation}\begin{array}{c}
  \arraypar{\BODY}
\end{array}\end{equation}
}
\def\urltilda{\kern -.15em\lower .7ex\hbox{\~{}}\kern .04em}
\allowdisplaybreaks

\begin{document}

% \fontsize{12}{14}\rm
% \addtolength{\abovedisplayskip}{.5mm}
% \addtolength{\belowdisplayskip}{.5mm}
%\AfterBegin{enumerate}{\addtolength{\itemsep}{2mm}}

\title{\LARGE \usefont{T1}{cmr}{m}{n} \selectfont
  C\lowercase{onditional decoupling of random interlacements}} 

\author{\normalsize \itshape S\lowercase{erguei} P\lowercase{opov} $^1$ \color{white} \tiny and}
\address{$^1$Department of Statistics, Institute of Mathematics,
 Statistics and Scientific Computation, University of Campinas --
UNICAMP, rua S\'ergio Buarque de Holanda 651,
13083--859, Campinas SP, Brazil\newline
e-mail: {\itshape \texttt{popov@ime.unicamp.br}}\newline
website: {\href{http://www.ime.unicamp.br/~popov/}{\itshape \texttt{www.ime.unicamp.br/\urltilda popov/}}}}

\author{\color{black} \normalsize \itshape C\lowercase{aio} A\lowercase{lves} $^2$}
\address{$^2$ Department of Statistics, Institute of Mathematics,
 Statistics and Scientific Computation, University of Campinas --
UNICAMP, rua S\'ergio Buarque de Holanda 651,
13083--859, Campinas SP, Brazil\newline
e-mail: {\itshape \texttt{narrowstreets@gmail.com}}}

\date{\today}

\newiconst{c:main'}
\newiconst{c:main''}
\begin{abstract}

We prove a conditional decoupling inequality for the model of random interlacements in dimension $d\geq 3$: the conditional law of random interlacements on a box (or a ball) $A_1$ given the (not very ``bad'') configuration on a ``distant'' set $A_2$ does not differ a lot from the unconditional law. The main method we use is a suitable modification
of the \emph{soft local time} method of~\cite{PopovTeixeira}, that 
allows dealing with conditional probabilities.

\noindent
\emph{Keywords and phrases.} Random interlacements, 
stochastic domination, soft local time
\newline
MSC 2010 \emph{subject classifications.}
Primary 60K35;
Secondary 60G50, 82C41. 
\end{abstract}

\maketitle

%\newpage

\section{Introduction}
\label{s_intro}
Random interlacements were introduced by Sznitman in~\cite{Szn10}, 
to model the trace of the simple random walk on the discrete torus $\Z_n^d:=\Z^d/n\Z^d$ or the discrete cylinder $\Z\times\Z^{d-1}$, in dimension~$d\geq 3$. Detailed treatments and reviews of recent 
results can be found in the recent books~\cite{CT12, DRS14, Szn12book}. Loosely speaking, the model of random interlacements in~$\Z^d$, $d\geq 3$, 
is a stationary Poissonian soup of bi-infinite simple random walk trajectories on the integer lattice. There is a parameter~$u>0$ 
entering the intensity measure of the Poisson process, the larger 
$u$~is the more trajectories are thrown in. The sites of~$\Z^d$ 
that are not touched by the trajectories constitute the \emph{vacant set}~$\mathcal{V}^u$, and the union of all trajectories constitutes the interlacement set $\I^u=\Z^d\setminus \mathcal{V}^u$. 
The random interlacements are constructed simultaneously for all $u>0$ in such a way that $\I^{u_1}\subset \I^{u_2}$ if $u_1<u_2$. 
In fact, the law of the 
vacant set at level~$u$ can be
uniquely characterized by the following identity: 
\begin{equation}
\label{eq_vacant>3}
 \IP[A\subset \mathcal{V}^u] = \exp\big(-u \capacity (A)\big),
\end{equation}
where $\capacity(A)$ is the \emph{capacity} of a finite 
set~$A\subset\Z^d$. Informally, the capacity measures how 
``big'' is the set  from the point of view of the walk, 
see Section~6.5 of~\cite{LawlerLimic} for formal definitions, or Section~\ref{s_def} below.

The model of random interlacements naturally has more independence 
built in than just one random walk on the torus or the cylinder (because on a fixed set one observes traces of \emph{independent} trajectories). 
Still, the analysis of random interlacements is difficult because of the long-range dependencies present there. For example, in~$(1.68)$ from~\cite{Szn10} we can see that
\begin{equation}
 \label{1-corr}
 \Cov(\1{x \in \mathcal{I}^u}, \1{y \in \mathcal{I}^u}) 
  \sim \frac{c_{d}u}{\|x-y\|^{d-2}} \quad\text{ as } \quad \| x-y \| \to \infty,
\end{equation}
which means that the ``degree of dependence'' decreases polynomially in the distance.

Naturally, one is interested in ``decoupling'' the events supported on distant regions; that is, to argue that they are approximately independent to a certain degree. One possible approach to quantify that degree is the following: given finite sets $A_1,A_2~\subset~\Z^d$ and functions  
$f_1:\{0,1\}^{A_1} \to [0,1]$ and $f_2:\{0,1\}^{A_2} \to [0,1]$ 
depending on the interlacements set intersected with $A_1$ and $A_2$ respectively, we have
\begin{equation}
 \label{basic_dec}
 \Cov_u(f_1,f_2) \leq c_{d} u
  \frac{\capacity(A_1) \capacity(A_2)}{\dist(A_1,A_2)^{d-2}},
\end{equation}
as proved in formula~$(2.15)$ of~\cite{Szn10}, see also $(8.1.1)$ in \cite{DRS14}. However, the polynomial error term in \eqref{basic_dec} can complicate one's life in many applications 
(and, e.g.\ in the case when the diameters of these sets are of the same order as the distance between them, \eqref{basic_dec} is simply of no use); on the other hand, while \eqref{basic_dec} can be improved 
to some degree \cite{BGP}, the error term there should always be at least polynomial, as \eqref{1-corr} shows. To circumvent this difficulty, one first may note that usually the ``interesting'' events/functions 
are  \emph{monotone} (i.e., increasing or decreasing). 
For e.g.\ increasing events, we know that their probabilities increase as the parameter~$u$ increases. Note also that the FKG inequality (see~\cite{Tei09}, Theorem~$3.1$) gives us
\begin{equation}
\label{main_incr}
 \IE^u[g_1 g_2] \geq \IE^{u}[g_1] 
\IE^{u}[g_2],
\end{equation}
for \emph{any} increasing functions $g_{1,2}$ with finite second moments. 
To complement the FKG inequality, we use \emph{sprinkling}, i.e., 
we slightly change the intensity of random interlacements in order to decrease the error term; this approach was used in \cite{Szn10} and \cite{Szn12}. Then, in particular, in \cite{PopovTeixeira} it was proved that 
\begin{equation}
\label{strong_dec}
 \IE^u[f_1 f_2] \leq \IE^{(1+\eps)u}[f_1] 
\IE^{(1+\eps)u}[f_2]
    + c_d(r+s)^d \exp(-c_d ' \eps^2us^{d-2});
\end{equation}
with $f_1:\{0,1\}^{A_1} \to [0,1]$ and $f_2:\{0,1\}^{A_2} \to [0,1]$ both increasing functions in the interlacements set, $r=\min(\diam(A_1),\diam(A_2))$, and $s=\dist(A_1,A_2)$. The same bound was also obtained for decreasing functions.

It is important to observe, however, that the decoupling in the above form may not always be useful for one's needs. Intuitively, one is tempted to understand inequalities like \eqref{basic_dec}  as ``what happens in one set does not influence a lot what happens in the other set''. 
Now, consider the following situation. Suppose that on top of the random interlacements we have some additional stochastic process (e.g., a random walk) that ``explores'' the interlacement set in some way. 
Assume that this process has already explored the interlacements in a given area, revealing a lot of information about it; think, for definiteness, that it simply revealed the interlacement set exactly. 
The probability of a particular configuration of the interlacement set is usually very small; so, \eqref{basic_dec} (even~\eqref{strong_dec}!) will blow up when one divides by that probability, because of the error term.
In fact, in the end of Section~\ref{s_def} we discuss a particular
model of the random walk on the interlacement set, where our 
main results turn out to be useful.

This justifies the need for \emph{conditional} decoupling, i.e., show that, given the configuration on some set, the law of the interlacement configuration on a distant set is still in some sense close to the unconditional law. This is what we are doing in this paper.
To prove our results, the main method we use is a suitable modification
(that allows dealing with conditional probabilities)
of the \emph{soft local time} method of~\cite{PopovTeixeira}.
We hope that this modification will be useful in other 
contexts, for instance, for dealing with the decoupling properties
of the \emph{loop measures} \cite{CS14}.

Another important observation is the following.
 There are strong connections between random interlacements and the Gaussian free field, see e.g. \cite{Szn12book, Szn12ecp}. In particular, there are decoupling inequalities similar to \eqref{basic_dec} and \eqref{strong_dec} for the Gaussian free field as well, see~\cite{PR15}. Notice, however, that the decoupling-with-sprinkling result for the 
Gaussian free field (Theorem~$1.2$ of \cite{PR15}) is \emph{already} conditional (the unconditional decoupling is obtained as a simple consequence, just by integration). On the other 
hand, note that the error terms in the conditional decoupling in the main result of this paper (Theorem~\ref{t_main1}) are much worse than that of \eqref{strong_dec}; related to this is the fact 
that in the conditional setting the minimal distance between sets that permits the result to work is much bigger. A comparison with the situation for the Gaussian free field suggests that, hopefully, there is still much room for improvement for the 
conditional decoupling for random interlacements.

\section{Definitions, notations and results}
\label{s_def}
In this section we will introduce the basic definitions, conventions and notation used in this paper. We will then be able to state our main result. We start by stating our convention regarding constants: 
$c$, $c'$, $c_1$, $c_2$, $c_3$,$\dots$ are always defined as strictly positive constants depending only on the dimension $d$. Constants can also change value from line to line, unless when the text explicitly states to the contrary.

We let $\| \cdot \|$ and $\| \cdot \|_{\infty}$ denote the Euclidean and $\ell_\infty$ norms in~$\Z^d$ respectively. For $x,y\in\Z^d$, we also let $\dist(x,y)\equiv \|x-y \|$. We say that 
two vertices $x,y\in\Z^d$ are neighbors when $\|x-y\|=1$, this notion introduces the usual nearest-neighbor graph structure in $\Z^d$. For~$x\in\Z^d$ and~$r\in\R_+$, we define
\begin{equation*}
B(x,r) :=\big\{ y\in\Z^d;\|y-x\| \leq r       \big\},
\end{equation*}
the discrete ball in the Euclidean norm centered on $x$ with 
radius $r$, and
\begin{equation*}
B_{\infty}(x,r) :=\big\{ y\in\Z^d;\|y-x\|_{\infty} \leq r       \big\},
\end{equation*}
the discrete ball in the $\ell_\infty$-norm centered on~$x$ 
with radius~$r$. Given a set $A\subseteq\Z^d$ we denote by
\begin{equation*}
A^C := \{x\in \Z^d;x\notin A\}
\end{equation*}
its complement and by
\begin{equation*}
\partial A := \big\{x\in A;\text{ there exists $y\in A^C$ 
such that $\|x-y\|=1$}\big\}
\end{equation*}
its (internal) boundary.

For any set $Z$ and any two functions $f,g: Z\mapsto\R$, we write $f(z)\asymp g(z)$ to denote the fact that there exist two strictly
positive constants, $c_1$ and $c_2$, such that $c_1 f(z) \leq g(z) \leq c_2 f(z)$ for all $z\in Z$. 
When $Z$ is equal to $\R$ we say that $f(z)=o(g(z))$ when $\frac{f(z)}{g(z)}$ goes to~$0$ as $z\rightarrow\infty$.

Given~$x\in\Z^d$, we let~$\IP_{x}$ denote the probability measure associated with the simple random walk in $\Z^d$ started at $x$. 
We will also let $(X_k,k\geq 0)$ denote the simple random walk process in $\Z^d$. Given a set $A\subset \Z^d$, we define the entrance 
time for the set $A$
\begin{equation*}
H_{A}:=\inf\big\{  k\geq 0;X_k\in A   \big\}.
\end{equation*}
We also let the hitting time for $A$ be defined as
\begin{equation*}
\tilde H_{A}:=\inf\big\{  k\geq 1;X_k\in A   \big\}.
\end{equation*}
When $A$ is finite we denote its harmonic measure by
\begin{equation*}
e_A(x)=\1{x\in A}\IP_x\big[  \tilde H_A=\infty  \big]\text{  for $x\in\Z^d$}.
\end{equation*}
We are then able to define the capacity of the set $A$ 
\begin{equation*}
\capacity (A):=\sum_{x\in A} e_A(x),
\end{equation*}
and the normalized harmonic measure
\begin{equation*}
\overline{e}_A(x):=e_A(x) \capacity(A)^{-1}.
\end{equation*}
We now write down the definition of the Green's function for the simple random walk in $\Z^d$: for $x,y\in\Z^d$, we let
\begin{equation*}
G(x,y):=\sum_{k\geq 0}\IP_{x}\big[ X_k=y \big].
\end{equation*}
Theorem~$1.5.4$ of~\cite{LawlerI} provides us with the 
following estimate on the Green's function:
\begin{equation}
\label{greenestimate}
G(x,y)\asymp \frac{1}{1+\|x-y\|^{d-2}}.
\end{equation} 

Let us briefly discuss the definition of the measure associated with the random interlacements process intersected with a given finite set $A\subset\Z^d$. Assume we have constructed a probability space where, for every $i\geq 1$, there exists a simple random walk process $(X^{(i)}_k,k\geq 0)$ with starting distribution given by 
$\overline{e}_A(\cdot)$, and such that $(X^{(i)}_k,k\geq 0)$ is independent from~$(X^{(j)}_k,k\geq 0)$ for~$i\neq j$. 
We also assume that in this space we can construct an independent Poisson process~$(J_u)_{u\geq 0}$ on the positive real line with intensity~$\capacity(A)$. The law of the random interlacements process $(\I^u)_{u\geq 0}$ intersected with the set $A$ can then be characterized by
\begin{equation}
\label{interlacementsdef}
(\I^u\cap A)_{u\geq 0} \eqd \Big(A\cap\bigcup_{i\leq J_u}\bigcup_{k\geq 0}X_k^{(i)} \Big)_{u\geq 0},
\end{equation}
as can be seen in \cite{Szn10}, Proposition~$1.3$, or in the paragraph before $(2.6)$ in \cite{CT14}. This definition gives rise to compatible measures in the following sense: 
Given two finite sets~$K_1\subset K_2\subset\Z^d$, we have that $((\I^u\cap K_2)_{u\geq 0})\cap K_1$ has the same law as $(\I^u\cap K_1)_{u\geq 0}$.

To state our main result, we need more definitions.
\begin{figure}[ht]
\centering
\includegraphics[scale = 1]{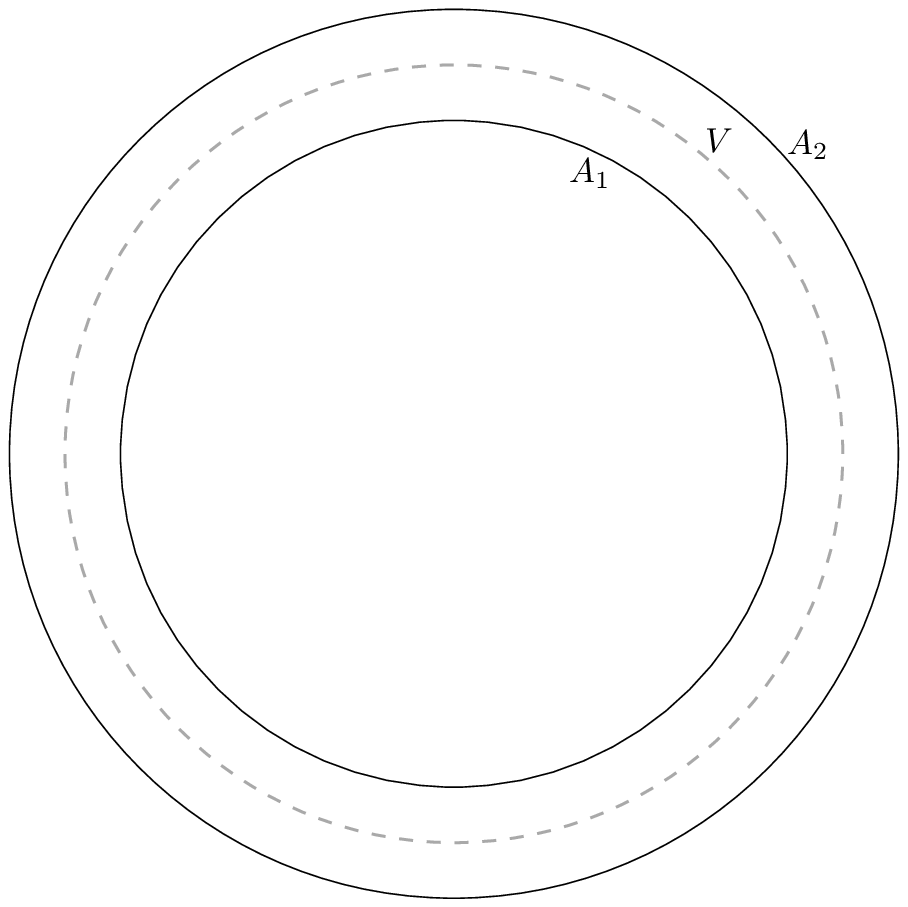}
\vspace{0.5cm}
\caption{Definition of the sets $A_1^{\tiny\Circle}$, $A_2^{\tiny\Circle}$ and $V^{\tiny\Circle}$.}
\label{processdeffig1}
\end{figure}
\begin{figure}[ht]
\centering
\includegraphics[scale = 1]{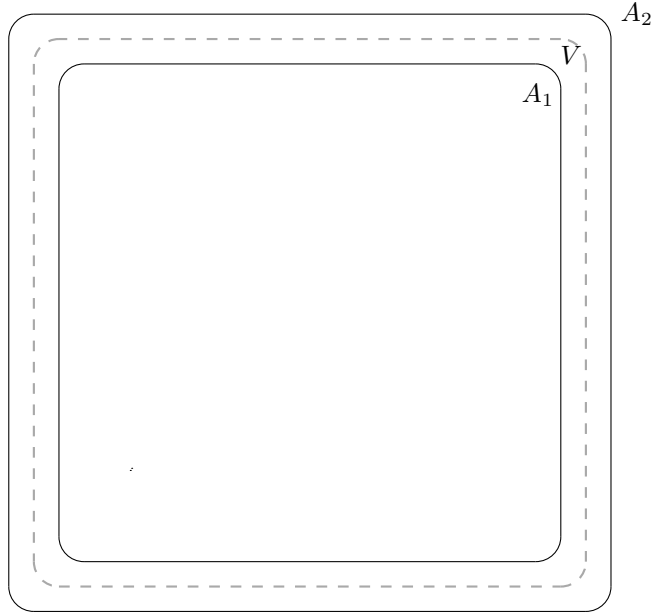}
\vspace{0.5cm}
\caption{Definition of the sets $A_1^{\tiny\Square}$, $A_2^{\tiny\Square}$ and $V^{\tiny\Square}$.}
\label{processdeffig1square}
\end{figure}
Let $r>0$ be sufficiently big, and let $s:=s(r)>0$, 
with $s=o(r)$. We define $A_1^{\tiny\Circle}:=A_1^{\tiny\Circle}(r)$ to be the discrete ball of radius $r$, that is
\begin{equation*}
A_1^{\tiny\Circle}:=\{x_1\in\Z^d;\dist(x_1,0)<r\}.
\end{equation*}
We also define $A_1^{\tiny\Square}:=A_1^{\tiny\Square}(r,s)$ to be a $d$-dimensional discrete `hypercube' with edge length~$r$ and a smoothed frontier such that for every point $x_1\in\partial A_1$ there exists a discrete Euclidean ball~$B_{x_1}$ of 
radius $s$ contained in $A_1$ such that $B_{x_1}\cap A_1^C=x_1$. More precisely, we let $\mathfrak{H}_{r-s}$ be a discrete $d$-dimensional hypercube  with edge length $r-s$ contained in~$\Z^d$ and define 
\begin{equation*}
A_1^{\tiny\Square}:=\{x_1\in\Z^d;\dist(x_1,\mathfrak{H}_{r-s})\leq s\}.
\end{equation*}
We refer the reader to \cite{PopovTeixeira}, Section~$8$, to see that $A_1^{\tiny\Square}$ possesses the desired properties. Note that, since $s=o(r)$, the diameter of $A_1^{\tiny\Square}$ is of order $r$.

We then define $A_2^{\tiny\Circle}:=A_2^{\tiny\Circle}(r,s)$ to be the set of points that are at least at distance $2s$ from $A_1^{\tiny\Circle}$:
\begin{equation*}
A_2^{\tiny\Circle}:=\{x_1\in\Z^d;\dist(x_1,x_2)>2s\text{ for every $x_2\in$}A_1^{\tiny\Circle}\}.
\end{equation*}
We finally define $V^{\tiny\Circle}:=V^{\tiny\Circle}(r,s)$ to be the boundary set
\begin{equation*}
V^{\tiny\Circle}:=\partial\{x_1\in\Z^d,\dist(x_1,x_2)\leq 
s\text{ for some $x_2\in$}A_1^{\tiny\Circle}\},
\end{equation*}
separating $A_1^{\tiny\Circle}$ from $A_2^{\tiny\Circle}$. We analogously define $A_2^{\tiny\Square}(r,s)$ and $V^{\tiny\Square}(r,s)$.  
It will also be useful to define the $d$-dimensional hypercube $\mathfrak{H}_{r+2s}$ of edge length $r+2s$ concentric with~$\mathfrak{H}_{r-s}$,  which will essentially be the unsmoothed version of $(A_2^{\tiny\Square})^C$. 

When there is no risk of confusion, 
or when the arguments presented work for both balls and 
smoothed hypercubes (which will be often so), we will omit the super-indexes~$\tiny\Circle,\tiny\Square$. 

Since $s=o(r)$, we have
\begin{equation*}
\capacity(V)=\capacity(A_2)(1+o(1))=\capacity(A_1)(1+o(1)),
\end{equation*}
and also, by Proposition~$2.2.1$ and equation~$(2.16)$ of  \cite{LawlerI},
\begin{equation}
\label{asymcapv}
\capacity(V)\asymp r^{d-2}.
\end{equation}

We will now state our main result. Heuristically, it says the following: Let~$s$ be bounded from below by a polynomial of~$r$ with a explicit given coefficient (strictly smaller than~$1$, depending only on the dimension~$d$ and whether~$A_1$ is a ball or a smoothed hypercube). Let~$A_3$ be a subset of~$A_2$ with finite boundary, that is, $A_3$ is either finite or has finite complement. If we pay a stretched exponentially small price (in~$s$) to guarantee that the interlacements configuration of~$\I^u\cap A_3$ is not too weird, then the distribution of~$I^u\cap A_1$ conditioned on this configuration is well approximated by the unconditional distribution, with high probability ($1$ minus a stretched exponential function of~$s$).
\begin{figure}
\centering
\includegraphics[scale = 0.75]{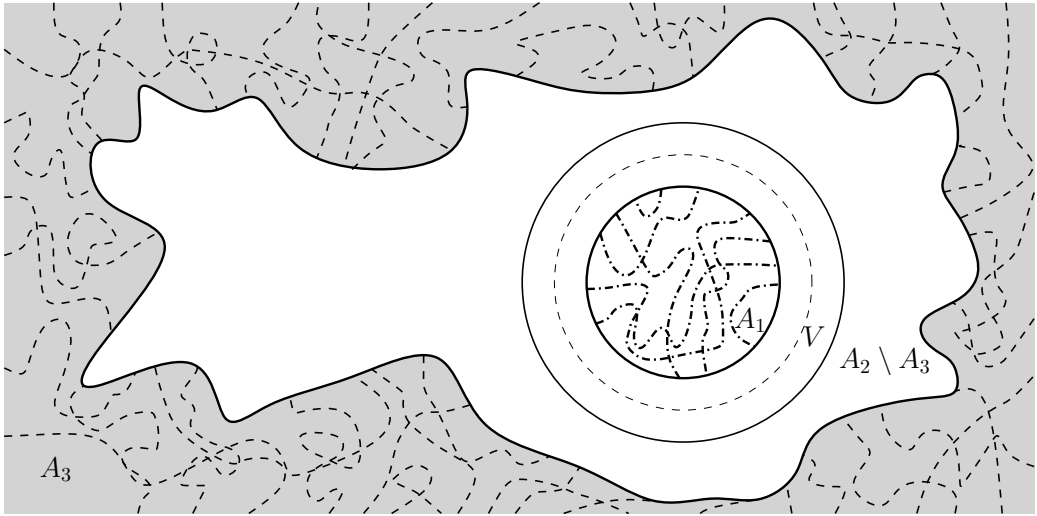}
\vspace{0.5cm}
\caption{Our main result says that if the interlacements configuration in a set~$A_3\subseteq A_2$ is not too weird, that is, it does not belong to a set with stretched exponentially small probability (in~$s$, as~$s\rightarrow\infty$), then with high probability ($1$ minus stretched exponential in~$s$) the distribution of the interlacements set intersected with~$A_1$ conditioned on the state of~$\I^u\cap A_3$ can be well approximated by the unconditional distribution.}
\label{A3def}
\end{figure}
\begin{theorem}
\label{t_main1}
Let the real numbers $b_{A_1^{\tiny\Circle}},b_{A_1^{\tiny\Square}}$ 
be such that
\begin{align}
\label{bcircledef}
1\leq b_{A_1^{\tiny\Circle}}<\frac{2d-2}{d},\\
1\leq b_{A_1^{\tiny\Square}}<\frac{4d-4}{3d-2}\label{bsquaredef}.
\end{align}
Then, define
\begin{align}
\label{acircledef}
a_{A_1^{\tiny\Circle}}&=2d-2-db_{A_1^{\tiny\Circle}}>0,\\
a_{A_1^{\tiny\Square}}&=4d-4-3db_{A_1^{\tiny\Square}}+2b_{A_1^{\tiny\Square}}>0.\label{asquaredef}
\end{align}
From now on we will again omit the indexes $\tiny\Circle,\tiny\Square$. Recall that $r$ is of the same order as the 
diameter of $A_1$, and that $s$ has the same order as the distance between $A_1$ and~$A_2$. Assume $r\asymp s^{b_{A_1}}$, 
let $s$ be sufficiently big. Let $\eps>0$ be smaller then $1/4$. Let~$A_3$ be a subset of~$A_2$ such that~$|\partial A_3|<\infty$. Define $\I^u_{A_j}:=\I^u\cap A_j$, for~$j=1,2,3$.

Then there are positive constants~$c,c'$ depending only 
on the dimension $d$, and a measurable (according to the random interlacements $\sigma$-field) set~$\mathcal{G}\in\{0,1\}^{A_3}$ 
such that
\begin{equation*}
\IP^u \big{[} \I^u_{A_3}\in\mathcal{G}\big{]}\geq 1-\exp\Big(-\frac{c'}{2}\eps^2 u s^{a_{A_1}}\Big),
\end{equation*}
and for any increasing function $f$ on the interlacements set intersected with $A_1$, with $\sup |f| < M$, we have 

\begin{align}
\label{e_conditionaldecoupling2}
\lefteqn{\big(\IE f( \I^{u(1-\eps)}_{A_1})-cM\exp\big(-c'\eps^2 u s^{a_{A_1}}\big)\big)\1{\I^u_{A_3}\in\mathcal{G}} 
\leq  \IE \big(f(\I^{u}_{A_1})\mid \I^{u}_{A_3} \big)\1{\I^u_{A_3}
\in\mathcal{G}}\nonumber} \phantom{**********************}
\\
\phantom{*****************}&\leq 
 \big(\IE f( \I^{u(1+\eps)}_{A_1})
+cM\exp\big(-c'\eps^2 u s^{a_{A_1}}\big)\big)\1{\I^u_{A_3}\in\mathcal{G}}. 
\end{align}

\end{theorem}
We also obtain a result analogous to Theorem~\ref{t_main1}, 
but this time we allow the sprinkling factor to be arbitrarily big.
This decreases the ``precision'' (in the result below, 
$\IE f( \I^{u + u'}_{A_1})$ can be 
very different from~$\IE f( \I^{u}_{A_1})$), but, in compensation,
the size of the complement of the ``good'' set as well as the 
``error term'' become smaller.
\begin{theorem}
\label{t_main3}
Let $u'>u>0$. We use the same definitions as Theorem~\ref{t_main1}.
There are positive constants~$c,c'$ depending only 
on the dimension $d$, and a measurable (according to the random interlacements $\sigma$-field) 
set~$\mathcal{G}_{u'}\in\{0,1\}^{A_3}$ such that
\begin{equation*}
\IP^u \big{[} \I^u_{A_3}\in\mathcal{G}_{u'}\big{]}
\geq 1-\exp\Big(-c' u' s^{a_{A_1}}\Big),
\end{equation*}
and for any increasing function $f$ on the interlacements set intersected with $A_1$, with $\sup |f| < M$, we have
\begin{equation}
\label{e_conditionaldecoupling3}
\IE\big(f(\I^{u}_{A_1})\mid \I^{u}_{A_3} \big)
\1{\I^u_{A_3}\in\mathcal{G}_{u'}} 
\leq 
 \big(\IE f( \I^{u + u'}_{A_1})+cM\exp\big(-c' u' s^{a_{A_1}}\big)\big)\1{\I^u_{A_3}\in\mathcal{G}_{u'}}. 
\end{equation}
\end{theorem}
\begin{remark}
We have to explain why we need to consider $A_3\subset A_2$. Indeed, at first sight it seems that conditioning on a configuration on~$A_3$ does not add generality to our results, since any fixed configuration on~$A_3$ corresponds to a set of configurations on~$A_2$. However, the problem with always setting $A_3=A_2$ is the following: the ``exceptional set''~$\mathcal{G}^c$ will then be supported on the whole~$A_2$, and this can be inconvenient for applications. For example, assume that we successively apply the conditional decoupling results to a process (such as the one of Section~\ref{s_applic}) that ``explores'' the interlacement environment. If that process has explored only a finite chunk of~$A_2$, we would not be able to say if the configuration is ``good'' (i.e., belongs to~$\mathcal{G}$) by only observing that finite chunk. This would force us to condition on the (configuration on the) whole~$A_2$, which would mean that a subsequent application of a conditional decoupling may be difficult, since we already ``revealed'' some information about the configuration on a set which is ``too big'' (i.e., when we apply the decoupling result for the next time, the ``new'' $A_1$ may be inside the ``previous'' $A_2$)
\end{remark}

\begin{remark}
In the course of the proof of the above theorems we actually prove a stronger result: the same conditional decoupling inequality  holds true if we replace the sets $\I^u_{A_1}\subset A_1$ and $\I^u_{A_3}\subset A_3$ by sets of \emph{random walk excursions} in $A_1$ and $A_3$ (we also have to replace the function $f$ by an increasing function on 
the set of excursions). That is, the conditional decoupling continues to work when we replace the ranges of the excursions (which constitute the random interlacements set) by the actual excursions themselves. 
We chose to state the results in the above manner for the sake of clarity and brevity. Note that this remark also applies to the decoupling obtained by Popov and Teixeira in~\cite{PopovTeixeira}.  
\end{remark}

\begin{remark}
The above theorems can be proved in the same way if we replace the smoothed hypercube $A_1^{\tiny \Square}$ by a smoothed version of a box $[0,a_1]\times\dots\times[0,a_d]$, with~$c^{-1}r<a_i<cr$ for all $i=1,\dots,d$, and some constant~$c>1$, and then replace the sets $A_2^{\tiny \Square}$ and~$V^{\tiny \Square}$ accordingly. 
We chose to prove the theorems for~$A_1^{\tiny \Square}$ only to simplify the notation. We also note that we prove the theorem for both balls and boxes because the error term obtained in the decoupling for balls is much smaller than the error obtained in the decoupling for boxes, 
but at the same time the decoupling between boxes tends to be more useful because boxes cover the space in a much more efficient manner.  
\end{remark}

\begin{remark} For $d=3$, the only way to obtain an exponentially small (instead of a \emph{stretched exponentially} small) error term in equations~\eqref{e_conditionaldecoupling2} and~\eqref{e_conditionaldecoupling3} is to allow the distance~$\sim s$ between the sets~$A_1$ and~$A_2$ to be of the same order of the minimal diameter~$\sim~r$.
\end{remark}

Here is an overview of the paper. In Subsection~\ref{s_applic},
we discuss an application of some of our results. 
In Section~\ref{s_slt} we recall the soft local times technique. In Section~\ref{s_simexc} we show how we simulate the interlacements set~$\I^u_{A_1}$ conditioned on the information given by~$\I^u_{A_2}$ using a suitable version of 
the soft local times method. Finally, in Section~\ref{s_cd}, we prove the main theorem using a large deviations estimate for the soft local times associated with~$\I^u_{A_1}$. The Appendix is then used to collect and derive the technical estimates we need.

\subsection{An application: biased random walk on the interlacement set}
\label{s_applic}
Let~$G$ be some (possibly random) subset of~$\Z^d$, $d\geq 2$.
Fix a parameter~$\beta>0$, 
which accounts for the bias; also, fix some non-zero vector~$\ell\in\Z^d$.
Let us define the \emph{conductances}
on the edges of~$\Z^d$ in the following way:
\[
 \mathcal{C}(x,y) = \begin{cases}
           e^{\beta(x+y)\cdot \ell}, & \text{if $x,y$ are 
             neighbors and belong to $G$}, \\
           0, & \text{otherwise},
          \end{cases}
\]
and we call the collection of all conductances
$\omega = \big\{\mathcal{C}(x,y), x,y\in\Z^d\big\}$ the random environment.
Consider a random walk $(X_n, n\geq 0)$ in this environment
of conductances; i.e., its transition probabilities
are given by
\[
 P^{\omega}[X_{n+1}=y \mid X_n=x] 
= \frac{\mathcal{C}(x,y)}{\sum_{z}\mathcal{C}(x,z)}
\]
(the superscript in~$P^{\omega}$ indicates that we are dealing
with the ``quenched'' probabilities, i.e., when the underlying 
random graph / conductancies are already fixed).

There have been significant interest towards this model in recent years,
mainly in the case when~$G$ is the infinite cluster of supercritical
Bernoulli percolation model, see e.g.\ 
\cite{BergerGantertPeres, Szn03, Fribergh}. 
In particular, one remarkable 
fact is the following: the walk is ballistic (transient and with 
positive speed) in the direction of the drift if $\beta>0$
is small enough; however, it moves only sublinearly fast 
(its displacement is only of order~$t^a$ by time~$t$ with~$a\in(0,1)$,
as proved in~\cite{FriberghHammond})
for large values of~$\beta$. 

In the work~\cite{FriberghPopov} the case $G=\I^u$ was considered.
It turned out that in dimension~$d=3$, for \emph{any} value
of~$\beta>0$, although still transient in the direction of the 
drift, the walk is not only sub-ballistic, but has also sub-polynomial
speed, in the sense that its distance to the origin grows
 slower than~$t^\eps$ for any $\eps>0$.
This is also in contrast with the result that the walk on~$\I^u$
without any drift is diffusive (so, loosely speaking, its ``speed''
is~$\sqrt{t}$),
as shown in~\cite{ProcacciaRosenthalSapozhnikov}.

We will not describe all the details of~\cite{FriberghPopov} here,
but the main idea is the following. As in the case of the biased
walk on the infinite percolation cluster, to prove zero speed
one needs to show that the walk frequently gets caught in traps.
These traps are ``dead ends'' of the environment looking in the 
direction of the bias, see Figure~\ref{f_trap}. 
\begin{figure}
 \centering \includegraphics{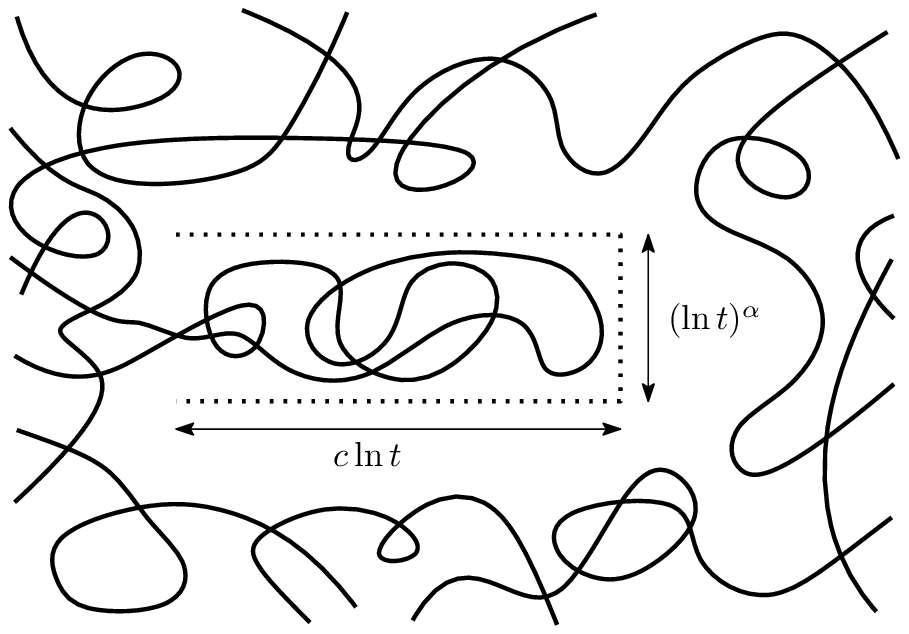} 
\caption{A trap for the random walk on the 
interlacement set (on this picture,
the bias is directed along the first coordinate vector).
 Only the interlacements are shown; 
the trajectory of the RWRE~$X$ is not present
on the picture.}
\label{f_trap}
\end{figure}
When the walk enters such a trap, the bias prevents it from
goint out, so there is a good chance that the walk will spend
quite a lot of time there, and this effectively leads to zero speed.
Now, the crucial fact is that, specifically in three dimensions,
it is much cheaper to have a trap in the interlacement set than 
in the (Bernoulli) percolation cluster. Indeed, it is possible 
to show that the capacity of the dotted set on Figure~\ref{f_trap}
is of order $\frac{\ln t}{\ln \ln t}$ for any fixed $\alpha<1$.
The formula~\eqref{eq_vacant>3} then shows that having a trap
as above has only a subpolynomial (in~$t$) cost; also, it turns
out that ``forcing'' a trajectory to create a ``dead end'' as shown
on the picture is not too costly as well. 

So, when the walk advances in the direction of the bias, 
 from time to time it will encounter a trap and be trapped.
However, to make such an argument rigorous, one has to face
the following difficulty. When the walk already explored 
some parts of the environment and then came to an unexplored area, 
we can no longer use~\eqref{eq_vacant>3} 
to estimate the probability that there is a trap in front of it,
due to the lack of independence.
It is here that the conditional decoupling enters the scene:
it is possible to use the main results of this paper to show
that probability of having a trap in front of the particle
(when it comes to an unexplored area) is not very small.
As mentioned above, the detailed argument can be found 
in~\cite{FriberghPopov}.

\section{Soft local times}
\label{s_slt}

In the present section we describe the technique introduced in \cite{PopovTeixeira}, the so called Soft Local Times method. This method essentially allows us to simulate any number of random variables taking values in a state space~$\Sigma$ using a realization of a Poisson point process in~$\Sigma\times\R_+$.

Let~$\Sigma$ be a locally compact Polish metric space,
 and let~$\mathcal{B}(\Sigma)$ be its Borel~$\sigma$-algebra. Let~$\mu$ be a Radon measure over~$\mathcal{B}(\Sigma)$, so that every compact set has finite $\mu$-measure.

Such measure space~$(\Sigma,\mathcal{B}(\Sigma),\mu)$ is the 
usual setup for the construction of a Poisson point process on~$\Sigma$. We consider the space of Radon point measures in~$\Sigma\times\R_+$
\begin{equation}
 \label{e_M1}
 \M = \vvviiiggg\{\m = \sum_{\lambda \in \Lambda} \delta_{(z_\lambda,
 v_\lambda)}; z_\lambda \in \Sigma, v_\lambda \in \mathbb{R}_+ 
\text{ and } \m(K) < \infty \text{ for all compact $K$} \vvviiiggg\},
\end{equation}
endowed with the $\sigma$-algebra generated by the evaluation maps \begin{equation*}
\eta\mapsto\eta(D),\phantom{*}D\in\mathcal{B}(\R_+)\otimes\mathcal{B}(\Sigma).
\end{equation*} 
We are then able to construct a Poisson point process~$\eta$ 
in the space $(L,\mathcal{D},\mathbb{Q})$ with intensity measure given by $\mu\otimes\d v$, where~$\d v$ is the Lebesgue measure on~$\R_+$, see \cite{Resnick1}, Proposition~$3.6$ on p.$130$.

The next proposition, originally seen in \cite{PopovTeixeira}, is at the core of the soft local times argument.

\begin{proposition}
\label{p_simslt}
Let $g:\Sigma \to \mathbb{R}_+$ be a measurable function with $\int g(z) \mu(\d z) = 1$. For $\m = \sum_{\lambda \in \Lambda} \delta_{(z_\lambda, v_\lambda)} \in \M$, we define
\begin{equation}
 \xi = \inf \{ t \geq 0; \text{ there exists $\lambda \in \Lambda$ 
such that $t g(z_\lambda) \geq v_\lambda$}\}.
\end{equation}
Then under the law $\mathbb{Q}$ of the Poisson point process~$\m$,
\begin{enumerate}[(i)]\addtolength{\itemsep}{2mm}\vspace{2mm}
 \item \label{e:io} there exists a.s.\
 a unique $\hat{\lambda} \in \Lambda$ such that 
 $\xi g(z_{\hat{\lambda}}) = v_{\hat{\lambda}}$,
 \item \label{e:xiio} $(z_{\hat{\lambda}}, \xi)$ is distributed as $g(z) \mu(dz) 
\otimes \Exp(1)$,
 \item \label{e:mprime} $\m' := \sum_{\lambda \neq \hat{\lambda}} \delta_{(z_\lambda,v_\lambda - 
\xi g(z_\lambda))}$ has the same law as~$\m$ and is independent 
of $(\xi, \hat{\lambda})$.
\end{enumerate}
\end{proposition}

The proof is remarkably simple, mainly relying on the independence of a Poisson process in disjoint sets, and can be seen in the original paper.

With the above proposition we are able to simulate as many 
random variables as we want:

\begin{figure}[ht]
\centering
\includegraphics[scale = 1]{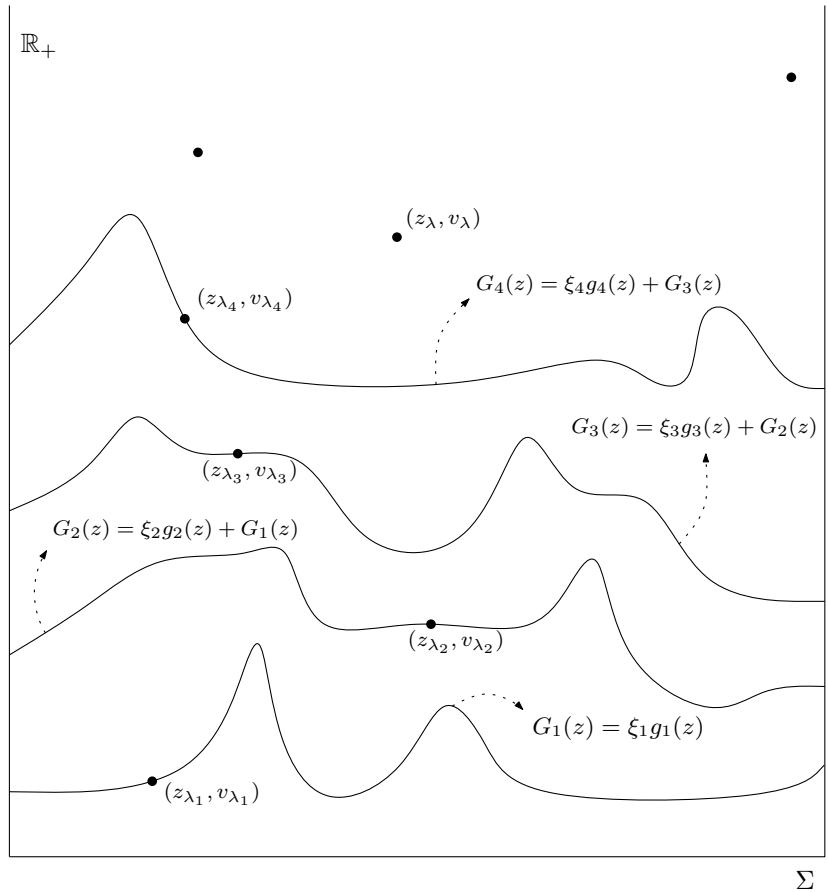}
\vspace{0.5cm}
\caption{An example showing the definition below. Under mild conditions we are able to use Proposition $\ref{p_simslt}$ to simulate a sequence of random variables over $\Sigma$.}
\label{softlocaltimesfig1}
\end{figure}

Let $X_1,\space X_2,\dots,\space  X_n$ be random variables on $\Sigma$ such that $X_1$'s distribution is absolutely continuous with respect to $\mu$ and,  for all $i=2,\dots,n$ the probability measure generated by $X_i$, conditioned on on the values taken by $X_1,\dots,\space  X_{i-1}$, is absolutely continuous with respect to $\mu$. Using the process $\m$ constructed above, we define
\begin{equation}
\begin{aligned}
 & g_1:\Sigma\mapsto\R_+\text{, the density function of $X_1$ with respect to $\mu$,}\\
 & \xi_{1} := \inf \big\{ t \geq 0; \text{ there exists $\lambda \in \Lambda$ such that $t g_1( z_\lambda) \geq v_\lambda$}\big\}, \\
 & G_{1}(z) := \xi_{1} \; g_1(z), \text{ for $z \in \Sigma$,}\\
 & (z_{\lambda_1}, v_{\lambda_1}) \text{, the unique pair in $\{(z_\lambda, v_\lambda)\}_{\lambda \in \Lambda}$ with  $G_1( z_{\lambda_1}) = v_{\lambda_1}$.}
\end{aligned}
\end{equation}
We now define $g_2:\Sigma\mapsto\R_+$ to be the density 
of $X_2$ conditioned on the event $\{X_1=z_{\lambda_1}\}$. Using the fact that $\m_1 := \sum_{\lambda \neq {\lambda_1}} \delta_{(z_\lambda,v_\lambda - 
\xi_1 g_1(z_\lambda))}$ has the same law as $\m$ and is independent 
from $(\xi_1, {\lambda_1})$ we define
\begin{equation}
\begin{aligned}
 & \xi_{2} := \inf \big\{ t \geq 0; \text{ there exists $\lambda \in \Lambda$ such that $t g_2( z_\lambda)+G_1(z_{\lambda}) \geq v_\lambda$}\big\}, \\
 & G_{2}(z) := \xi_{2} \; g_2(z)+G_1(z), \text{ for $z \in \Sigma$,}\\
 & (z_{\lambda_2}, v_{\lambda_2}) \text{, the unique pair in $\{(z_\lambda, v_\lambda)\}_{\lambda \in \Lambda}$ with $ G_2( z_{\lambda_2}) = v_{\lambda_2}$.}
\end{aligned}
\end{equation}
Then, recursively, for $1\leq k\leq n$ we define $g_k:\Sigma\mapsto\R_+$ to be the density function of $X_k$ conditioned on the event $\{X_1=z_{\lambda_1},\dots,X_{k-1}=z_{\lambda_{k-1}}\}$,
\begin{equation}
\begin{aligned}
 & \xi_{k} := \inf \big\{ t \geq 0; \text{ there exists $\lambda \in \Lambda$ such that $t g_k( z_\lambda)+G_{k-1}(z_{\lambda}) \geq v_\lambda$}\big\}, \\
 & G_{k}(z) := \xi_{k} \; g_{k}(z)+G_{k-1}(z), \text{ for $z \in \Sigma$,}\\
 & (z_{\lambda_k}, v_{\lambda_k}) \text{, the unique pair in $\{(z_\lambda, v_\lambda)\}_{\lambda \in \Lambda}$ with $ G_k( z_{\lambda_k}) = v_{\lambda_k}$.}
\end{aligned}
\end{equation}
We refer to Figure~$\ref{softlocaltimesfig1}$. Using Proposition \ref{p_simslt} together with the above construction, we are able to state the following proposition:

\begin{proposition}
\label{p_decslt}
The vector $(z_{\lambda_1},\dots,z_{\lambda_n})$ has the same law as $(X_1,\dots,X_n)$.
\end{proposition}

We call the function $G_n(z)$ the soft local time of the vector $(X_1,\dots,X_n)$ up to time~$n$ with respect to the measure $\mu$, or more usually simply the soft local time. If $T$ is a stopping 
time with respect to the canonical filtration generated by the variables $X_i$, it is simple to define $G_T(z)$, the soft local time up to time $T$.

Note that by controlling the value of the soft local times function we will automatically control the values our random variables take, as the next corollary summarizes:

\begin{corollary}
\label{c_couplez}
For any measurable function $h: \Sigma \to \R_+$ we have, using the same notation as above,
\begin{equation}
 \label{e_couplesingle}
 \mathbb{Q} \Big[ \{z_1, \dots, z_T\} \subseteq \{z_{\lambda}; v_{\lambda} \leq h(z_{\lambda})\} \Big] \geq \mathbb{Q} \big[ G_T(z) \leq h(z), \text{ for $\mu$-a.e. $z \in \Sigma$} \big],
\end{equation}
for any finite stopping time $T \geq 1$.
\end{corollary}

\section{Simulating excursions}
\label{s_simexc}

In this section we will show a way of simulating the intersection of the random interlacements set with a given subset of $\Z^d$ in such a way as to make explicit the dependence each random walk excursion has with its entrance and exit points on the subset. We refer the reader to Figure~$\ref{clotheslineslt}$ for a brief overview of the arguments used in this section.
\begin{figure}[ht]
\centering
\includegraphics[scale = 1]{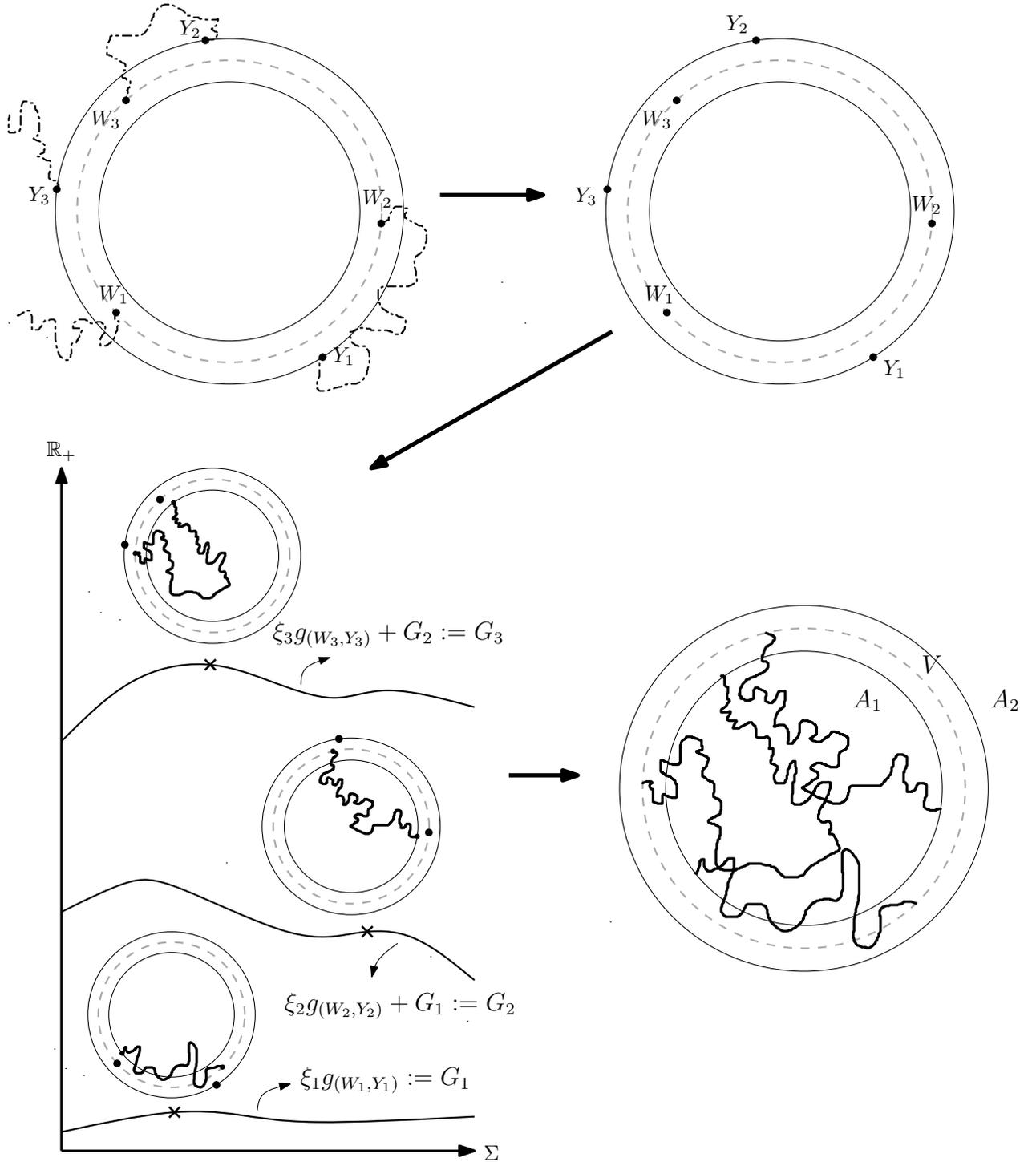}
\caption{The figure shows how we will use the soft local times technique to simulate the range of a simple random walk trajectory intersected with~$A_1$. We first simulate a process of 
pairs of points $((W_k,Y_k),k\geq 0)$ denoting the entrance at $V$ and exit at $\partial A_2$ of a simple random walk trajectory that 
starts at $V$. We then use the soft local times method to simulate the pieces of trajectory that lie between each of the pairs $(W_k,Y_k)$.}
\label{clotheslineslt}
\end{figure}

It is clear from \eqref{interlacementsdef} the fact that in order to simulate the random interlacements set at level $u$ in a bounded subset $K$ of $\Z^d$ we need only to pick a~$N^u_K\eqd\textit{Poisson}(u\capacity(K))$ number of points in $\partial K$, each point chosen  according to the measure~$\overline e_K (\cdot)$, and from each point start a simple random walk.

We intend to study $\I^u_{A_1}=\I^u\cap A_1$, showing that this set is not much influenced by the random interlacements set intersected with $A_2$, $\I^u_{A_2}=\I^u\cap A_2$. We will later clarify what we mean by ``influence". For now, we observe that the only 
``information" $\I^u_{A_1}$ receives from $\I^u_{A_2}$ is the location of the entrance and exit points of the excursions on $\partial A_2$ of the random walks that constitute $\I^u_{A_2}$. 

Let us begin the work towards our result. We first generate the points of entrance at~$V$ and exit from $A_2^C$ of each excursion on $V$ of a random walk trajectory. These points will be the clothesline onto 
which we will hang the pieces of trajectory that meet $A_1$, we will do so using the soft local times method.

Let us define the successive return and departure times between $V$ and $A_2$. Given a trajectory that starts at $V$, we define
\begin{align}
\nonumber
D_0 &= 0, \qquad & & R_1 = H_{\partial A_2},\\
D_1 & = H_{\V} \circ \theta_{R_1} + R_1, \qquad & & R_2 = H_{\partial A_2} \circ \theta_{D_1} + D_1,\label{exctimedef}\\
D_2 & = H_{\V} \circ \theta_{R_2} + R_2 \qquad & & \text{and so on.} \nonumber
\end{align}
We also define the random time
\begin{equation}
 \label{e_tdelta}
 T_\Delta = \inf \{k \geq 1; R_k = \infty\},
\end{equation}
which is almost surely finite, as the walk is transient.

Let $(X_n,n\geq 0)$ be the simple random walk with initial distribution given by $\overline e_V (\cdot)$. 
Let~$\Delta$ be an artificial cemetery state. 
We construct a random sequence of elements of $(V\times \partial A_2)\cup\{\Delta\}$ in the following way: Conditioned on the event $\{T_{\Delta}=m\}$, we let
\begin{align*}
\lefteqn{\big((W_1,Y_1),\dots,(W_{m-1},Y_{m-1}),(W_{m},Y_{m}),(W_{m+1},Y_{m+1}),\dots\big)}
\phantom{********************}\\
&=
\big((X_{D_0},X_{R_1}),\dots,(X_{D_{m-2}},X_{R_{m-1}}),\Delta,\Delta,\dots\big).
\end{align*}
It is then elementary to prove that the process $((W_k,Y_k))_{k\geq 1}$ inherits the Markov property from the simple random walk. We call $((W_k,Y_k))_{k\geq 1}$ the clothesline process started at~$W_1$. When there is no risk of confusion we will also denote by $\IP_{w_0}$ the probability measure associated with the clothesline process started at a given point $w_0\in V$.
\begin{figure}[ht]
\centering
\includegraphics[scale = .8]{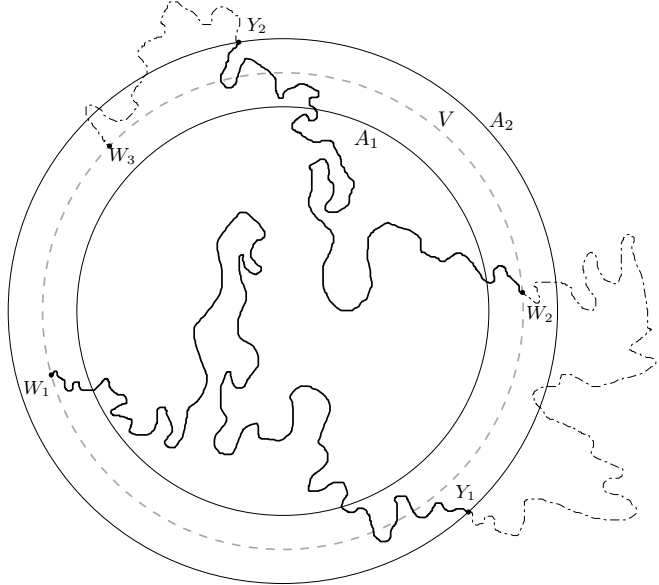}
\vspace{0.5cm}
\caption{An example of the process $((W_k,Y_k))_{k\geq 1}$.}
\label{processdeffig2}
\end{figure}

Let us now use the soft local times method to generate the trajectories inside $A_1$, given the entrance and exit points $((W_k,Y_k))_{k\geq 1}$.
 We first define the underlying space $\Sigma$ where our pieces of trajectories will live. We let $\mathcal{K}$ be the set of nearest-neighbor paths in $A_2^C$ with one endpoint in $\partial A_1$ and the other in $V$,
\begin{equation}
\mathcal{K} := \big\{(x_0,x_1,\dots , x_n);\space n\in\N,\space x_i\in A_2^C\text{ for $1\leq i\leq n$},\space x_0\in\partial A_1,\space x_n\in V\big\}.
\end{equation}
\begin{figure}[ht]
\centering
\includegraphics[scale = .8]{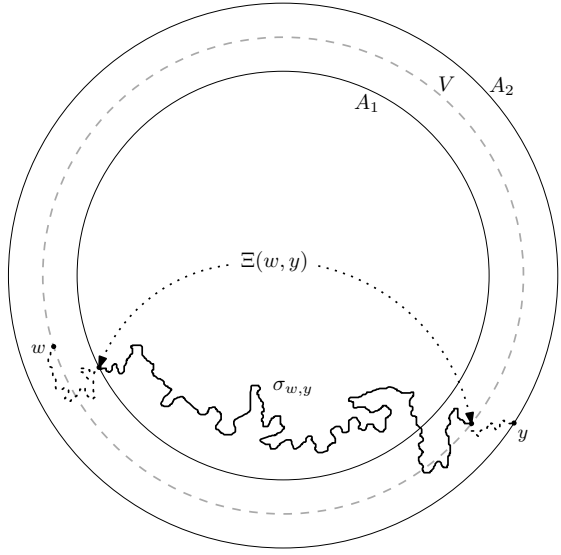}
\vspace{0.5cm}
\caption{The definition of $\sigma(w,y)$ and $\Xi(w,y)$.}
\label{processdeffig3}
\end{figure}
We introduce yet another artificial state $\Theta$ for reasons that will be made clear in a few moments. We let $\Sigma := \mathcal{K}\cup \{\Theta\} $ and let $\mu$ be a measure on $\Sigma$ defined in the following way: given $A\subseteq \Sigma$,
\begin{equation}
\mu(A):= \sum_{(x_0,\dots,x_n)\in A}\mathbb{P}_{(x_0,x_n)}[X_0=x_0,\dots,X_n=x_n]+1_{\{\Theta\in A\}},
\end{equation}
where $\mathbb{P}_{(x_0,x_n)}$ is the simple random walk measure 
conditioned on the event where $x_0$ is the walk's initial point and $x_n$ is its last point on $V$ before reaching $\partial A_2$. Notice that~$\mu(\{\Theta\})=1$.

Given $(w,y)\in V\times \partial A_2$ we let $\mathbb{P}_{w,y}$ be the measure associated with simple random walk starting at $w$ conditioned on the event where $y$ is the first point the walk hits in~$\partial A_2$, that is:
\begin{equation}
\mathbb{P}_{w,y}[\cdot]:=\mathbb{P}_{w}[\text{ }\space\cdot\mid X_{H_{\partial A_2}}=y]
\end{equation}
We want to randomly select (according to the conditional simple random walk measure above) a piece of trajectory in $A_1$ given a 
starting point in $V$ and an ending point in $\partial A_2$. Given $w\in V$ and $y\in\partial A_2$ we define the random element $\sigma_{w,y}\in\Sigma$ in the following way:
\begin{itemize}
\item Let $\mathcal{B}_{w,y}$ be a Bernoulli random variable with parameter $\mathbb{P}_{w,y}[H_{\partial A_1}<H_{\partial A_2}]$.
\item If $\mathcal{B}_{w,y}=0$ we let $\sigma_{w,y}\equiv\Theta$. 
\item If $\mathcal{B}_{w,y}=1$ we let, for $\mathfrak{A}\subseteq\mathcal{K}$:
\begin{equation}
\mathbb{P}[\sigma_{w,y}\in \mathfrak{A}]=\sum_{(a_0,\dots,a_n)\in \mathfrak{A}}\mathbb{P}_{w,y}\left[\begin{array}{c}X_{H_{A_1}}=a_0,X_{H_{A_1}+1}=a_1,\dots,X_{H_{A_1}+n}=a_n,\\X_k\notin A_1\text{ for every }k=H_{A_1}+n+1,\dots,H_{A_2}\end{array}\right].
\end{equation}
\end{itemize}
In other words, the random element $\sigma_{w,y}\in\Sigma$ will either be $\Theta$, on the event where a random walk starting at $w$ and exiting at $y$ fails to reach $A_1$, or a simple random walk trajectory~$(x_{0}^{w,y},x_{1}^{w,y},\dots , x_{k(w,y)}^{w,y})\in\mathcal{K}$ 
distributed so that $x_{0}^{w,y}$ is the first point in $A_1$ after the start at $w$ and $x_{k(w,y)}^{w,y}$ is the last point 
in $V$ before reaching $y\in\partial A_2$. We then define~$g_{(w,y)}:\Sigma\mapsto\R_+$ to be the $\mu$-density of $\sigma_{w,y}$. We refer to Figure~$\ref{processdeffig3}$.

Given $z=(x_0,\dots,x_n)\in \mathcal{K}$ we denote by $\Xi(z)$ the pair~$(x_0,x_n)$, the path's starting and ending points. We also let~$\Xi(\Theta)=\Theta$ so that $\Xi(z)$ is defined for all~$z\in \Sigma$. For $(w,y)\in V\times \partial A_2$ we define~$\Xi(w,y)$ to be the random element~$\Xi(\sigma_{w,y})$. 

Let us calculate~$g_{(w,y)}$ using the above notation. For $\mathfrak{A}\subseteq\Sigma$ we want to express the probability 
$\mathbb{P}[\sigma_{w,y}\in \mathfrak{A}]$ as a~$\mu$-integral over $\mathfrak{A}$.
\begin{equation}
\begin{array}{e}
 \mathbb{P}[\sigma_{w,y}\in \mathfrak{A}] & = & \sum_{a\in \mathfrak{A}}\mathbb{P}[\sigma_{w,y}=a]\\
  & = & 
  
  1_{\{\Theta\in \mathfrak{A}\}}\mathbb{P}_{w,y}[\Xi(w,y)=\Theta]\\ \\
  &   &+
  \sum_{\substack{a\in \mathfrak{A} \\ a\neq \Theta}} \mathbb{P}_{w,y}[\Xi(w,y)=\Xi(a)] \mathbb{P}_{w,y}[a\mid \Xi(w,y)=\Xi(a)]\\
  
  & = &
  
  1_{\{\Theta\in \mathfrak{A}\}}\mathbb{P}_{w,y}[\Xi(w,y)=\Theta]+\sum_{\substack{a\in \mathfrak{A} \\ a\neq \Theta}} \mathbb{P}_{w,y}[\Xi(w,y)=\Xi(a)] \mathbb{P}_{\Xi(a)}[a]\\

  & = & \sum_{a\in \mathfrak{A}}\mathbb{P}_{w,y}[\Xi(w,y)=\Xi(a)]\mu(a) \\
  
  & = &
  
    \int_\mathfrak{A} \mathbb{P}_{w,y}[\Xi(w,y)=\Xi(z)]\mu(\d z),
 \end{array}
\end{equation}
so that $g_{(w,y)}(z)= \mathbb{P}_{w,y}[\Xi(w,y)=\Xi(z)]$. 
Notice that the function $g_{(w,y)}(z)$ only depends on the pair $\Xi (z)$, the path's initial and ending points.

Let $(L,\mathcal{D},\mathbb{Q})$ be the measure space of the Poisson point process on $\Sigma\times\R_{+}$ with intensity measure $\mu\otimes\d v$, where $\d v$ is the Lebesgue measure on $\R_{+}$. 
A weighted sum of functions $g_{(\cdot,\cdot)}$ indexed by clothesline processes $((W_k,Y_k))_{k\geq 1}$ will be the soft local time used to simulate the pieces of trajectory we need. This way we will be able to simulate the intersection of a simple random walk trajectory with $A_1$. 
As we have seen in the random interlacements process's definition, to simulate the interlacements set inside $V$ we need a number $N^u_V\eqd Poisson(u\capacity (V))$ of independent random walks. We will need the same number of independent clothesline processes. 
For such task we will need a much bigger probability space, easily definable as a product between the Poisson point process space and an infinite product of independent simple random walk spaces starting on $V$. 
We call this bigger space the global probability space, and denote 
by $\GP$ its probability measure, which we will call the `global probability'.

Given a clothesline process $((W_k,Y_k))_{k\geq 1}$, we define the trajectory's soft local time:
\begin{equation}
\label{sltclothdef}
G(z)=\sum_{k=1}^{T_\Delta}\xi_k g_{(W_k,Y_k)}(z).
\end{equation}
We will also need to consider the soft local time up to a 
random time $T\leq T_\Delta$:
\begin{equation}
G_T(z)=\sum_{k=1}^{T}\xi_k g_{(W_k,Y_k)}(z).
\end{equation}
Analogously, we define for any deterministic time $n\geq 1$
\begin{equation}
G_n(z)=\sum_{k=1}^{n}\xi_k g_{(W_k,Y_k)}(z).
\end{equation}
We denote by $z_k$ the piece of trajectory randomly selected by the $k$-th soft local time,~$G_k$.

As we have seen before, in order to simulate the random interlacements set at level~$u$ in~$A_1$, we actually need a
\begin{equation*}
N^u_V\eqd\textit{Poisson}(u\capacity (V))
\end{equation*}
number of random walk trajectories, each started at a point in $V$ distributed as $\overline e_V (\cdot)$. For~$j=1,\dots,N^u_V$ we let~$((W_{k}^j,Y_{k}^j))_{k\geq 1}$ be a clothesline process started at $W_{1}^j$, so that $((W_{k}^j,Y_{k}^j))_{k\geq 1}$ 
is independent from $((W_{k}^i,Y_{k}^i))_{k\geq 1}$ for $i\neq j$, and so that $W_{1}^j$ is distributed as $\overline e_V (\cdot)$.
 Let $T_{\Delta}^j$ be the killing time associated with $((W_{k}^j,Y_{k}^j))_{k\geq 1}$. We denote by
\begin{equation}
\label{e_sltdef1}
G^j(z)=\sum_{k=1}^{T_{\Delta}^j}\xi_{k}^j g_{(W_{k}^j,Y_{k}^j)}(z)
\end{equation}
the soft local time associated with the $j$-th clothesline process. It should be clear from Proposition $\ref{p_decslt}$ that we can simulate all the random elements $(\sigma_{W_{k}^j,Y_{k}^j})_{j,k\geq 1}$ 
at the same time using only one realization of a Poisson point process in $\Sigma\times\R_+$. As the Corollary $\ref{c_couplez}$ shows, in order to control the values our random elements take we only need to control the function
\begin{equation}
\label{e_sltdef2}
G^{\Sigma}_{u}(z)=\sum_{j=1}^{N^u_V}G^j(z),
\end{equation}
the soft local time associated with the whole process. With such objective in mind we for now set our goals at estimating the 
soft local time's moments. We first show an easier way to express the expectation of $G(z)$.

\begin{proposition}
\label{t_expslt}
Using the same notation as above, we have
\begin{equation}
 \IE (G(z))=\IE\Big(\sum_{k=1}^{T_\Delta}1_{ \{\Xi(X_{D_{k-1}},X_{R_k})=\Xi(z)\}}\Big).
\end{equation}
\end{proposition}
\begin{proof}
In fact,
\begin{equation}
\begin{array} {e}
 \IE (G(z)) & = & \IE\Big(\sum_{k=1}^{T_\Delta}g_{(W_k,Y_k)}(z)\Big)  =  \IE\Big(\sum_{k=1}^{T_\Delta}\IP_{W_k,Y_k}[\Xi(W_k,Y_k)=\Xi(z)]\Big) \\
 & = & \IE\Big(\sum_{k=1}^{T_\Delta}1_{ \{\Xi(W_k,Y_k)=\Xi(z)\}}\Big) = \IE\Big(\sum_{k=1}^{T_\Delta}1_{ \{\Xi(X_{D_{k-1}},X_{R_k})=\Xi(z)\}}\Big).
\end{array}
\end{equation}
\end{proof}
We have then that the expectation of $G(z)$, for $z\neq\Theta$, is the same as the expectation of how many times a random walk started at $W_1$ will do a excursion on $A_2^C$ with starting and ending points given by $\Xi(z)$. 

It is clear that the same computation works for any starting distribution for $W_1$. Given~$y\in\partial A_2$, we let $\beta_y(\cdot)$ be the hitting measure on $V$ of a simple random walk started at~$y$. 
We are then able to take $\beta_y(\cdot)$ as the starting distribution of $W_1$. Let then~$\GP_{\beta_y}$ be the global process's measure in which the clothesline process's starting distribution is given by $\beta_y(\cdot)$, and let $\IE_{\beta_y}$ be its associated expectation. 
We are then required to allow the clothesline process to start at the cemetery state $\Delta$, denoting the failure of the random walk trajectory started at $y$ to reach $V$. 
In an analogous definition, we let $\GP_{w_0}$ be the global process's measure with $w_0\in V$ as the clothesline process's starting point, and let $\IE_{w_0}$ be its associated expectation. 

The next proposition, adapted from Theorem~$4.8$ of \cite{PopovTeixeira}, gives a bound on the second moment $\IE(G(z))^2$.

\begin{proposition}
 \label{t_2ndmoment}
 For any $w_0 \in V$,
 \begin{equation}
   \IE_{w_0} \big(G(z) \big)^2 
  \leq 2 \IE_{w_0} \big(G(z)\big)
  \big( \sup_{w'\in V} \IE_{w'} G(z) +\sup_{w,y} g_{(w,y)}(z) \big) .
 \end{equation}
\end{proposition}

\begin{proof}
 Recall that the second moment of a $\Exp(1)$ random variable equals $2$. For $z \in \Sigma$ and $n \geq 1$, we write
 \begin{align*}
   \IE_{w_0} \big( G_n(z) \big)^2 & 
= \IE_{w_0} \Big( \sum_{k=1}^n \xi_k g_{(W_k,Y_k)}(z) \Big)^2 \\
   & = \IE_{w_0} \Big( \sum_{k=1}^n \xi_k^2 g_{(W_k,Y_k)}^{2}(z) \Big) 
+ \IE_{w_0} \Big(2 \sum_{k < k' \leq n} 
\xi_k \xi_{k'} g_{(W_k,Y_k)}(z) g_{(W_{k'},Y_{k'})}(z) \Big)\\
   & \leq \sum_{k=1}^n \IE\xi_k^2 \sup_{w,y} g_{(w,y)}(z) \IE_{w_0}
 g_{(W_k,Y_k)}(z) + 2 \sum_{k=1}^{n-1} \sum_{k'=k+1}^{n}
 \IE_{w_0} \big(g_{(W_k,Y_k)}(z) g_{(W_{k'},Y_{k'})}(z) \big) \\
   & \leq 2 \sup_{w,y} g_{(w,y)}(z) \IE_{w_0} G_n(z) + 2
 \sum_{k=1}^{n-1} \sum_{k'=k+1}^{n} \IE_{w_0} \big( g_{(W_k,Y_k)}(z)
 \IE_{w_0} (g_{(W_k',Y_k')}(z)\mid W_k,Y_k ) \big)\\
   & \leq 2 \sup_{w,y} g_{(w,y)}(z) \IE_{w_0}
 G_n(z) + 2 \sum_{k=1}^{n-1} \IE_{w_0} 
\Big( g_{(W_k,Y_k)}(z) \IE_{\beta_{Y_{k}}} \Big( \sum_{m=1}^{n-k}
 g_{(W_m,Y_m)}(z) \Big) \Big)\\
& \leq 2 \sup_{w,y} g_{(w,y)}(z)  \IE_{w_0}
 G_n(z) + 2 \sup_{w'}\IE_{w'} \Big( \sum_{m=1}^{n-k}
 g_{(W_m,Y_m)}(z) \Big)
\IE_{w_0} \Big(\sum_{k=1}^{n-1} 
g_{(W_k,Y_k)}(z) \Big)\\
&\leq 2 \IE_{w_0} \big(G_n(z)\big)
\big( \sup_{w'} \IE_{w'} G_n(z) +\sup_{w,y} g_{(w,y)}(z) \big),
 \end{align*}
so that the result is proved for time~$n$. 
Letting~$n$ go to infinity, by the 
monotone convergence theorem we  can prove the result 
for the stopping time $T_\Delta$.
\end{proof}

For this paper's results, an estimate on the exponential moments of~$G$ will be essential. The next proposition, 
again adapted from \cite{PopovTeixeira} 
(propositions $\ref{t_expmoment}$ and $\ref{t_2ndmoment}$ are proved in the context of Markov chains in the original paper), gives us such an estimate. 

\begin{proposition}
 \label{t_expmoment}
 Given $\hat z \in \Sigma$ and measurable $\Gamma \subset \Sigma$, let
 \begin{equation}
  \begin{split}
  \alpha & = \inf\Big\{\frac{g_{(w,y)}(z')}{g_{(w,y)}({\hat z})}; 
  (w,y) \in V\times\partial A_2, z' \in \Gamma,\hat z \in \mathcal{K}\Big\},\\
  N(\Gamma) & = \#\{k \leq T_\Delta; z_k \in \Gamma\} ,
\text{ and}\\
  \ell & \geq \; \smash{\sup_{(w,y) \in V\times\partial A_2}} \; g_{(w,y)}(\hat z).
  \end{split}
 \end{equation} 
Then, for any $v \geq 2$,
\begin{equation}
\label{e_boundexpslt}
  \GP[G({\hat z}) \geq v \ell] 
\leq \GP[G({\hat z}) \geq \ell] 
\Big( \exp \big\{-\big(\tfrac v2 - 1 \big) \big\} + 
\sup_{w'}\GP_{w'} \big[ \eta( \Gamma 
\times [0, \tfrac{1}{2} v \ell\alpha] ) \leq N(\Gamma) \big]
%P[Y \leq N(T_\Delta)] 
\Big)
 \end{equation}
(note that $\eta( \Gamma 
\times [0, \tfrac{1}{2} v \ell\alpha] )$ 
is a random variable with distribution 
$\textnormal{Poisson} \big( \tfrac{1}{2}v \ell \alpha\mu(\Gamma) \big)$).
\end{proposition}

The number $\alpha=\alpha(\Gamma)$ above gives us a regularity condition: whenever $\alpha$ is uniformly larger than some constant $c>0$, we have that the density function $g_{(w,y)}(\cdot)$ when restricted to the subset $\Gamma$ cannot vary too much. 

We first explain the intuition behind the terms in the right-hand side of ~$(\ref{e_boundexpslt})$. The first term in the product is explained by the fact that in order for $G(\hat z)$ 
to get past $v \ell$, it must first overcome $\ell$. The first summand inside the parenthesis corresponds to the probability that the sum $G(\hat z)$ 
overcomes $\ell$ at the same ``time'' it overcomes $v \ell 2^{-1}$, that is, a overshooting probability. The second summand corresponds to a large deviation estimate, and generally, as $v$ grows, $N(\Gamma)$ becomes much smaller than the expected value of~$\eta(\Gamma \times [0, \tfrac{1}{2}v\ell\alpha])$. 

\begin{proof}
 We define the stopping time (with respect to the filtration $\mathcal{F}_n = \sigma((W_k,Y_k), \xi_k, k \leq n))$
 \begin{equation}
  T_\ell = \inf\{ k \geq 1; G_k({\hat z}) \geq \ell\}.
 \end{equation}
For $v \geq 2$, we have
\begin{align}
\lefteqn{\GP[G ({\hat z}) \geq v \ell]}\phantom{*************************************************}\nonumber \\
  \leq \GP \big[T_\ell < \infty, \; 
G_{T_\ell}({\hat z}) \geq \tfrac v2 \ell \big] + 
\GP \big[T_\ell <  \infty, \; G_{T_\ell}({\hat z}) < \tfrac v2 \ell, \; G({\hat z}) - G_{T_\ell}({\hat z}) > \tfrac v2 \ell \big] \label{e:Qtwoterms}
\end{align}
(note that $\GP[G({\hat z})\geq \ell]
=\GP[T_\ell<\infty]$).
 We first estimate the first term in the right side of the above inequality. By the memoryless property of the exponential distribution, we have
\begin{align}
 \label{e:Qfirstterm}
 \nonumber
  \smash{\sum_{n \geq 1}} \IE & \Big( G_{n-1}({\hat z}) < \ell,  
\GP \big[ \xi_n g_{(W_n,Y_n)}({\hat z})
 > \tfrac v2 \ell - G_{n-1} ({\hat z}) \mid W_{n-1}, Y_{n-1}, G_{n-1} \big] \Big)\\
 \nonumber
  & \leq \sum_{n \geq 1} \IE \Big( G_{n-1}({\hat z}) < \ell, 
\GP[ \xi_1 g_{(W_n,Y_n)}({\hat z}) > \ell - G_{n-1}] \, \GP 
\big[\xi_1 g_{(W_n,Y_n)}({\hat z}) > \big( \tfrac v2 -1 \big) \ell \big] \Big)\\
  & \leq \GP[T_{\ell} < \infty] 
\sup_{(w',y')} \GP\big[\xi_1 g_{(w',y')}({\hat z}) 
> \big( \tfrac v2 -1 \big) \ell \big]\\
 \nonumber
  & \leq \GP[T_{\ell} < \infty] 
\exp \big\{-\big(\tfrac v2 - 1 \big) \big\}.
\end{align} 
Now, to bound the second term in the right side of~\eqref{e:Qtwoterms}, we write
 \begin{equation}
 \label{e:Qsecterm}
  \begin{split}
   \IE & \big(T_\ell <  \infty, \; 
G_{T_\ell}(\hat z) < \tfrac v2 \ell, 
 \GP [G({\hat z}) - G_{T_\ell}({\hat z}) 
> \tfrac v2 \ell \mid G_1, \dots, G_{T_\ell}]\big)\\
   & \leq \GP \big[T_\ell <  \infty \big] \; \smash{\sup_{w'}}
 \; \GP_{w'} [G({\hat z}) > \tfrac v2 \ell ].
  \end{split}
 \end{equation}
Using that for any $z' \in \Sigma$
 \begin{equation}
  G(z') = \sum_{k=1}^{T_\Delta} \xi_k g_{(W_k,Y_k)}(z') \geq \sum_{k=1}^{T_\Delta} \alpha \xi_k g_{(W_k,Y_k)}({\hat z}) \1{\Gamma}(z') = \alpha G({\hat z}) \1{\Gamma}(z').
 \end{equation}
 we obtain, for all~$z'$,
 \begin{equation}
 \label{e:GbyPoisson}
  \begin{split}
   \GP\big[G({\hat z}) \geq \tfrac v2 \ell \big] 
& \leq \GP \Big[ G(z') \geq \frac{1}{2 }v \ell\alpha, \text{ for every $z' \in \Gamma$} \Big]\\
   & \leq \GP \big[ \eta( \Gamma 
\times [0, \tfrac{1}{2} v \ell\alpha] ) \leq N(\Gamma) \big].
  \end{split}
 \end{equation}
Collecting \eqref{e:Qtwoterms}, \eqref{e:Qfirstterm}, \eqref{e:Qsecterm} 
and \eqref{e:GbyPoisson} we finish the proof of the result.
\end{proof}

\section{Conditional decoupling}
\label{s_cd}
We begin this section gathering some facts needed for the proof of the main theorem of this paper. But first we give an overview of main argument presented in this section. 
We will simulate the random interlacements set intersected with $A_1$ in two ways. In the first way we will simulate $\I_{A_1}^u$ using $G^{\Sigma}_u$, that is, we will simulate $\I_{A_1}^u$ using the soft local times indexed by the clothesline processes.
 In the second way, we will construct a set made up from random walk trajectories in $A_1$ in a similar way to the construction  of~$\I^u_{A_1}$, the only difference will be that the soft local times used in this second construction will be indexed by a given nonrandom sequence $\hat{\zeta}$ of pairs of points belonging $V\times \partial A_2$. 
We will denote this second random set by~$\I^u_{A_1\mid\hat{\zeta}}$, and we will show using the soft local times method that~$\I^u_{A_1\mid\hat{\zeta}}$ and~$\I^u_{A_1}$ are usually very similar to each other. We then prove a similar result when 
the pairs of points that constitute the nonrandom sequence all belong to the boundary of a set contained in~$A_2$. 

Throughout this section we will again only differentiate between $A_{1}^{\tiny\Circle}$ and $A_{1}^{\tiny\Square}$ when the need arises. We start by stating the following bound
\begin{equation}\label{supprob2}
\sup_{\substack{w'\in V \\ y'\in \partial A_2}}\IP_{w',y'}\big{[} \Xi(w',y')=(w_0,y_0) \big{]}\leq c s ^{-2(d-1)} ,\end{equation}
for which the proof is technical and we thus postpone it to subsection $\ref{s_t1}$ of the appendix. 

Let $z\in\Sigma$ be such that~$\Xi(z)=(w_0,y_0)$, and let $h:=\dist(w_0,y_0)$. We let $F(w_0,y_0)$ stand for $G(z)$, making explicit the dependence of the soft local time on the endvertices~$\Xi(z)$. We define
\begin{equation}
\label{e_expecdef}
\pi(w_0,y_0):=\IE(F(w_0,y_0)).
\end{equation}
We define $f_{A_1}(w_0,y_0)$ to be the probability that the simple random walk started at $w_0$ visits $y_0$ before hitting $A_2$. We will prove in the appendix  (see Section~$\ref{s_t1}$, propositions~$\ref{p_boundprocircle}$ and~$\ref{p_boundprosquare}$) the following bounds for these probabilities:
\begin{itemize}
\item[(i)]Given $(w_0,y_0)\in A_1^{\tiny\Circle}\times V^{\tiny\Circle}$, there are constants $c_1,c_2>0$ such that
\begin{equation}
\label{e_circlebound}
c_1\frac{s^{2}}{h^{d}} \leq f_{A_1^{\tiny\Circle}}(w_0,y_0) \leq c_2\frac{s^{2}}{h^{d}}.
\end{equation} 
\item[(ii)]Let $(w_0,y_0)\in A_1^{\tiny\Square}\times V^{\tiny\Square}$, and recall the definition of $\mathfrak{H}_{r+2s}$, the unsmoothed version of $ {A_2^{\tiny\Square}}^C$. Let $\mathfrak{H}^{d-1}_i$; $i=1,\dots,2d$; denote the $(d-1)$-dimensional hyperfaces of~$\mathfrak{H}_{r+2s}$, and let $l^{w_0}_i:=\min \{\dist(w_0,\mathfrak{H}^{d-1}_i),h\}$, and $l^{y_0}_i:=\min \{\dist(y_0,\mathfrak{H}^{d-1}_i),h\}$. Then there are constants $c_1,c_2>0$ such that
\begin{equation}
\label{e_squarebound}
c_1\frac{l^{w_0}_1\dots l^{w_0}_{2d}}{h^{2d}}\cdot\frac{1}{h^{d-2}}\cdot\frac{l^{y_0}_1\dots l^{y_0}_{2d}}{h^{2d}}\leq f_{A_1^{\tiny\Square}}(w_0,y_0) \leq c_2\frac{l^{w_0}_1\dots l^{w_0}_{2d}}{h^{2d}}\cdot\frac{1}{h^{d-2}}\cdot\frac{l^{y_0}_1\dots l^{y_0}_{2d}}{h^{2d}}.
\end{equation}
\end{itemize}

The following lemma, whose proof we also postpone to the appendix (Section~$\ref{s_t2}$), gives us an estimate on $\pi(w_0,y_0)$.

\begin{lemma}
\label{l_expecslt}
Using the notation defined above we have, for constants $c_1,c_2,c_3,c_4>0$:\\
\begin{itemize}
 \item[(i)]$c_1\capacity(V)^{-1}s^{-1}f_{A_1}(w_0,y_0)\leq\pi_{A_1}(w_0,y_0)\leq c_2\capacity(V)^{-1}s^{-1} f_{A_1}(w_0,y_0)$,
 \\
 \item[(ii)]$\IE(F(w_0,y_0)^2)\leq c_3 \capacity(V)^{-1}s^{-2d+2}f_{A_1}(w_0,y_0)$.
  \\
  \\
   Moreover, since $\dist (w_0,y_0)\geq s$, we have
   \\
 \item[(iii)]$\sup_{w_0,y_0}\pi(w_0,y_0)\leq c_4\capacity(V)^{-1}s^{-(d-1)}$.
 \end{itemize}
\end{lemma}

We now provide a large deviation bound for $F(w_0,y_0)$.

\begin{lemma}
\label{l_expmoment}
There are constants $c,c_1,c_2>0$ such that for every $(w_0,y_0)\in V\times \partial A_2$, we have
\begin{equation}
\GP\big{[}  F(w_0,y_0)>v c s^{-2(d-1)}  \big{]}\leq c_1 s^{2d-3}f_{A_1}(w_0,y_0)\capacity(V)^{-1}e^{-c_2 v}
\end{equation}
for any $v\geq 2$ (we can also assume $c_2\leq 1$ without loss of generality).
\end{lemma}

\begin{proof}
In the proof of this particular result it will be important for us to distinguish between the constants. We will use Proposition $\ref{t_expmoment}$ for $F(w_0,y_0)$, with
\begin{equation*}
\Gamma_{w_0,y_0} :=\{(w_0 ',y_0 ')\in \partial A_1\times V ;\space\space \max \{\|w_0 ' - w_0\|,\|y_0 ' - y_0\|\} \leq c_4 s \},
\end{equation*} 
with $0<c_4<1$ defined in Section~$\ref{s_t3}$ of the appendix.

Using the same notation as in Proposition $\ref{t_expmoment}$, we note that $(\ref{supprob2})$ implies
\begin{equation*}
l\leq c s^{-2(d-1)}
\end{equation*}
and observe that $\mu(\Gamma_{w_0,y_0})\geq c_5 s^{2(d-1)}$ for some constant $c_5>0$. Also, as can be seen in Section~$\ref{s_t3}$ of the appendix, we have
\begin{equation*}
\alpha \geq c_3>0.
\end{equation*}
Chebyshev's inequality and Lemma~$\ref{l_expecslt}$ then imply
\begin{equation}
\label{e_timebound}
\GP\big{[}     T_l < \infty  \big{]}\leq\GP\big{[}  F(w_0,y_0)>c s^{-2(d-1)}  \big{]}\leq\frac{\pi (w_0,y_0)}{c s^{-2(d-1)}}
\leq
c_1 s^{2d-3}f_{A_1}(w_0,y_0)\capacity (V)^{-1}.
\end{equation}

We denote by $N(\Gamma_{w_0,y_0})$ the number of times the simple random walk trajectory associated with $F(w_0,y_0)$ makes an excursion of the form $z'\in\Sigma$ on $A_2^C$ such that
 $\Xi(z')=(w',y')\in\Gamma_{w_0,y_0}$. We also let $\eta_{w_0,y_0}$ stand for the number of points of the Poisson process associated with our soft local times that belong to $\Gamma_{w_0,y_0}\times\big[0,\frac{1}{2}v c c_3 s^{-2(d-1)}\big]$. We note that both definitions are consistent with Proposition~$\ref{t_expmoment}$ and write
\begin{equation*}
\GP\Big{[}  \eta_{w_0,y_0}\leq   N(\Gamma_{w_0,y_0})  \Big{]}\leq \GP\Big{[}  \eta_{w_0,y_0}\leq   \frac{c c_3 c_5 v}{4}  \Big{]} + 
\GP\Big{[}    N(\Gamma_{w_0,y_0}) \geq  \frac{c c_3 c_5 v}{4}   \Big{]}.
\end{equation*}
We claim that both terms in the right side of the above inequality are exponentially small in $v$. To see why this is true, observe that:
\begin{itemize}
\item   $\eta_{w_0,y_0}$ has Poisson distribution with  parameter at least $\frac{cc _3 c_5 v}{2}$, and
\item   every time the simple random walk
 associated with $F(w_0,y_0)$ hits $\partial A_2$, with uniform positive probability the walk never reaches $\Gamma_{w_0,y_0}$ again. This way $N(\Gamma_{w_0,y_0})$ is dominated by a 
Geometric$(c_6)$ random variable, for some constant $c_6<1$.
\end{itemize}
Together with $(\ref{e_timebound})$ and Proposition~$\ref{t_expmoment}$, this finishes the proof of the lemma.
\end{proof}

Let $\Psi_{w_0,y_0}(\lambda)=\IE (e^{\lambda F(w_0,y_0)})$ be the moment generating function of $F(w_0,y_0)$. We are going to use the bounds above to estimate $\Psi_{w_0,y_0}$. 
It is elementary to obtain that $e^{t}-1\leq t+t^2$ for $t\in [0,1]$. 
With this observation in mind, we write for $0\leq \lambda\leq \frac{c_2 s^{2(d-1)}}{2c}$, where $c$ and $c_2$ are the same as in the theorem above:
\begin{align}
\lefteqn{\Psi_{w_0,y_0}(\lambda)-1   =\phantom{******}}\nonumber \\ \nonumber 
&\phantom{****} = \IE (e^{\lambda F(w_0,y_0)}-1)\1{\lambda F(w_0,y_0)\leq 1} + \IE (e^{\lambda F(w_0,y_0)}-1)\1{\lambda F(w_0,y_0)> 1} \\ \nonumber 
&\phantom{****} \leq  \IE(\lambda F(w_0,y_0)+\lambda^2 F(w_0,y_0)^2)+\IE (e^{\lambda F(w_0,y_0)}-1)\1{\lambda F(w_0,y_0)> 1} \\ \nonumber 
&\phantom{****} \leq  \lambda\pi(w_0,y_0)+c_1\lambda^2\capacity(V)^{-1}s^{-2d+2}f_{A_1}(w_0,y_0)+\IE (e^{\lambda F(w_0,y_0)}-1)\1{\lambda F(w_0,y_0)> 1} \\ \nonumber 
&\phantom{****} \leq  \lambda\pi(w_0,y_0)+c'\lambda^2\capacity(V)^{-1}s^{-2d+2}f_{A_1}(w_0,y_0) +\lambda\int\limits_{\lambda^{-1}}^{\infty}e^{\lambda y}\GP \big{[}
F(w_0,y_0)>y  \big{]}\d y \\ \nonumber 
&\phantom{****} \leq  \lambda\pi(w_0,y_0)+f_{A_1}(w_0,y_0)\capacity(V)^{-1}\Big(c'\lambda^2s^{-2d+2} +\lambda c's^{2d-3}\int\limits_{\lambda^{-1}}^{\infty} \exp{\Big(\frac{-c_2 s^{2(d-1)}y}{2c}\Big)}\d y \Big) \\ \nonumber 
&\phantom{****} \leq  \lambda\pi(w_0,y_0)+f_{A_1}(w_0,y_0)\capacity(V)^{-1}\Big(c'\lambda^2 s^{-2d+2} + c'\lambda s^{-1}\exp{\Big(\frac{-c_2 s^{2(d-1)}\lambda^{-1}}{2c}\Big)} \Big)\nonumber \\  \label{e_moment1}
&\phantom{****}\leq\lambda\pi(w_0,y_0)+c'\lambda^2
\capacity(V)^{-1}s^{-2d+2}f_{A_1}(w_0,y_0),
\end{align}
where we used Lemma~$\ref{l_expecslt}$ and Lemma~$\ref{l_expmoment}$. Now since $e^{-t}-1\leq -t+t^2$ for all $t\geq 0$,
 we obtain for $\lambda\geq 0$
\begin{equation}
\label{e_moment2}
\Psi_{w_0,y_0}(-\lambda)-1\leq  -\lambda\pi(w_0,y_0)+c\lambda^2\capacity(V)^{-1}s^{-2d+2}f_{A_1}(w_0,y_0),
\end{equation}
(the large deviation bound of Lemma~$\ref{l_expmoment}$ 
is not necessary is this case).

Observe that if $(\chi_k,k\geq 1)$ are i.i.d.\ random variables with common moment generating function~$\Psi$ and $N$ is an 
independent Poisson random variable with parameter $\theta$, then
\begin{equation*}
\IE\exp\big(\lambda\sum\limits_{k=1}^{N}\chi_k \big)=e^{(\theta(\Psi(\lambda)-1))}.
\end{equation*}
We let $F_k(w_0,y_0)$ denote the expectation $\IE(G^k(z))$ defined in $(\ref{e_sltdef1})$, when $z\in\Sigma$ is such that $\Xi(z)=(w_0,y_0)$. Using Lemma~$\ref{l_expecslt}$ and $(\ref{e_moment1})$, we have, for $N_{\hat u}^{V}\eqd Poisson(\hat u\capacity(V))$ and any $\delta>0$
\begin{equation}
\label{e_moment3}
\begin{array}{e}
\lefteqn{\GP \big{[} G^{\Sigma}_{\hat u}(z)\geq (1+\delta)\hat u \capacity(V)\pi (w_0,y_0) \big{]}=\phantom{********}}\\
&\phantom{********} = & \GP \Big{[}  \sum_{k=1}^{N_{\hat u}^{V}} F_k (w_0,y_0)\geq (1+\delta)\hat u \capacity(V)\pi (w_0,y_0) \Big{]} \\ \\
&\phantom{********} \leq & \frac{\IE(\exp \big(\lambda \sum_{k=1}^{N_{\hat u}^{V}} F_k (w_0,y_0)\big))}{\exp\big(\lambda (1+\delta)\hat u \capacity(V)\pi (w_0,y_0)\big)} \\ \\
&\phantom{********} \leq &  \exp\big(-\lambda (1+\delta)\hat u \capacity(V)\pi (w_0,y_0)+\hat u\capacity (V)(\Psi_{w_0,y_0}(\lambda)-1)\big) \\ \\
&\phantom{********} \leq & \exp\big(-(\lambda \delta\hat u \capacity(V)\pi (w_0,y_0)-c'\lambda^2\hat u s^{-2d+2}f_{A_1}(w_0,y_0))\big) \\ \\
&\phantom{********} \leq & \exp\big(-(\lambda \delta\hat u c s^{-1} f_{A_1}(w_0,y_0)-c'\lambda^2\hat u s^{-2d+2}f_{A_1}(w_0,y_0))\big) .
\end{array}
\end{equation}
Analogously, with $(\ref{e_moment2})$ instead of $(\ref{e_moment1})$, we obtain
\begin{equation}
\label{e_moment4}
\GP \big{[} G^{\Sigma}_{\hat u}(z)\leq (1-\delta)\hat u \capacity(V)\pi (w_0,y_0) \big{]} \leq  \exp\big(-(\lambda \delta\hat u c s^{-1}-c'\lambda^2\hat u s^{-2d+2})f_{A_1}(w_0,y_0)\big) .
\end{equation}
We choose $\lambda=c_7\delta s^{2d-3}$ with $c_7$ small enough so that $\lambda\leq\frac{c_2 s^{2(d-1)}}{2c}$, and observe that the bounds for $f_{A_1}(w_0,y_0)$ given in $(\ref{e_circlebound})$ and $(\ref{e_squarebound})$ imply
\begin{align*}
\inf_{w_0,y_0}f_{A_1^{\tiny\Square}}(w_0,y_0)& \geq  c s^{2d} r^{-3d+2},\\
\inf_{w_0,y_0}f_{A_1^{\tiny\Circle}}(w_0,y_0)& \geq  c s^{2}r^{-d}.
\end{align*}
Recall the definition of $b_{A_1^{\tiny\Circle}}$, a number such that 
\begin{equation*}
1\leq b_{A_1^{\tiny\Circle}}<\frac{2d-2}{d},
\end{equation*}
and the definition of $b_{A_1^{\tiny\Square}}$, a number such that
\begin{equation*}
1\leq b_{A_1^{\tiny\Square}}<\frac{4d-4}{3d-2}.
\end{equation*}
Recall that $r\asymp s^{b_{A_1}}$. Then there exist constants $a_{A_1^{\tiny\Circle}}=2d-2-db_{A_1^{\tiny\Circle}}>0$ and $a_{A_1^{\tiny\Square}}=4d-4-3db_{A_1^{\tiny\Square}}+2b_{A_1^{\tiny\Square}}>0$ such that 
\begin{equation*}
 \GP \big{[} G^{\Sigma}_{\hat u}(z)\geq (1+\delta)\hat u \capacity(V)\pi (w_0,y_0) \big{]}\leq  \exp\big(-c\delta^2 \hat{u} s^{a_{A_1}}\big) .
\end{equation*}
Using the union bound (note that $ \partial A_1\times V$ has $O(r^{2(d-1)})$ elements),
\begin{align}
\label{e_moment5}
\lefteqn{\GP \big{[} (1-\delta)\hat u \capacity(V)\pi (\Xi(z)) \leq G^{\Sigma}_{\hat u}(z)\leq (1+\delta)\hat u \capacity(V)\pi (\Xi(z))\text{, for all $z\in\mathcal{K}$}\big{]}\geq}
\phantom{********************************}\nonumber\\
  &\geq 1-cr^{2(d-1)}\exp\big(-c'\delta^2\hat u s^{a_{A_1}}\big).
\end{align}
Observe that we can suppose $c'\leq 1$ 
without loss of generality. We define the interval 
\begin{equation*}
I_{\hat u,z}^{\delta}:=[(1-\delta)\hat u \capacity(V)\pi (\Xi(z)), (1+\delta)\hat u \capacity(V)\pi (\Xi(z))]
\end{equation*}
and the event
\begin{equation*}
D_{\hat u}^{\delta}:=\{ G_{\hat u}^{\Sigma}\in I_{\hat u,z}^{\delta}\text{ for all $z\in\mathcal{K}$} \}.
\end{equation*}
Using $(\ref{e_moment5})$ and the union bound we obtain, for $\eps>0$ sufficiently small,
\begin{equation}
\nonumber
\GP \Big{[} D_{u}^{\eps/4},D_{u(1-\eps)}^{\eps/4} ,D_{u(1+\eps)}^{\eps/4} \Big{]}
\geq 1-cr^{2(d-1)}\exp\big(-c'\eps^2 u s^{a}\big).
\end{equation}
Since $r\asymp s^{b_{A_1}}$, by replacing the constants $c$ and $c'$ 
in the above equation we obtain 
\begin{equation}
\label{e_moment6}
\GP \Big{[} D_{u}^{\eps/4},D_{u(1-\eps)}^{\eps/4} ,D_{u(1+\eps)}^{\eps/4} \Big{]}
\geq 1-c\exp\big(-c'\eps^2 u s^{a}\big).
\end{equation}
We have just proved that with high probability, the soft local time associated to each of the processes~$\I^u_{A_1}$, $\I^{u(1-\eps)}_{A_1}$ and $\I^{u(1+\eps)}_{A_1}$ stays confined between the graphs of two explicit deterministic functions. This happened when we let the ``information'' given by $\I^{u}_{A_2}$; namely the points of entrance at~$V$ and exit at~$\partial A_2$ of the excursions on~$A_1$ of the simple random walk trajectories of the interlacements process at level~$u$; to be distributed according to the right law, that is, the law of the clothesline processes. When we ``average'' those points according to these laws we obtain a good concentration for the whole function~$G_{u}^{\Sigma}$, but our goal is to obtain a similar concentration when these points are deterministic. The heuristic argument is that when something happens with high probability in the annealed law, then most of the times it will also happen with high probability in the quenched law. We will introduce some new notation to make this argument rigorous and prove our main theorem.

Given any two finite sets $K_1,K_2\subset\Z^d$, not necessarily disjoint, we want to describe a collection of generalized clothesline processes between~$K_1$ and~$K_2$ associated with the interlacements process at level~$u$. We construct an infinite family $(X^{(j)}_k,k\geq 0)_{0<j<\infty}$ of independent simple random walks with starting point distributed according to the normalized harmonic measure on~$K_1$, as we did in definition~\eqref{interlacementsdef}. We let $\tau^j_0\equiv 0$ and define inductively
\begin{equation*}
\begin{array}{e}
\tau_{k+1}^j &:= 1\{X^{(j)}_{\tau_{k}^j}\in K_1\}\inf\{t>\tau_{k+1}^j; X^{(j)}_t\in K_2\} 
\\
&\qquad+ 
 1\{X^{(j)}_{\tau_{k}^j}\in K_2\} \inf\{t>\tau_{k+1}^j; X^{(j)}_t\in K_1\},
\end{array}
\end{equation*}
where $1\{\cdot\}$ denotes the indicator function of an event. We also define the random time
\begin{equation*}
T_j:=\inf_{k\geq 0}\{\tau_{k+1}^j=\infty\}.
\end{equation*}
We let yet again $N_u^{K_1}\eqd\textit{Poisson}(u\capacity (K_1))$ be a random variable independent from\[(X^{(j)}_k,k\geq 0)_{0<j<\infty}.\] We then define the interlacements' clothesline processes between~$K_1$ and~$K_2$ at level~$u$ by
\begin{equation*}
\mathrm{Cloth}_u(K_1,K_2):=\Big\{\big(X^{(j)}_{\tau^j_k}\big)_{k=0}^{T_j}   \Big\}_{j=1}^{N_u^{K_1}}.
\end{equation*}
When  $K_1=V$ and $K_2=\partial A_2$, we have
\begin{equation*}
\mathrm{Cloth}_u(V,\partial A_2)\eqd\Big\{\big( W^j_k,Y^j_k   \big)_{k=1}^{T^j_\Delta}   \Big\}_{j=1}^{N_u^{V}}.
\end{equation*}
We define
\begin{equation*}
\Big( \mathcal{S}_u(K_1,K_2),\sigma_u(K_1,K_2),\IP^u_{K_1,K_2}     \Big)
\end{equation*}
to be the probability space in which $\mathrm{Cloth}_u(K_1,K_2)$ is defined, and in which $\sigma_u(K_1,K_2)$ is the smallest $\sigma$-field in which $\mathrm{Cloth}_u(K_1,K_2)$ is measurable. If $\hat\zeta\in\mathcal{S}_u(V,\partial A_2)$ and $\IP^u_{V,\partial A_2}(\hat\zeta)>0$, then we can write~$\hat\zeta$ as a finite collection of finite sequences of points belonging to~$V$ and~$\partial A_2$
\begin{equation*}
\hat{\zeta}:=\big\{  \hat{\zeta}_1,\dots, \hat{\zeta}_K   \big\},
\end{equation*}
where for each $j=1,\dots,K$; $\hat{\zeta}_j$ is a finite sequence alternating between points of~$V$ and~$\partial A_2$. In other words, $\hat{\zeta}_j$ is a possible realization of a clothesline process. We write
\begin{equation*}
\hat{\zeta}_j:=\big(  \zeta^j_0,\dots,\zeta^j_{n(j)} \big),
\end{equation*} 
where~$n(j)$ is odd, every even entry belongs to~$V$ and every odd entry belongs to~$\partial A_2$.
\begin{figure}
\centering
\includegraphics[scale = 1]{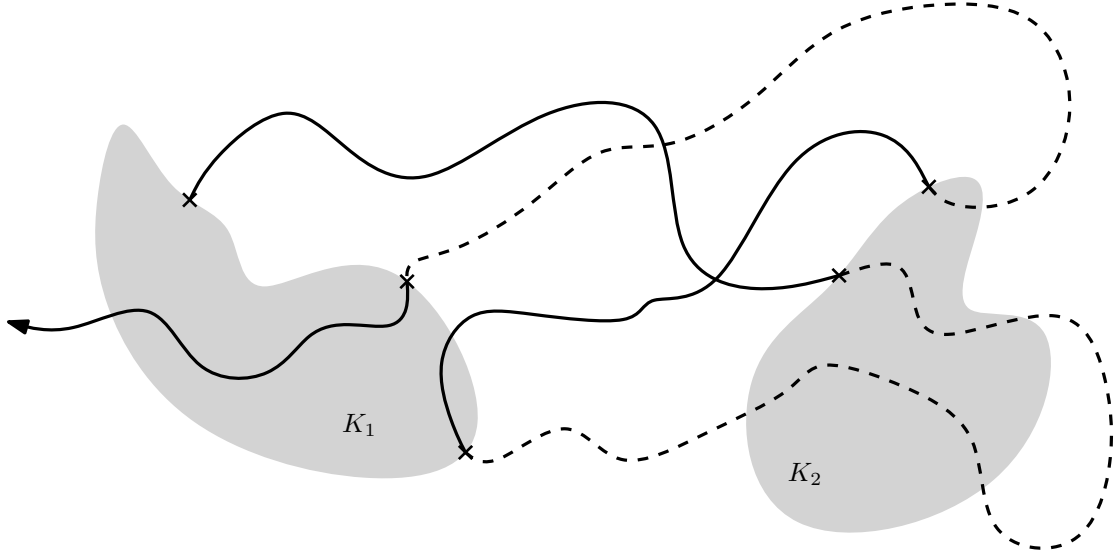}
\vspace{0.5cm}
\caption{The generalized clothesline process between~$K_1$ and~$K_2$, here represented by the X marks.}
\label{genclothdef}
\end{figure}
We then define the soft local time associated with~$\hat{\zeta}$. Using the same realization of the Poisson point process on~$\Sigma\times\R_+$ defined on Section~$\ref{s_simexc}$, we construct the soft local times
\begin{equation*}
G^{\hat\zeta_j}(z):=\sum_{k=0}^{\frac{n(j)+1}{2}}\tilde{\xi}_{k}^{j} g_{(\zeta^j_{2k},\zeta_{2k_1}^j)}(z),
\end{equation*}
where $\tilde{\xi}^j_k$ is an exponential random variable defined in the manner of~\eqref{sltclothdef}. We then define
\begin{equation*}
G^{\hat{\zeta}}(z):=\sum_{j=1}^{K}G^{\hat\zeta_j}(z).
\end{equation*}
This function should be viewed as a quenched version of the soft local times $G^\Sigma_u$, when the collection o clothesline processes $\{( W^j_k,Y^j_k   )_{k=1}^{T^j_\Delta}  \}_{j=1}^{N_u^{V}}$ is given by yhe deterministic element~$\hat{\zeta}$. We denote by $\I^u_{A_1\mid\hat{\zeta}}$ the interlacements process inside~$A_1$ determined by the ranges  of the excursions of~$\Sigma$ bellow~$G^{\hat{\zeta}}$. $\I^u_{A_1\mid\hat{\zeta}}$ is distributed as the random interlacements process inside~$A_1$ when its associated random walks excursions have entrance points at~$V$ and exit points at~$\partial A_2$ given by~$\hat{\zeta}$. The next proposition implies that~$G^{\hat{\zeta}}$ is usually between~$G^\Sigma_{u(1-\eps)}$ and~$G^\Sigma_{u(1+\eps)}$ with high probability.
 
\begin{proposition}
\label{t_varpi}
There exists a set $\mathcal{A}\in\sigma_u(V,\partial A_2)$ such that 
\begin{equation*}
\IP^u_{V,\partial A_2}\big{[} \mathcal{A} \big{]}\geq 1-\exp\Big(-\frac{c'}{2}\eps^2 u s^{a_{A_1}}\Big),
\end{equation*}
and for all fixed $\hat{\zeta}\in\mathcal{A}$,
\begin{align*}
\lefteqn{\GP \big{[} G_{u(1-\eps)}^{\Sigma}(z)\leq G^{\hat{\zeta}}(z)\leq G_{u(1+\eps)}^{\Sigma}(z)\text{ for all $z\in\mathcal{K}$} \big{]}}\phantom{*********************}\nonumber\\  &\phantom{*****}\geq 1-c \exp\Big(-\frac{c'}{2}\eps^2 u s^{a_{A_1}}\Big).
\end{align*}
\begin{proof}
Observe that $(\ref{e_moment6})$ implies
\begin{align}
\lefteqn{\int\GP \big{[} G_{u(1-\eps)}^{\Sigma}(z)\leq G^{\hat{\zeta}}(z)\leq G_{u(1+\eps)}^{\Sigma}(z)\text{ for all $z\in\mathcal{K}$} \big{]}\IP^u_{V,\partial A_2}\big{[}\d \hat{\zeta}\big{]}}\phantom{************************}\nonumber\\  &\phantom{********}\geq 1-c \exp\big(-c'\eps^2 u s^{a_{A_1}}\big)\label{e_moment7}.
\end{align}
Let
\begin{equation*}
\begin{array}{e}\lefteqn{\mathcal{A}:=\Big\{\hat{\zeta}\in \mathcal{S}_u(V,\partial A_2)\text{ such that: }\GP\big{[}  G_{u(1-\eps)}^{\Sigma}(z)\leq G^{\hat{\zeta}}(z)\leq G_{u(1+\eps)}^{\Sigma}(z)\text{ for all $z\in\mathcal{K}$} \big{]} }\phantom{********************}\\ &\phantom{*************}\geq 1-c \exp\Big(-\frac{c'}{2}\eps^2 u s^{a_{A_1}}\Big) \Big\}.\end{array}
\end{equation*}
\begin{figure}[ht]
\centering
\includegraphics[scale = .7]{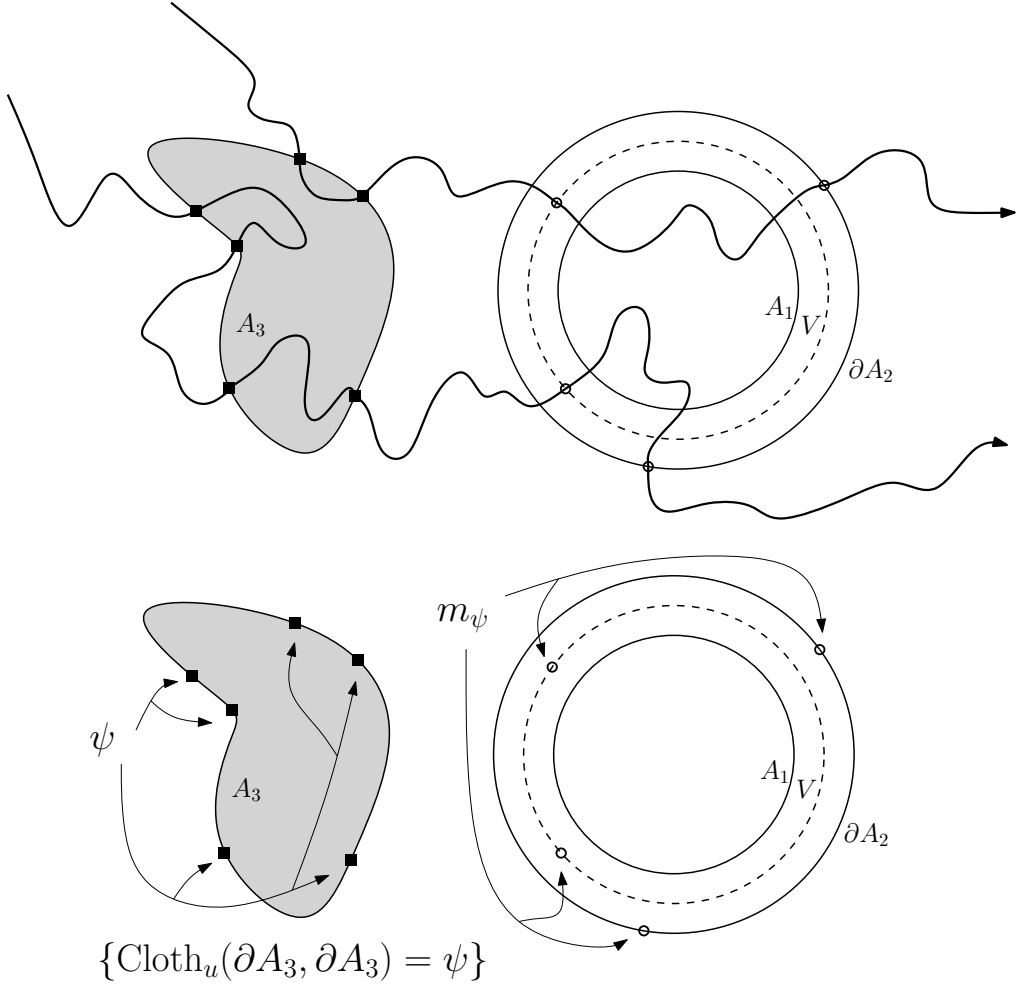}
\vspace{0.5cm}
\caption{A visual representation of the random element $m_{\hat{\psi}}$.}
\label{mpsidef}
\end{figure}
Then $(\ref{e_moment7})$ implies 
\begin{align*}
\lefteqn{\IP^u_{V,\partial A_2}\big{[}\mathcal{A}\big{]}+\Big(1-c \exp\Big(-\frac{c'}{2}\eps^2 u s^{a_{A_1}}\Big)\Big)\Big(1-\IP^u_{V,\partial A_2}\big{[}\mathcal{A}\big{]}\Big)}\phantom{*****************************}\\
&\geq 1-c \exp\big(-c'\eps^2 u s^{a_{A_1}}\big),
\end{align*}
so that
\begin{equation*}
\IP^u_{V,\partial A_2}\big{[}\mathcal{A}\big{]}\geq  1-\exp\Big(-\frac{c'}{2}\eps^2 u s^{a_{A_1}}\Big).
\end{equation*}
This finishes the proof of the proposition.
\end{proof}
\end{proposition}
\begin{figure}[ht]
\centering
\includegraphics[scale = 1]{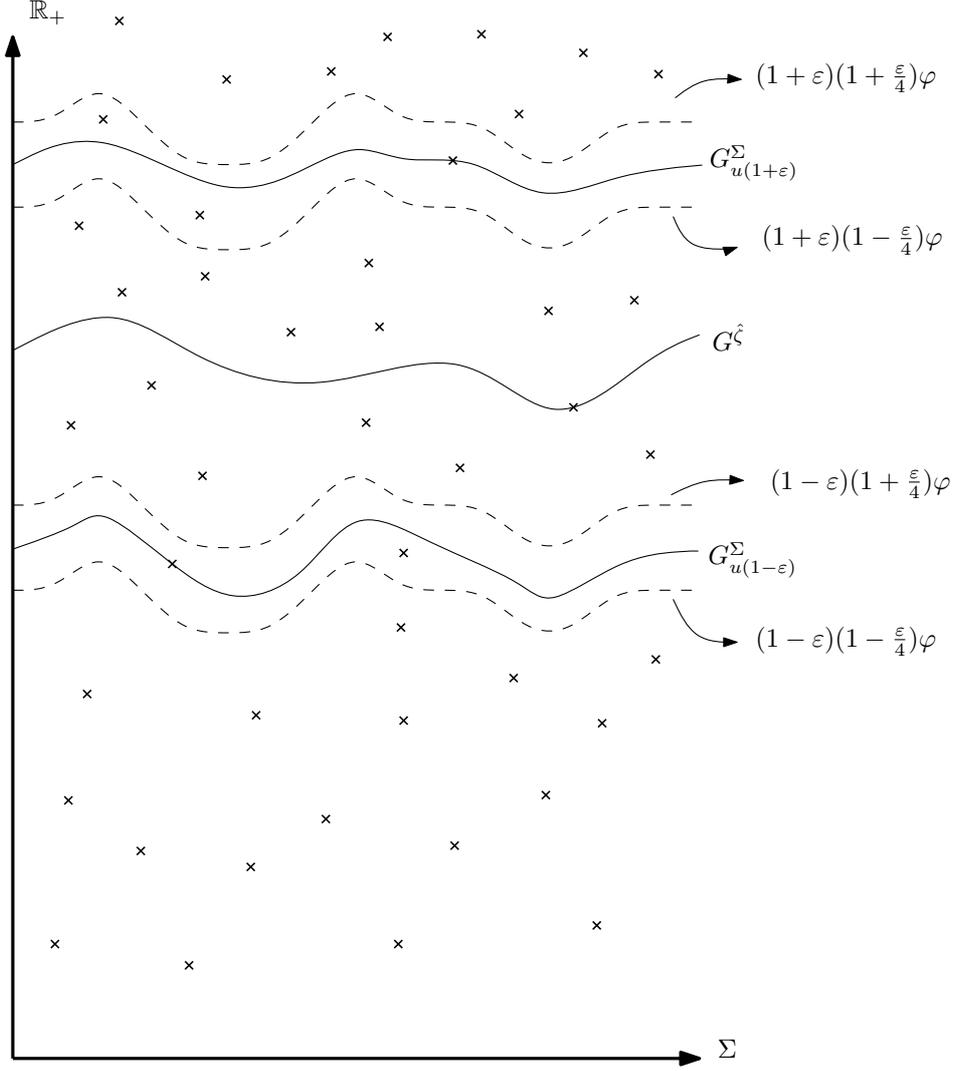}
\vspace{0.5cm}
\caption{When the sequence $\hat{\zeta}$ belongs to a well behaved set $\mathcal{A}$, the decoupling probability is greater than $1-c \exp\Big(-\frac{c'}{2}\eps^2 u s^{a_{A_1}}\Big)$. 
The symbol~$\varphi$ in the figure stands for the function $u\capacity(V)\pi(\Xi(z))$.
 The figure shows the decoupling event, where $G_{u(1-\eps)}^{\Sigma}(z)\leq G^{\hat{\zeta}}(z)\leq G_{u(1+\eps)}^{\Sigma}(z)$ for all~$z\in\mathcal{K}$.  }
\label{softlocaltimesfig2}
\end{figure}
	Proposition $\ref{t_varpi}$ implies that, for
$\hat{\zeta}\in\mathcal{A}$, there exists a process~($\hat{\I}^u_{A_1}$, $u\geq 0$) distributed as the random interlacements 
set intersected with $A_1$, and a coupling $\GP$ such that, for all $\eps>0$ sufficiently small and $r>0$ sufficiently big, we have
\begin{equation}
\label{e_varpi}
\GP\big{[} \hat{\I}^{u(1-\eps)}_{A_1} \subseteq  \I^u_{A_1\mid\hat{\zeta}}   \subseteq    \hat{\I}^{u(1+\eps)}_{A_1}   \big{]} \geq 1-c \exp\Big(-\frac{c'}{2}\eps^2 u s^{a_{A_1}}\Big).
\end{equation}

To complete the proof of our main theorem we need to show that a result similar to Proposition~$\ref{t_varpi}$ remains valid under a different conditioning.

Let $A_3\subset A_2$ be such that $|\partial A_3|<\infty$, and write $\I^u_{A_3}:=\I^u\cap A_3$ . Then $\mathrm{Cloth}_u(\partial A_3, \partial A_3)$ is well defined. Given $\hat{\psi}\in\mathcal{S}_u(\partial A_3,\partial A_3)$, we define $m_{\hat{\psi}}\equiv m_{\hat{\psi}}(\partial A_3)$ as a random element of 
$\mathcal{S}_u(V,\partial A_2)$ distributed as~$\mathrm{Cloth}_u(V, \partial A_2)$ conditioned on the event where the entrance and exit points at~$\partial A_3$ of the simple random walk excursions of $\I^u_{A_3}$ are given by~$\hat{\psi}$. We denote by $\I^u_{A_1\mid\hat{\psi}} $ the random interlacements process on~$A_1$ conditioned on the event where~$\mathrm{Cloth}_u(\partial A_3, \partial A_3)$ is equal to the deterministic element~$\hat\psi$. Notice that all ``information'' given by~$\I^u_{A_3}$ to $\I^u\cap A_3^C$ is contained in $\mathrm{Cloth}_u(\partial A_3,\partial A_3)$, that is, conditioned on $\mathrm{Cloth}_u(\partial A_3,\partial A_3)$, $\I^u_{A_3}$ and $\I^u\cap A_3^C$ are independent.

Inequality~\eqref{e_varpi} then implies, for $\hat{\psi}\in\mathcal{S}_u(\partial A_3,\partial A_3)$,
\begin{align}
\GP\big{[} \hat{\I}^{u(1-\eps)}_{A_1} \subseteq  \I^u_{A_1\mid\hat{\psi}}   \subseteq    \hat{\I}^{u(1+\eps)}_{A_1}   \big{]}
 & =\!\!\!\!\!\!\!\!\! \sum_{\hat{\zeta}\in\mathcal{S}_u(V,\partial A_2)} \!\!\!\!\GP\big{[} \hat{\I}^{u(1-\eps)}_{A_1} \subseteq  \I^u_{A_1\mid\hat{\psi}}   \subseteq    \hat{\I}^{u(1+\eps)}_{A_1}  \mid m_{\hat{\psi}}=\hat{\zeta} \big{]}\GP\big{[}  m_{\hat{\psi}}=\hat{\zeta}  \big{]}\nonumber\\
 & = \!\!\!\!\!\!\!\!\! \sum_{\hat{\zeta}\in\mathcal{S}_u(V,\partial A_2)} \!\!\!\!\GP\big{[} \hat{\I}^{u(1-\eps)}_{A_1} \subseteq  \I^u_{A_1\mid\hat{\zeta}}   \subseteq    \hat{\I}^{u(1+\eps)}_{A_1}   \big{]}\GP\big{[}  m_{\hat{\psi}}=\hat{\zeta}  \big{]}\nonumber\\
 & \geq   \Big(1-c \exp\Big(-\frac{c'}{2}\eps^2 u s^{a_{A_1}}\Big)\Big)\GP\big{[}  m_{\hat{\psi}}\in\mathcal{A} \big{]}\label{e_psi}.
\end{align}

Let $\mathcal{E}$ be the set of all $\hat{\psi}\in\mathcal{S}_u(\partial A_3,\partial A_3)$ such that
\begin{equation*}
\GP\big{[}  m_{\hat{\psi}}\in\mathcal{A}^C \big{]}\geq \sqrt{\IP^u_{V,\partial A_2}\big{[} \mathcal{A}^C\big{]}}.
\end{equation*}
Since
\begin{equation*}
\IP^u_{V,\partial A_2}\big{[} \mathcal{A}^C\big{]}=\int\GP\big{[}  m_{\hat{\psi}}\in\mathcal{A}^C \big{]}\IP^u_{\partial A_3,\partial A_3}\big{[} \d \hat{\psi}\big{]}\geq\IP^u_{\partial A_3,\partial A_3}\big{[}  \mathcal{E} \big{]}\sqrt{\IP^u_{V,\partial A_2}\big{[} \mathcal{A}^C\big{]}},
\end{equation*}
we have 
\begin{equation*}
\IP^u_{\partial A_3,\partial A_3}\big{[} \mathcal{E}\big{]}\leq \sqrt{\IP^u_{V,\partial A_2}\big{[} \mathcal{A}^C\big{]}}.
\end{equation*}
We have proved the following theorem, which implies Theorem~\ref{t_main1}:
\begin{theorem}
\label{t_main2}
Using the same notation as above, we have that, for constants $c,c'>0$, there exists a set $\mathcal{G}\in\sigma_u(\partial A_3, \partial A_3)$ such that
\begin{equation*}
\IP^u_{\partial A_3,\partial A_3}\big{[} \mathcal{G}\big{]}\geq 1-\exp\Big(-\frac{c'}{4}\eps^2 u s^{a_{A_1}}\Big),
\end{equation*}
and for all $\hat{\psi}\in\mathcal{G}$, 
\begin{equation}
\label{e_psi2}
\GP\big{[} \hat{\I}^{u(1-\eps)}_{A_1} \subseteq  \I^u_{A_1\mid\hat{\psi}}   \subseteq    \hat{\I}^{u(1+\eps)}_{A_1}   \big{]} \geq 1-c \exp\Big(-\frac{c'}{2}\eps^2 u s^{a_{A_1}}\Big).
\end{equation}
Moreover, for any increasing function $f$ on the interlacements set intersected with $A_1$, with $\sup |f| < M$, we have
\begin{align}
\label{e_conditionaldecoupling}
\lefteqn{\big(\IE(f( \I^{u(1-\eps)}_{A_1}))-cM\exp\big(-c'\eps^2 u s^{a_{A_1}}\big)\big)\1{\mathcal{G}} \leq  \IE(f(\I^{u}_{A_1})\mid \I^{u}_{A_3} )\1{\mathcal{G}}\nonumber} \phantom{*************************}
\\
\phantom{*****************}&\leq  \big(\IE (f( \I^{u(1+\eps)}_{A_1}))+cM\exp\big(-c'\eps^2 u s^{a_{A_1}}\big)\big)\1{\mathcal{G}}.
\end{align}
\end{theorem}
We finish the section with a brief proof of Theorem~\ref{t_main3}.

\begin{proof}[Proof of Theorem~\ref{t_main3}]
Note that, on equation~\eqref{e_moment3},~$\delta$ can be any real number greater than~$0$, whereas in equation~\eqref{e_moment4}, we need to have $0<\delta<1$. Recall that $u'>u>0$.
We have, by substituting the appropriate~$\delta$ in~$\eqref{e_moment5}$ and ignoring the union bound term~$cr^{2d-2}$,
\begin{align*}
\GP\big[G^\Sigma_{u}(z)< G^\Sigma_{u+u'}(z)    \big] &\geq 1-\GP\big[G^\Sigma_{u}(z)> (u+u' 4^{-1})\capacity(V)\pi(\Xi(z))   \big]
\\
&\quad -\GP\big[  G^\Sigma_{u+u'}(z)  <2^{-1}(u+u')\capacity(V)\pi(\Xi(z)) \big]
\\
&\geq 1-\exp\left(-¨\frac{c}{4}(u+u')s^{a_{A_1}}\right)-\exp\left(-\frac{c}{16} \frac{(u')^2}{u^2} us^{a_{A_1}}\right)
\\
&\geq 1-\exp\left(-c' u' s^{a_{A_1}}\right).
\end{align*}
Now, proceeding in the same manner as we did in the proof of Theorem~\ref{t_main2}, we are able to prove Theorem~\ref{t_main3}.
\end{proof}
\appendix
\section{Technical estimates}
\label{appendix}
\subsection{Bounding the relevant probabilities}
\label{s_t1}
For $w_0\in\partial A_1$ and $y_0\in V$ we want to bound the supremum
\begin{equation}
\label{supprob}
\sup_{\substack{w'\in V \\ y'\in \partial A_2}}\IP_{w',y'}\big{[} \Xi(w',y')=(w_0,y_0) \big{]}
\end{equation}
from above. To do so we will bound the ``hanging'' probability $\IP_{w,y}\big{[} \Xi(w,y)=(w_0,y_0) \big{]}$ for arbitrary~$w\in V$ and~$y\in \partial A_2$.

Given a finite nearest neighbor path $\gamma$, we denote by $|\gamma|$ its length. We will say that a path $\gamma$ belongs to an event $E$ if $E$ occurs every time the simple random walk $(X_k,k\geq 0)$ 
first $|\gamma|$ steps coincide with $\gamma$. 
We also let $\IP_x\big{[}\gamma\big{]}$ denote the probability that the first $|\gamma|$ steps of the simple random walk started at $x$ coincide with $\gamma$.
  
In order to avoid a cumbersome notation we now introduce what, hopefully, will be a simpler way to denote our events of interest. For $w,y_0\in V$, $w_0\in\partial A_1$ and $y\in\partial A_2$ we define:
\begin{itemize}
\item$w\xrightarrow{1}w_0$: 
The collection of all finite nearest-neighbor trajectories starting at $w$ that do not reach neither $\partial A_1$ nor $\partial A_2$, except at its ending point $w_0\in\partial A_1$. 
Note that this collection can be thought of as the event where the simple random walk started at $w$ hits~$\partial A_1$ for the first time at $w_0$ before reaching~$\partial A_2$.
\\
\item$w_0\xrightarrow{2}y_0$: The collection of all finite nearest-neighbor trajectories starting at $w_0$ and ending at $y_0$ without reaching~$\partial A_2$.
\\
\item$y_0\xrightarrow{3}y$: The collection of all finite nearest-neighbor trajectories starting at $y_0$ that hit~$\partial A_2$ for the first time at $y$ before returning to $V$. 
Note that this collection can be thought of as the event where the simple random walk started at $y_0$ hits $\partial A_2$ before returning to $V$ and its entrance point in $\partial A_2$ is $y$.
\\
\item$w\xrightarrow{4}y$: The event where the entrance point in $\partial A_2$ of the simple random walk started at~$w$ is~$y$. 
This event clearly can also be regarded as a collection of simple random walk trajectories starting at $w$ and hitting $\partial A_2$ for the first time at $y$.
\end{itemize}
\begin{figure}[ht]
\centering
\includegraphics[scale = 1]{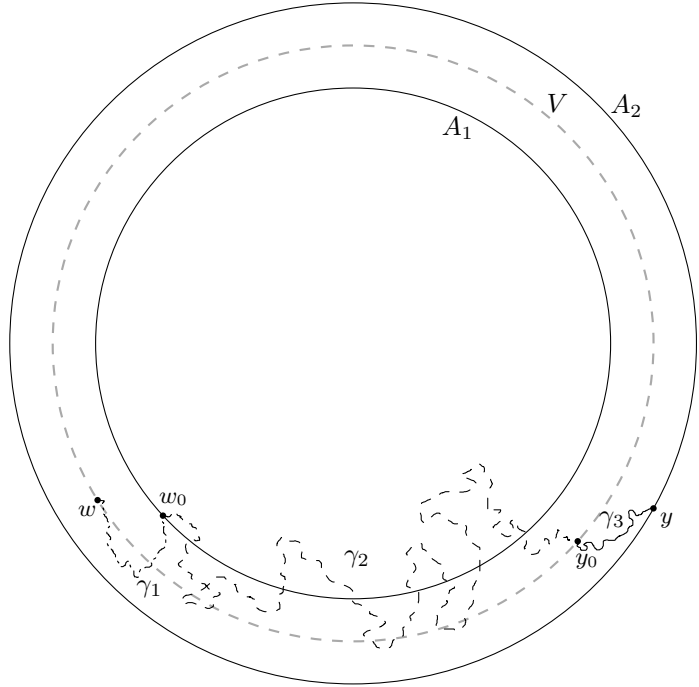}
\vspace{0.5cm}
\caption{$\gamma$ as the concatenation of the three paths $\gamma_1$, $\gamma_2$ and $\gamma_3$.}
\label{w1w02y03yfig1}
\end{figure}

We also let $w\xrightarrow{1}w_0\xrightarrow{2}y_0\xrightarrow{3}y$ be the ``concatenation'' of the first three collections, where the first trajectory's ending point becomes the second trajectory's
 starting point and so on. That is, if $\gamma\in w\xrightarrow{1}w_0\xrightarrow{2}y_0\xrightarrow{3}y$ 
then $\gamma$ is the concatenation of three distinct paths: $\gamma_1\in w\xrightarrow{1}w_0$, $\gamma_2 \in w_0\xrightarrow{2}y_0$, $\gamma_3 \in y_0\xrightarrow{3}y$. Note that, as an event,
\begin{equation*}
w\xrightarrow{1}w_0\xrightarrow{2}y_0\xrightarrow{3}y = \{\Xi(w',y')=(w_0,y_0)   \}.
\end{equation*}
With our new notation the hanging probability becomes
\begin{equation}
\IP_w\big{[} w\xrightarrow{1}w_0\xrightarrow{2}y_0\xrightarrow{3}y \mid w\xrightarrow{4}y \big{]}=\frac{\IP_w\big{[} w\xrightarrow{1}w_0\xrightarrow{2}y_0\xrightarrow{3}y  \big{]}}{\IP_w\big{[} w\xrightarrow{4}y \big{]}}.
\end{equation}
We have
\begin{equation}
\label{pathproba}
\begin{array} {e}
\IP_w\big{[} w\xrightarrow{1}w_0\xrightarrow{2}y_0\xrightarrow{3}y \big{]} & = & \sum_{\gamma\in w\xrightarrow{1}w_0\xrightarrow{2}y_0\xrightarrow{3}y}\frac{1}{2^{|\gamma|}} \\ & = & \sum_{\gamma_{1}\in w\xrightarrow{1}w_0}\frac{1}{2^{|\gamma_1|}}\sum_{\gamma_2 \in w_0\xrightarrow{2}y_0}\frac{1}{2^{|\gamma_2|}}\sum_{\gamma_3 \in y_0\xrightarrow{3}y}\frac{1}{2^{|\gamma_3|}}.
\end{array}
\end{equation}
Let us focus on the second sum, $\sum_{\gamma_2 \in w_0\xrightarrow{2}y_0}\frac{1}{2^{|\gamma_2|}}$, for a moment. Each path $\gamma_2 \in w_0\xrightarrow{2}y_0$ can be seen as the concatenation of one path $\gamma^{0}_{2}$ 
responsible for the walk's first visit to~$y_0$ and a sequence of paths $\gamma^{1}_{2},\dots,\gamma^{k}_{2}$ associated with the returns the walk makes to $y_0$ before hitting $\partial A_2$, see Figure~$\ref{w02y0fig1}$. So that
\begin{figure}[ht]
\centering
\includegraphics[scale = 1]{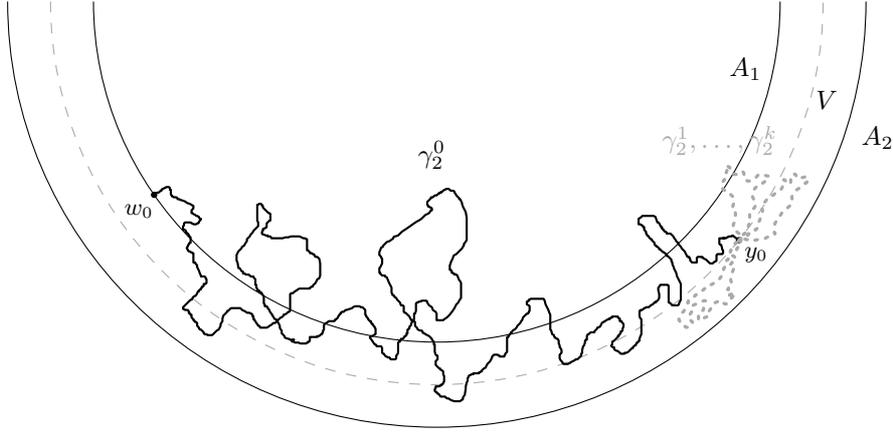}
\vspace{0.5cm}
\caption{$\gamma_{2}$ as the concatenation of the paths $\gamma^{0}_{2},\gamma^{1}_{2},\dots,\gamma^{k}_{2}$.}
\label{w02y0fig1}
\end{figure}
\begin{equation}
\label{pathprobb}
\sum_{\gamma_2}\IP_{w_0}\big{[}\gamma_2 \big{]}=\sum_{\gamma^{0}_{2}}\IP_{w_0}\big{[}\gamma^{0}_{2} \big{]}\sum_{k\geq 1}\sum_{\gamma^{1}_{2},\dots,\gamma^{k}_{2}}\IP_{y_0}\big{[}\gamma^{1}_{2} \big{]}\dots\IP_{y_0}\big{[}\gamma^{k}_{2} \big{]}.
\end{equation}
But for a fixed $k_0 > 0$, the last sum $\sum_{k\geq k_0}\sum_{\gamma^{1}_{2},\dots,\gamma^{k_0}_{2}}\IP_{y_0}\big{[}\gamma^{1}_{2} \big{]}\dots\IP_{y_0}\big{[}\gamma^{k_0}_{2} \big{]}$ equals the probability that the simple random walk started at $y_0$ returns 
to $y_0$ at least $k_0$ times before hitting $\partial A_2$. 
Since the walk is transient, we can use the strong Markov property to show that there exists a constant $0<c<1$ such that
\begin{equation}
\label{pathprobc}
\sum_{k\geq k_0}\sum_{\gamma^{1}_{2},\dots,\gamma^{k_0}_{2}}\IP_{y_0}\big{[}\gamma^{1}_{2} \big{]}\dots\IP_{y_0}\big{[}\gamma^{k_0}_{2} \big{]}<c^{k_0}.
\end{equation}
We have thus shown the existence of a constant $c>0$ such that
\begin{equation}
\label{pathprobd}
\sum_{\gamma^{0}_{2}}\IP_{w_0}\big{[}\gamma^{0}_{2} \big{]}\leq\sum_{\gamma_2}\IP_{w_0}\big{[}\gamma_2 \big{]}\leq c\sum_{\gamma^{0}_{2}}\IP_{w_0}\big{[}\gamma^{0}_{2} \big{]}
\end{equation}
where $\gamma^{0}_{2}$ represents any nearest neighbor path that starts at $w_0$ and ends at its only visit to~$y_0$, without ever reaching $\partial A_2$. Let us update our collection's definition in view of this last computation. We denote by
\begin{itemize}
\item$w_0\xrightarrow{2'}y_0$: The collection of all finite nearest-neighbor paths starting at $w_0$ and ending at their first visit to $y_0$, without hitting $\partial A_2$. This collection now can be thought of as the event where the simple random walk started at $w_0$ makes a visit to $y_0$ before hitting~$\partial A_2$. 
\end{itemize}
Combining $(\ref{pathproba})$ with $(\ref{pathprobd})$ we get
\begin{equation}
\label{pathprobe}
\IP_w\big{[} w\xrightarrow{1}w_0\xrightarrow{2}y_0\xrightarrow{3}y \big{]}\leq c  \IP_w\big{[} w\xrightarrow{1}w_0 \big{]}\IP_{w_0}\big{[} w_0\xrightarrow{2'}y_0\big{]}\IP_{y_0}\big{[}y_0\xrightarrow{3}y \big{]}.
\end{equation}
Our work will now reside in giving upper bounds for these three probabilities, besides giving a lower bound for $\IP_w\big{[} w\xrightarrow{4}y \big{]}$.

There will be two results about the simple random walk we will make extensive use of. The first, which can be seen as a direct consequence of Proposition 6.5.4 of \cite{LawlerLimic}, 
essentially says that the probability that the random walk started at a distance at least $h_0$ from a sphere of radius $h_0$ enters that sphere at a specific point is of order $h_{0}^{-(d-1)}$, 
that is, the hitting measure on a sphere is comparable to the uniform distribution when the starting point of the walk is sufficiently distant.  The second result is a simple application of the optional stopping theorem for submartingales and supermartingales, and can be seen in the proof of Lemma $8.5$ of \cite{PopovTeixeira}. 
We state it here for the reader's convenience.

\begin{lemma}
\label{l_popovteixeira}
Let $0<\rho_1<\rho_2$ be sufficiently large real numbers, and let $x\in B(0,\rho_2)\setminus B(0,\rho_1)$. Then
\begin{equation}
\frac{|x|^{-\big(d-\frac{5}{2}\big)}-(\rho_2-1)^{-\big(d-\frac{5}{2}\big)}}{(\rho_1+1)^{-\big(d-\frac{5}{2}\big)}-(\rho_2)^{-\big(d-\frac{5}{2}\big)}} \leq \IP_x\big{[} H_{\partial B(0,\rho_1)}<H_{\partial B(0,\rho_2)} \big{]} \leq \frac{|x|^{-(d-1)}-(\rho_2)^{-(d-1)}}{(\rho_1-1)^{-(d-1)}-(\rho_2)^{-(d-1)}}.
\end{equation}
\end{lemma}

\vspace{0.5cm}

\subsubsection{The hanging probabilities for the ball}

In this subsection we will be concerned with the sets $A_1^{\tiny\Circle},V^{\tiny\Circle}$ and $A_2^{\tiny\Circle}$, and the related simple random walk probabilities. 

\vspace{0.5cm}

$\boldsymbol{w\xrightarrow{1}w_0}$: Let $h_1$ be the Euclidean distance between $w$ and $w_0$. We look to $\Z^d$ as a subset of $\R^d$. Let $e_1,\dots,e_d$ be the canonical basis of $\R^d$. Without loss of generality we assume that~$w$ and~$w_0$  belong to the plane generated by the first vectors $e_1,e_2$. If $\rho,\Phi_1,\dots,\Phi_{d-1}$ are the corresponding spherical coordinates of $\R^d$, we let, for $i_1=1,\dots,\big\lfloor\frac{2\pi r}{s}\big\rfloor$ and $i_k=1,\dots,\big\lfloor\frac{\pi r}{s}\big\rfloor$, $k=2,\dots,d-1$:
\begin{equation}
E_{i_1,\dots,i_{d-1}}= \left\{ \begin{array}{l}
(\rho,\Phi_1,\dots,\Phi_{d-1}) \in\R^d,\text{ }r\leq\rho\leq r+2s,\text{ }\frac{(i_1-1)s}{2\pi r}\leq \Phi_1 \leq \frac{i_1 s}{2\pi r},\\ \text{ }\frac{(i_k -1)s}{\pi r}\leq \Phi_k \leq \frac{i_k s}{\pi r}\text{ for all } k=2,\dots,d-1
\end{array} \right\}
\end{equation}
We also let $C_1$ be a discrete ball of radius $s$ contained in $A_1$ in such a way  that it intersects~$\partial A_1$ only at $w_0$. We refer any reader skeptic about the existence of such discrete ball to \cite{PopovTeixeira}, Section~$8$. 
There is a constant $c_1>0$ such that the random walk started at~$w$ will have to cross at least $\big\lfloor \frac{c_1 h_1}{s} \big\rfloor$ sets of the form $E_{i_1,\dots,i_{d-1}}$ to reach $w_0$. Each time the walk reaches a set $E_{i_1 ',\dots,i_{d-1} '}$, the probability that it will reach another set of the form $E_{i_1,\dots,i_{d-1}}$ at  distance at least $s$ from $E_{i_1 ',\dots,i_{d-1} '}$, before hitting either $\partial A_1$ or $\partial A_2$, is bounded from above by a constant $0<c_2<1$, as can be seen using Donsker's Invariance Principle (see Section~$3.4$ of \cite{LawlerLimic}). 
Using the strong Markov property, we can show that the probability that the walk started at $w$ crosses at least  $\frac{c_1 h_1}{s}$ sets of the form $E_{i_1,\dots,i_{d-1}}$ 
before hitting~$\partial A_1\cup\partial A_2$ is smaller than $c_{2}^{\lfloor\frac{c_1 h_1}{s}\rfloor}$.

\begin{figure}[ht]
\centering
\includegraphics[scale = 1]{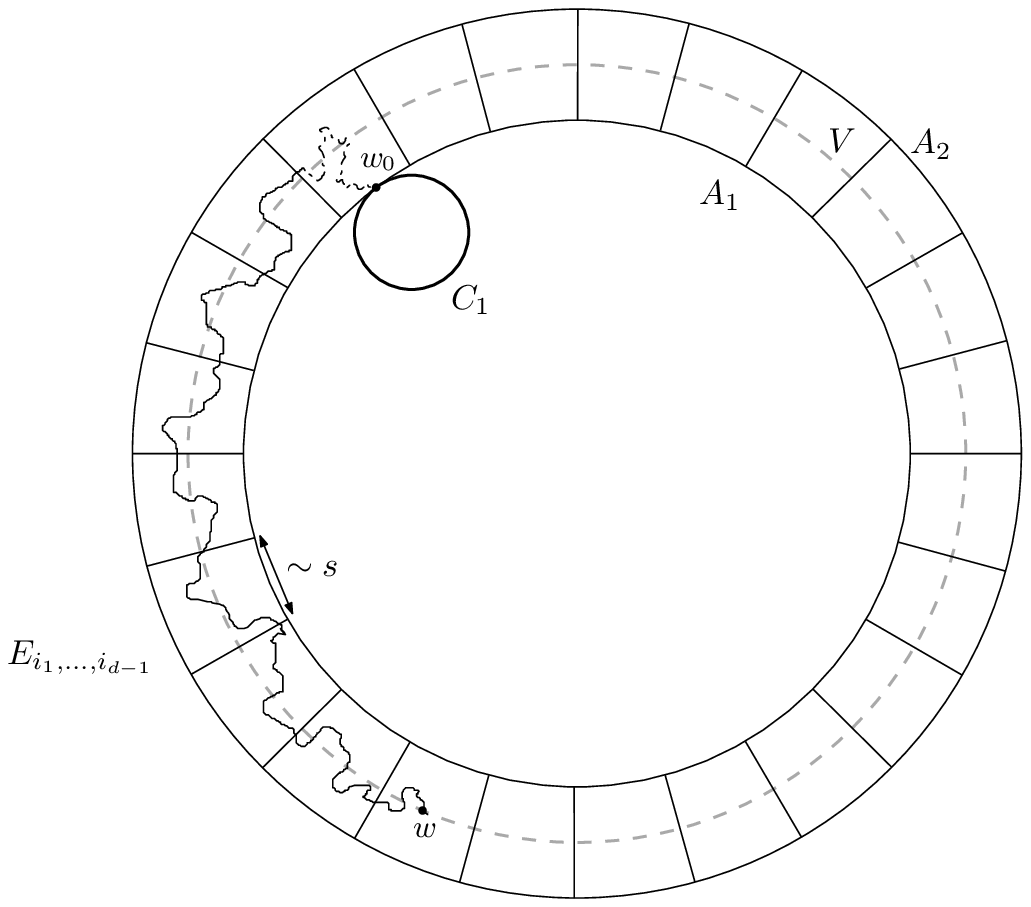}
\vspace{0.5cm}
\caption{A path belonging to $w\xrightarrow{1}w_0$ has to cross $\frac{c_1 h_1}{s}$ sets of the form $E_{i_1,\dots,i_{d-1}}$ before hitting $w_0$ in $C_1$.}
\label{w1w0fig1}
\end{figure}

We note that it is harder for a walk  started at some $x\in A_2^C\setminus A_1$ to hit $w_0$ before any other point in $\partial A_1$ than it is to hit $w_0$ before any other point in $C_1$,
\begin{equation*}
\IP_x\big{[}    X_{H_{\partial A_1}}=w_0     \big{]}\leq \IP_x\big{[}    X_{H_{\partial C_1}}=w_0     \big{]}.
\end{equation*}

We have already noted that the probability of hitting a discrete sphere of radius $s$ at a specific point at distance of order $s$, is of order $s^{-(d-1)}$, as can be seen in Proposition~$6.5.4$ of \cite{LawlerLimic}. In conjunction with last paragraph's argument, this shows the existence of  constants $c_3,c_4>0$ such that
\begin{equation}
\label{pathprobf}
\IP_w\big{[} w\xrightarrow{1}w_0 \big{]}\leq c_4 e^{\frac{-c_3 h_1}{s}}s^{-(d-1)}.
\end{equation}

$\boldsymbol{y_0\xrightarrow{3} y} $: We define $y\xrightarrow{3'} y_0 $ to be the event where the walk, started at $y$, hits $y_0$ in $V$ before reaching any other point in $V$ or $\partial A_2$. From the simple random walk's reversibility, we have
\begin{equation}
\label{pathprobg}
\IP_y\big{[} y\xrightarrow{3'}y_0 \big{]}=\IP_{y_0}\big{[} y_0\xrightarrow{3}y \big{]}.
\end{equation}

Let $C_1$ now be a discrete ball of radius $\frac{s}{2}$ contained in $A_2$ in such a way that $C_1\cap\partial A_2=\{y\}$, and let $C_2$ be a discrete ball of radius~$\frac{s}{3}$ concentric with $C_1$. We can use Lemma~$\ref{l_popovteixeira}$ and some elementary calculus to show that if the simple random walk starts at $y$, 
the probability that it hits $C_2$ before hitting $C_{1}$ is of order $s^{-1}$. Using the strong Markov property, the argument then continues the same way as the argument for the bound 
for $\IP_w\big{[} w\xrightarrow{1}w_0 \big{]}$. Let $h_3$ be the Euclidean distance between $y_0$ and $y$. Then there are constants $c_1,c_2>0$ such that
\begin{equation}
\label{pathprobh}
\IP_w\big{[} y\xrightarrow{3}y_0 \big{]}\leq c_1 e^{\big\lfloor\frac{-c_2 h_3}{s}\big\rfloor}s^{-(d-1)}s^{-1}.
\end{equation}

$\boldsymbol{w_0\xrightarrow{2'}y_0}$: Let $h$ be the Euclidean distance between $w_0$ and $y_0$. Assume $h>20s$. Let~$w_1$ be the   point  on $\partial A_2$ closest to $w_0$. We let $C_2$ now be a discrete ball of radius~$\frac{h}{6}$ that intersects $\partial A_2$ only at $w_1$ and lies outside of $A_2^C$. We let $C_3$ be the discrete ball of radius~$\frac{h}{3}$ that is concentric with $C_2$. In order for the walk started at $w_0$ to reach $y_0$ without leaving $A_2^C$, it first has to reach $\partial C_3$ before hitting $C_2$. Lemma~$\ref{l_popovteixeira}$ and some calculus show that the probability of such event is of order $\frac{s}{h}$.

\begin{figure}[ht]
\centering
\includegraphics[scale = 1]{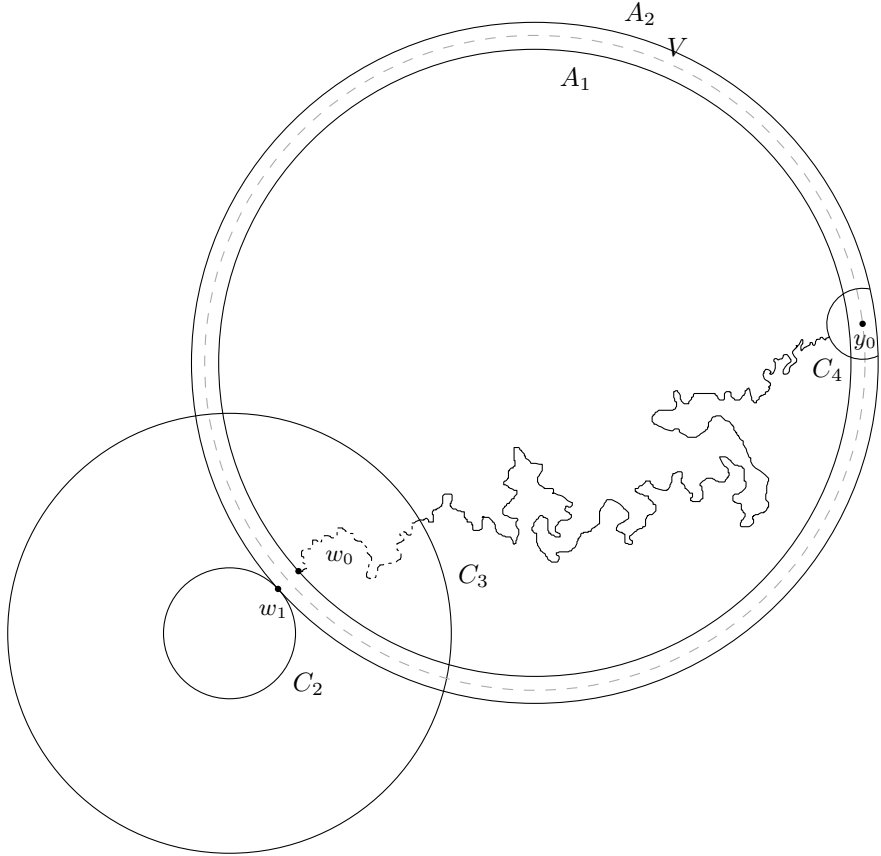}
\vspace{0.5cm}
\caption{A walk started at $w_0$ has to reach $\partial C_3$ before $\partial C_2$ and then reach $C_4\setminus A_{2}^{C}$ in order to reach $y_0$.}
\label{w02y0fig2}
\end{figure}

In order for the walk to reach a $y_0$, it has first to reach a sphere $\partial C_4$ of radius $3s$ centered at $y_0$. Conditioned on the event where $\partial C_4$ is reached before the walk hits $\partial A_2$, the probability that the walk reaches $y_0$ before reaching $\partial A_2$ is smaller than $c s^{-(d-2)}$, for a constant $c>0$, as can be seen using the Green's function estimate \eqref{greenestimate}.

\begin{figure}[ht]
\centering
\includegraphics[scale = 1]{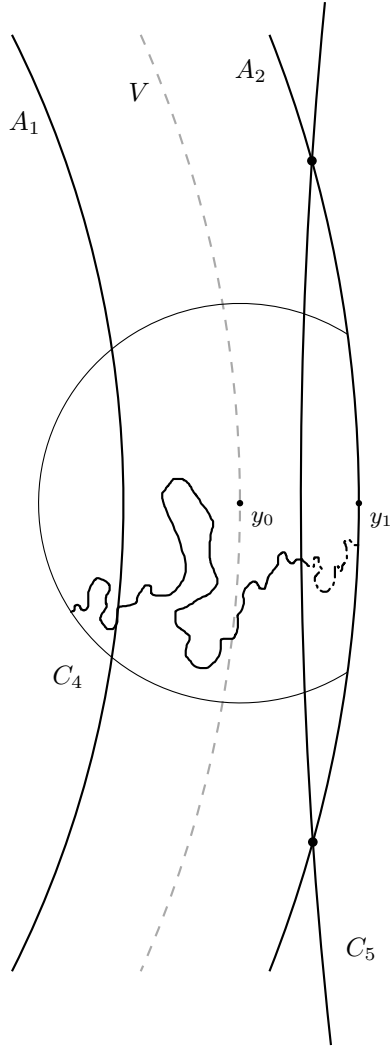}
\vspace{0.5cm}
\caption{We show that, if starting at a distant point $w_2$, the probability of the simple random walk hitting $C_4\setminus A_{2}^{C}$ and the probability of hitting $C_5\cap A_2^C$ are comparable.}
\label{w02y0fig3}
\end{figure}

Let $y_1$ be the   point  on $\partial A_2$ closest to $y_0$. Let $C_5$ be a discrete ball of radius $h$ such that the intersection $C_5\cap\partial A_2$ has diameter $6s$ and center of mass as close as possible to $y_1$. By Donsker's Invariance Principle, there is a 
constant $c_1>0$ such that a simple random walk started at any point in $\partial C_4\cap A_2^C$ has probability at least $c_1$ of reaching $C_5\cap\partial A_2$ before $\partial A_2 \setminus C_5$. 
Let $w_2\in A_2^C$ be any point at distance at least $\frac{h}{2}$ from $y_0$. For a simple random walk starting at $w_2$ we define the events:

$D_{C_5\cap \partial A_2}:=\{H_{C_5\cap\partial A_2}\leq H_{\partial A_2\setminus C_5}\}$; the event where the simple random walk reaches $C_5\cap\partial A_2$ before reaching any other point in $\partial A_2$.

$D_{\partial C_5\cap A_2^C}:=\{H_{\partial C_5\cap A_2^C}\leq H_{\partial A_2\setminus C_5}\}$; the event where the simple random walk reaches $\partial C_5\cap A_2^C$ before reaching any other point in $\partial A_2$.

$D_{C_4\setminus A_{2}}:=\{H_{C_4\setminus A_{2}}\leq H_{\partial A_2\setminus C_4}\}$; the event where the simple random walk reaches $C_4\setminus A_{2}$ before reaching any other point in $\partial A_2$.

$D_{y_0}:=\{H_{y_0}\leq H_{\partial A_2}\}$; the event where the simple random walk reaches $y_0$ before hitting $\partial A_2$.

From the above discussion it is clear that:
\begin{equation}
\label{pathprobi}
\IP_{w_2}\big{[} D_{C_5\cap \partial A_2} \big{]} \leq \IP_{w_2}\big{[} D_{\partial C_5\cap A_2^C} \big{]}   ,
\end{equation}
\begin{equation}
\label{pathprobj}
\IP_{w_2}\big{[} D_{y_0} \big{]} = \IP_{w_2}\big{[} D_{y_0}\mid D_{C_4\setminus A_{2}} \big{]} \IP_{w_2}\big{[} D_{C_4\setminus A_{2}}\big{]}   ,
\end{equation}
\begin{equation}
\label{pathprobk}
\IP_{w_2}\big{[} D_{C_4\setminus A_{2}^{C}} \big{]} \leq \frac{1}{c_1}\IP_{w_2}\big{[} D_{C_5\cap \partial A_2} \big{]}   .
\end{equation}
Using Proposition 6.5.4 of \cite{LawlerLimic} we can see that there is a constant $c>0$ such that
\begin{equation}
\label{pathprobl}
\IP_{w_2}\big{[} D_{\partial C_5\cap A_2^C} \big{]} \leq c\frac{s^{d-1}}{h^{d-1}}  .
\end{equation}
Collecting the estimates $(\ref{pathprobi},\ref{pathprobj},\ref{pathprobk},\ref{pathprobl})$, using the strong Markov property, and bounding 
\begin{equation*}
\IP_{w_2}\big{[}D_{y_0}\mid~D_{C_4\setminus A_{2}}\big{]}
\end{equation*}
by the Green's function estimate \eqref{greenestimate}, we see that there is a constant $c>0$ such that
\begin{equation}
\label{pathprobm}
\IP_{w_0}\big{[} w_0\xrightarrow{2'}y_0 \big{]}\leq c \frac{s}{h}\cdot\frac{s^{d-1}}{h^{d-1}}s^{-(d-2)}  = c\frac{s^{2}}{h^{d}}  .
\end{equation}
If $h<20s$ the result follows after using Green's Function.

We also provide a lower bound for $\IP_{w_0}\big{[} w_0\xrightarrow{2'}y_0 \big{]}$, which we will need later. Suppose $h\leq \frac{r}{2}$. Let $C_3 '$ be a discrete ball of radius $2h$ contained in $A_2^C$ that intersects $\partial A_2$ only at~$w_1$. Let $C_2 '$ be a discrete ball of radius $\frac{h}{2}$ concentric with $C_3 '$. 
Let us describe an event of probability greater than $c_1\frac{s^2}{h^d}$, for some constant $c_1>0$, that is contained in $ w_0\xrightarrow{2'}y_0$. First the walk needs to hit~$\partial C_2 '$ before hitting $\partial C_3 '$. The probability of such event is of 
order $\frac{s}{h}$, as can be seen using Lemma~$\ref{l_popovteixeira}$. We will denote by $w_2$ the point in which the walk enters $\partial C_2 '$.

We define $C_5 '$ to be the discrete ball  of radius $2h$ such that its center lies inside $A_2^C$ and the intersection $C_5 '\cap \partial A_2$ coincides with $C_4 \cap \partial A_2$. 
In addition to all events defined in the proof of the upper bound for $\IP_{w_0}\big{[} w_0\xrightarrow{2'}y_0 \big{]}$, we define the event, for a simple random walk starting in the interior of $C_5 '$: \\

$D_{\partial C_5 '\setminus  A_2^C}:=\{H_{\partial C_5 '\setminus A_2^C}\leq H_{\partial A_2\setminus C_5 '}\}$; the event where the simple random walk started in the interior of $C_5 '$ reaches $\partial C_5 '\setminus A_2^C$ before reaching $ \partial A_2\setminus C_5 ' $. \\

We note that $w_2$ is in the interior of $C_5 '$ and that $D_{\partial C_5 '\setminus A_2^C}\subset D_{\partial C_4\setminus A_{2}}$. We then have:
\begin{equation}
\label{pathprobp}
\begin{array}{e}
\IP_{w_0}\big{[} w_0\xrightarrow{2'}y_0 \big{]} & \geq & \sum_{w_2\in\partial C_2 '}\IP_{w_0}\big{[} H_{\partial C_2 '}< H_{\partial C_3 '},\space X_{H_{\partial C_2 '}}=w_2 \big{]} \IP_{w_2}\big{[} D_{y_0} \big{]}
\\
& = & \sum_{w_2\in\partial C_2 '}\IP_{w_0}\big{[} H_{\partial C_2 '}< H_{\partial C_3 '},\space X_{H_{\partial C_2 '}}=w_2 \big{]} \IP_{w_2}\big{[} D_{y_0} \mid D_{C_4\setminus A_{2}}\big{]}\IP_{w_2}\big{[} D_{C_4\setminus A_{2}}\big{]}
\\
& \geq & \sum_{w_2\in\partial C_2 '}\IP_{w_0}\big{[} H_{\partial C_2 '}< H_{\partial C_3 '},\space X_{H_{\partial C_2 '}}=w_2 \big{]} \IP_{w_2}\big{[} D_{y_0} \mid D_{C_4\setminus A_{2}}\big{]}\IP_{w_2}\big{[} D_{\partial C_5 ' \setminus A_{2}}\big{]}.
\end{array}
\end{equation}
Using Harnack's Principle $($Theorem $6.3.9$ 
of \cite{LawlerLimic}$)$ we are able to show the existence of a constant $c_2>0$ such that
\begin{equation}
\IP_{w_2}\big{[} D_{\partial C_5 ' \setminus A_{2}}\big{]} \geq c_2\frac{s^{d-1}}{h^{d-1}}.
\end{equation}
With this and $(\ref{pathprobp})$ we can find a constant $c_1>0$ such that
\begin{equation}
\IP_{w_0}\big{[} w_0\xrightarrow{2'}y_0 \big{]} \geq c_1\frac{s^{2}}{h^{d}} .
\end{equation}
If $h\geq\frac{r}{2}$ we simply replace the balls $C_3 '$ and $C_5 '$ by $A_2^C$, the ball $C_2'$ by a ball concentric with $A_2^C$ but with the diameter halves, and continue the proof identically.

$\boldsymbol{w\xrightarrow{4}y}$: Let $w_3$ be the closest point to $w$ in $\partial A_2$. Let $h_4$ be the Euclidean distance between $w$ and $y$, and suppose $h_4\leq \frac{r}{2}$. Let $C_6$ be a discrete ball of radius $2h_4$ contained in $A_2^C$ that intersects $\partial A_2$ 
only at $w_3$. Let $C_7$ be a discrete ball of radius $\frac{h_4}{2}$ concentric with $C_6$. 
Then again Lemma~$\ref{l_popovteixeira}$ and some calculus show that the probability that a simple random walk started at $w$ will reach~$\partial C_7$ before reaching~$\partial C_6$ is less than the probability that the same walk will reach~$\partial C_7$ before hitting~$\partial A_2$ and bigger than~$c_1\frac{s}{h_4}$, for some constant~$c_1>0$.

Let $C_8$ be a discrete ball of radius $2h_4$ contained in $A_2^C$ that intersects $\partial A_2$ only at~$y$. Let $y_3$ be a fixed point in $\partial C_7$. Then the probability that a simple random walk started at $y_3$ hits $y$ before hitting any other point in $\partial C_8$ is smaller than the probability that the same walk 
reaches $y$ before any other point in $\partial A_2$ and bigger than $\frac{c_2}{h_{4}^{d-1}}$, for some constant~$c_2>0$, by 
the Harnack's Principle $($Theorem $6.3.9$ of \cite{LawlerLimic}$)$ and Lemma~$6.3.7$ of~\cite{LawlerLimic}. Figure~$\ref{w4yfig3}$ illustrates the argument. Using the strong Markov property, we then have
\begin{equation}
\label{pathprobn}
\IP_{w_0}\big{[} w\xrightarrow{4}y \big{]} \geq c\frac{s}{h_4}h_{4}^{-(d-1)}  .
\end{equation}
\begin{figure}[ht]
\centering
\includegraphics[scale = .7]{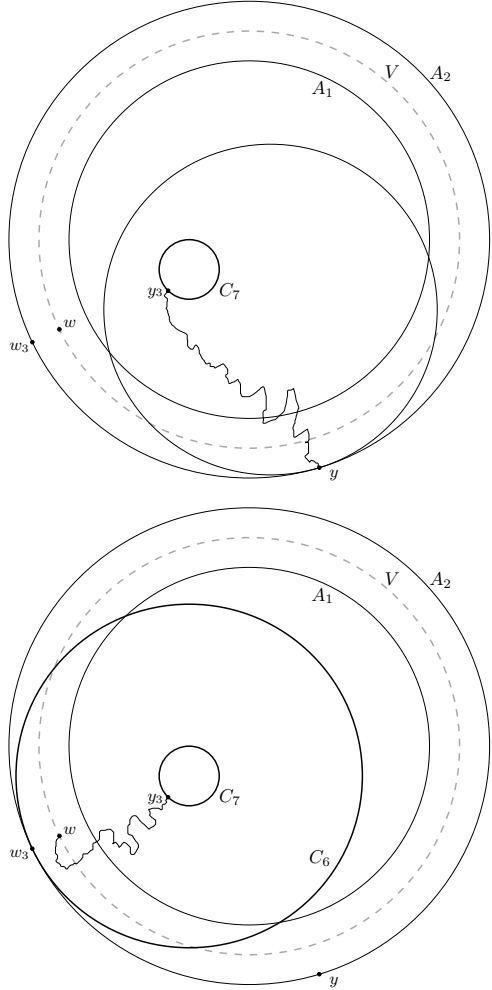}
\vspace{0.5cm}
\caption{We can give a lower bound for $\IP_{w_0}\big{[} w\xrightarrow{4}y \big{]}$ by describing the event where the walk started at $w$ reaches a small sphere $C_6$ before reaching $\partial C_7$ and then hits $y$ before any other point in $\partial C_8$. }
\label{w4yfig3}
\end{figure}

If $h_4\geq\frac{r}{2}$ we simply replace the balls $C_6$ and $C_8$ by $A_2^C$, the ball $C_7$ by an discrete ball concentric with $A_2^C$ but with half the diameter, and continue the proof identically.

Let us now provide an upper bound for $\IP_{w_0}\big{[} w\xrightarrow{4}y \big{]}$, which will be needed in the next section. We let $C_6 '$ be a discrete ball of radius $\frac{h_4}{6}$ lying outside $A_2^C$ and intersecting $\partial A_2$ only at $w_3$. 
We also let $C_7 '$ be a discrete ball of radius $\frac{h_4}{3}$ concentric with $C_6 '$. Finally we let $C_8 '$ be a discrete ball of radius $h_4$ lying outside $A_2^C$ and intersecting $\partial A_2$ only at $y$.

Then, for the simple random walk started at $w$ to hit $\partial A_2$ at $y$, it has first to reach~$\partial C_7 '$ before hitting~$\partial C_6 '$ and then hit~$y$ before any other point 
in $\partial C_8 '$. As we have already seen, the probability of the first event is of order~$\frac{s}{h_4}$ and the probability of the latter is of order~$h_{4}^{-(d-1)}$. This way, we can find a constant~$c>0$ such that:
\begin{equation}
\IP_{w_0}\big{[} w\xrightarrow{4}y \big{]} \leq c\frac{s}{h_4}h_{4}^{-(d-1)}.
\end{equation}

Finally, using ($\ref{pathprobh}$) and ($\ref{pathprobi}$) we see that the supremum in ($\ref{supprob}$) is reached when~$h_1$ and~$h_3$ are of order $s$. This way, $h$ should have the same order as $h_4$. 
Gathering the bounds ($\ref{pathprobh}$), ($\ref{pathprobi}$), ($\ref{pathprobm}$) and ($\ref{pathprobn}$) we have, for a constant $c>0$
\begin{equation}
\label{pathprobo}
\sup_{\substack{w\in V \\ y\in \partial A_2}}\IP_w\big{[} w\xrightarrow{1}w_0\xrightarrow{2}y_0\xrightarrow{3}y \mid w\xrightarrow{4}y \big{]}\leq c s ^{-2(d-1)}.
\end{equation}

We have proved the following proposition:

\begin{proposition}
\label{p_boundprocircle}
Regarding the sets $A_1^{\tiny\Circle},V^{\tiny\Circle}$ and $A_2^{\tiny\Circle}$, we have that, using the notation defined above, for some constants $c_k>0$, $k=1,\dots,9$, the following bounds are valid:
\begin{equation*}\IP_w\big{[} w\xrightarrow{1}w_0 \big{]}\leq c_1 e^{\frac{-c_2 h_1}{s}}s^{-(d-1)},\end{equation*}
\begin{equation*}\IP_w\big{[} y\xrightarrow{3}y_0 \big{]}\leq c_3 e^{\frac{-c_4 h_3}{s}}s^{-(d-1)}s^{-1},\end{equation*}
\begin{equation*}c_5\frac{s^{2}}{h^{d}} \leq \IP_{w_0}\big{[} w_0\xrightarrow{2'}y_0 \big{]} \leq c_6\frac{s^{2}}{h^{d}},\end{equation*}
\begin{equation*}  c_7\frac{s}{h_{4}^d}\leq\IP_{w_0}\big{[} w\xrightarrow{4}y \big{]} \leq c_8\frac{s}{h_{4}^d} .\end{equation*}
\begin{equation*}\sup_{\substack{w\in V \\ y\in \partial A_2}}\IP_w\big{[} w\xrightarrow{1}w_0\xrightarrow{2}y_0\xrightarrow{3}y \mid w\xrightarrow{4}y \big{]}\leq c_9 s ^{-2(d-1)} .\end{equation*}
\end{proposition}

\subsubsection{The hanging probabilities for the smoothed hypercube}

In this subsection we will focus on sets $A_1^{\tiny\Square},V^{\tiny\Square}$ and $A_2^{\tiny\Square}$, and the related simple random walk probabilities. 

\vspace{0.1cm}

$\boldsymbol{w\xrightarrow{1}w_0}$: We will essentially use the same argument used when the underlying sets were balls. We assume without loss of generality that $\mathfrak{H}_{r+2s}$ is centered at the origin, and let $h_1:=\dist (w_0,y_0)$. 
We will subdivide the set $A_2^{C}\setminus A_1$ in sets of diameter of order $ s$ in such a way that for a simple random walk trajectory started at $w$ to reach $w_0$ it will first have to cross a number of order $\frac{h_1}{s}$ of these sets.

Given $j\in\{1,\dots,d\}$, $(m_1,\dots,m_d)\in\{-1,1\}^d$, $k\in\{1,\dots,j-1,j+1,\dots,d\}$, and $i_k\in\{1,\dots,\big\lfloor\frac{r}{s}\big\rfloor \}$, we define
\begin{equation*}
\begin{array}{e}  
E_{j,i_1,\dots,i_{j-1},i_{j+1},\dots,i_d}^{(m_1,\dots,m_d)} & = &
\Big\{     
(x_1,\dots,x_d)\in\Z^d;\\ &   & \quad x_j\in[\min\{m_j 2^{-1}r,m_j (2^{-1}r+2s)\},\max\{m_j 2^{-1}r,m_j (2^{-1}r+2s)\}],\\ &   &  \quad x_k\in[m_k (i_k-1) s,m_k i_k s]\cup[m_k i_k s,m_k (i_k-1)s].
\Big\} 
\end{array}
\end{equation*}
\begin{figure}[ht]
\centering
\includegraphics[scale = 1]{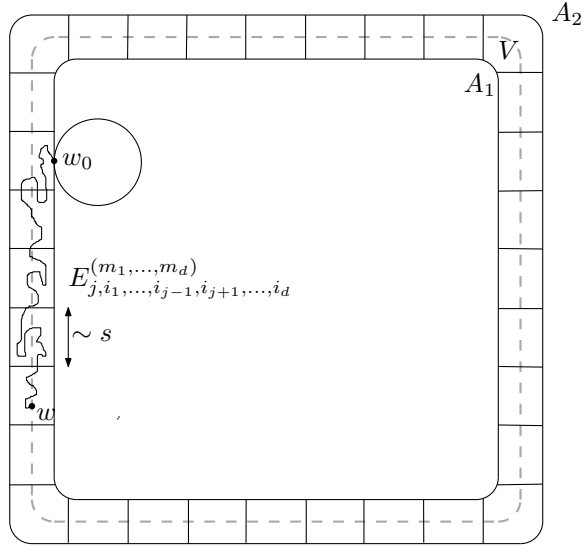}
\caption{A path belonging to $w\xrightarrow{1}w_0$ has to cross $\frac{c_1 h_1}{s}$ sets of the form $E_{j,i_1,\dots,i_{j-1},i_{j+1},\dots,i_d}^{(m_1,\dots,m_d)}$ before hitting $w_0$ in $C_1'$.}
\label{w1w0fig2}
\end{figure}
so that there exists a $c_1 >0$ such that in order for the walk started at $w$ to hit $w_0$ in~$A_1$, it will first have to cross at least~$\big\lfloor \frac{c_1  h_1}{s} \big\rfloor$ sets of the form $E_{j,i_1,\dots,i_{j-1},i_{j+1},\dots,i_d}^{(m_1,\dots,m_d)}$.  Each time the walk reaches a set $E_{j',i_1',\dots,i_{j-1}',i_{j+1}',\dots,i_d'}^{(m_1',\dots,m_d')}$, the probability that 
it will reach another set of the form $E_{j,i_1,\dots,i_{j-1},i_{j+1},\dots,i_d}^{(m_1,\dots,m_d)}$ at  distance at least $s$ from $E_{j',i_1',\dots,i_{j-1}',i_{j+1}',\dots,i_d'}^{(m_1',\dots,m_d')}$, 
before hitting either $\partial A_1$ or $\partial A_2$, is bounded from above by a constant $0<c_2 <1$, as can be seen using Donsker's Invariance Principle. 
Using the strong Markov property, we see that the probability that the walk started at $w$ crosses $\big\lfloor \frac{c_1  h_1}{s} \big\rfloor$ sets of the form $E_{j,i_1,\dots,i_{j-1},i_{j+1},\dots,i_d}^{(m_1,\dots,m_d)}$ is bounded from above by $c_{2}^{\lfloor\frac{c_1 h_1}
{s}\rfloor}$. See Figure~$\ref{w1w0fig2}$.

Let $C_1'$ be a discrete ball of radius $s$ contained in $A_1$ such that $C_1'\cap\partial A_1~=~{w_0}$. Recall that the probability that a simple random walk started at a distance of
 order $s$ from $C_1'$ will hit $C_1'$ at $w_0$ is of order $s^{-(d-1)}$, and that it is harder for a walk started at $x\in A_2^C\setminus A_1$ to first hit $A_1$ at $w_0$ then it is for the same walk to first hit $C_1'$ at $w_0$, that is,
\begin{equation*}
\IP_x\big{[}    X_{H_{\partial A_1}}=w_0     \big{]}\leq \IP_x\big{[}    X_{H_{\partial C_1'}}=w_0     \big{]}.
\end{equation*}
In conjunction with last paragraph's argument and the strong Markov property, this shows the existence of a constant $c_3,c_4>0$ such that
\begin{equation}
\label{pathprobfsquare}
\IP_w\big{[} w\xrightarrow{1}w_0 \big{]}\leq c_3 e^{\frac{-c_4  h_1}{s}}s^{-(d-1)}.
\end{equation}

$\boldsymbol{y_0\xrightarrow{3} y} $: The proof of this 
bound is essentially the same as that of the corresponding bound in the case when the underlying sets are balls instead of smoothed hypercubes. 
We have, for some $c_1,c_2>0$, and $h_3:=\dist(y,y_0)$,
\begin{equation}
\label{pathprobhsquare}
\IP_w\big{[} y\xrightarrow{3}y_0 \big{]}\leq c_1 e^{\big\lfloor\frac{-c_2  h_3}{s}\big\rfloor}s^{-(d-1)}s^{-1}.
\end{equation}

$\boldsymbol{w_0\xrightarrow{2'}y_0}$: Let $h$ denote the Euclidean distance between $w_0$ and $y_0$. If $h<100s$, a simple application of the 
Green's function bound gives the desired result. We then assume $h>100s$. Define $\tilde B_{x}$ to be the discrete ball in the $\ell_\infty$-norm centered in $x$ with radius $\frac{h}{4\sqrt{d}}$.

We will break up the path $\gamma_2^0\in w_0\xrightarrow{2'}y_0$ in pieces that are easier to work with. Let~$w_4\in \partial \tilde B_{w_0}\cap A_2^C$, $y_4\in \partial \tilde B_{y_0}\cap A_2^C$. We define the collection of finite paths:
\begin{itemize}
\item$w_0\xrightarrow{5}w_4$: The collection of all finite nearest-neighbor paths starting at $w_0$ whose only intersection with $\partial \tilde B_{w_0}\cup\partial A_2$ is at its ending point $w_4\in\partial \tilde B_{w_0}\cap A_2^C$. It is straightforward to see this collection as a simple random walk event.
\\
\item$w_4\xrightarrow{6}y_4$: The collection of finite nearest-neighbor paths starting at $w_4$ and ending at~$y_4$, without intersecting $\partial A_2$.
\\
\item$y_4\xrightarrow{7}y_0$: The collection of all finite nearest-neighbor paths that start at $y_4$, never return to $\partial \tilde B_{y_0}\cap A_2^C$, and end at $y_0$ without ever reaching $\partial A_2$. It is simple to see this collection as a simple random walk event.
\end{itemize}
\begin{figure}[ht]
\centering
\includegraphics[scale = 1]{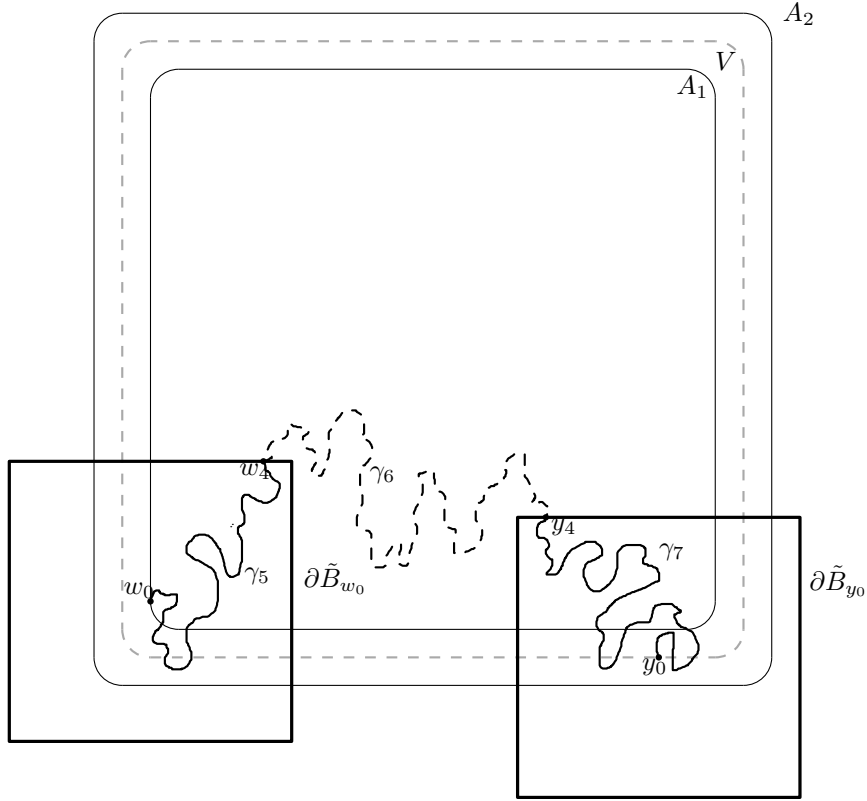}
\caption{Definition of the paths $\gamma_5$, $\gamma_6$ and $\gamma_{7}$.}
\label{w02y0fig4}
\end{figure}
As before, we denote by $w_0\xrightarrow{5}w_4\xrightarrow{6}y_4\xrightarrow{7}y_0$ the concatenation of these three collections.

Analogously to what we noted at the start of this section, we observe that $\gamma_2^0 \in w_0\xrightarrow{2'}y_0$ if and only if there exists $w_4\in \partial \tilde B_{w_0}\cap A_2^C$ and 
$y_4\in \partial \tilde B_{y_0}\cap A_2^C$ such that $\gamma_2^0$ is the concatenation of three paths: $\gamma_5 \in w_0\xrightarrow{5}w_4$, $\gamma_6\in w_4\xrightarrow{6}y_4$, and $\gamma_7\in y_4\xrightarrow{7}y_0$.

We also define
\begin{itemize}
\item$w_4\xrightarrow{6'}y_4$ The collection of finite simple random walk trajectories starting at $w_4$ and ending at its first visit to $y_4$ without intersecting $\partial A_2$. 
This collection can also be seen as the event where the simple random walk started at $w_4$ visits $y_4$ before it hits~$\partial A_2$.
\end{itemize}
Using the same trick we used to obtain the bound~($\ref{pathprobd}$), we can find a constant $c>1$ such that
\begin{equation}
\label{pathprobq}
\IP_{w_4}\big{[} w_4\xrightarrow{6'}y_4 \big{]}\leq \sum_{ \gamma_6 \in w_4\xrightarrow{6}y_4}\frac{1}{2^{|\gamma_6|}} \leq c
\IP_{w_4}\big{[} w_4\xrightarrow{6'}y_4 \big{]}.
\end{equation}
We then have 
\begin{align}
\IP_{w_0}\big{[} w_0\xrightarrow{2'}y_0 \big{]}
&=\sum_{\substack{w_4\in \partial \tilde B_{w_0}\cap A_2^C \\ \gamma_5 \in w_0\xrightarrow{5}w_4}}\frac{1}{2^{|\gamma_5|}}
\sum_{\substack{y_4\in \partial \tilde B_{y_0}\cap A_2^C \\ \gamma_6 \in w_4\xrightarrow{6}y_4}}\frac{1}{2^{|\gamma_6|}}\sum_{\gamma_7 \in y_4\xrightarrow{6}y_0}\frac{1}{2^{|\gamma_7|}}\nonumber  \\ 
&\leq c  \sum_{w_4}\IP_{w_0}\big{[} w_0\xrightarrow{5}w_4 \big{]}\sum_{y_4}\IP_{w_4}\big{[} w_4\xrightarrow{6'}y_4 \big{]}\IP_{y_4}\big{[} y_4\xrightarrow{7}y_0 \big{]}.
\end{align}
We then use the Green's function estimate \eqref{greenestimate} to bound $\IP_{w_4}\big{[} w_4\xrightarrow{6'}y_4 \big{]}$ by $ch^{d-2}$ (note that $\dist(w_4,y_4)=O(h)$). Using this bound on the above equality, we obtain
\begin{equation}
\IP_{w_0}\big{[} w_0\xrightarrow{2'}y_0 \big{]} \leq 
ch^{d-2}\sum_{w_4}\IP_{w_0}\big{[} w_0\xrightarrow{5}w_4 \big{]}\sum_{y_4}\IP_{y_4}\big{[} y_4\xrightarrow{7}y_0 \big{]}.
\end{equation}

We define the events
\begin{itemize}
\item$w_0\xrightarrow{5}\partial\tilde B_{w_0}$: The event where the simple random walk started at $w_0$ reaches $\partial\tilde B_{w_0}$ before reaching  $\partial A_2$.
\\
\item$y_0\xrightarrow{5}\partial\tilde B_{y_0}$: The event where the simple random walk started at $y_0$ reaches $\partial\tilde B_{y_0}$ before reaching $\partial A_2$.
\end{itemize}
Note that 
\begin{equation}
\sum_{w_4}\IP_{w_0}\big{[} w_0\xrightarrow{5}w_4 \big{]}=\IP_{w_0}\big{[} w_0\xrightarrow{5}\partial\tilde B_{w_0} \big{]},
\end{equation}
and using the simple random walk's reversibility, we also have
\begin{equation}
\sum_{y_4}\IP_{y_4}\big{[} y_4\xrightarrow{5}y_0 \big{]}=\IP_{y_0}\big{[} y_0\xrightarrow{5}\partial\tilde B_{y_0} \big{]},
\end{equation}
so that we obtain the following bound
\begin{equation}
\label{5arrow3}
\IP_{w_0}\big{[} w_0\xrightarrow{2'}y_0 \big{]}\leq 
\frac{c}{h^{d-2}}\IP_{w_0}\big{[} w_0\xrightarrow{5}\partial\tilde B_{w_0} \big{]}\IP_{y_0}\big{[} y_0\xrightarrow{5}\partial\tilde B_{y_0} \big{]}.
\end{equation}

We still have to obtain a bound for these last two probabilities. 
Since they are similarly defined, the bound for both of them follows from the same arguments, and thus we will only provide a bound for $\IP_{w_0}\big{[} w_0\xrightarrow{5}\partial\tilde B_{w_0} \big{]}$.

We will do so by looking at the projections of the random walk trajectory in each of the $d$ orthogonal axes. 
Since we will need to look at these projections independently, we will change our object of study from the simple random walk on $\Z^d$ to the continuous time simple random walk on $\Z^d$ with waiting times between steps distributed as $\Exp(1)$ random variables. 
Since we will be studying properties of the random walk's trajectories, this change of framework will in no way impact the probabilities of interest. 
We will denote by $\IP_x^c$, with $x\in\Z^d$, the probability measure associated with such continuous time random walk starting at $x$.

We recall the definition of $\mathfrak{H}_{r+2s}$, the unsmoothed version of $A_2^C$. Here we will assume~$\mathfrak{H}_{r+2s}$ takes the form
\begin{equation*}
\mathfrak{H}_{r+2s}:=\big{\{} (x_1,\dots,x_d)\in\Z^d: 0\leq x_i\leq r+2s,\text{ for all } i=1,\dots,d \big{\}}
\end{equation*}

Without loss of generality we assume that $0\in\Z^d$ is
the point belonging to $\{0,r+2s\}^{d}$ which is closest to $w_0$. We denote $w_0\equiv(w_0^1,\dots,w_0^d)$ and for each $j\in\{1,\dots,d\}$ we let~$(X_t^j,t\geq 0)$ be the projection on the $j$-th axis of the 
continuous time random walk started at $w_0$. This projection is itself a continuous time random walk started at $w_0^j$ with waiting time between jumps given by a $\Exp(d)$ random variable, and, as we already noted, these random walks are independent from each other. 
We will  define $\IP^j_{x}$ to be the probability measure associated with this projected random walk when it starts at $x\in\Z$.

We define, for $j\in\{1,\dots,d\}$ and $A\subset\Z$ , the hitting times
\begin{align}
\tau^j (A)&:=\inf_{t\geq 0}\{X^j_t\in A\},
\\
\tau^j&:=\tau^j (\{\max\{0,w_0^j-h\},w_0^j+h\}),
\end{align}
and we let $J_t^j$ denote the number of jumps the continuous time walk projected on the $j$-th direction makes before time $t$.

Since $J_t^j$ has Poisson distribution with parameter $td^{-1}$, we have (using a convenient large deviation estimate):
\begin{equation}
\label{poissondev}
\IP_{\! w_0^j}^j\Bigg{[}  \frac{(1-\delta)t}{d} \leq J^j_t  \leq  \frac{(1+\delta)t}{d}    \Bigg{]}\geq 1-e^{c(\delta )t}.
\end{equation}

Given $w_0$, we divide the $d$ directions of $\Z^d$ in two kinds. The first kind will be such that $\max\{0,w_0^j-h\}=0$, the second will be such that $\max\{0,w_0^j-h\}=w_0^j-h$. We assume without loss of generality the first $d_0$ directions to be of the first kind and the remaining directions to be of the second kind.

Given $t\in\R_+$, we will need to bound the probability $\IP_{\!  w_0^j}^j\big{[} \tau^j>t    \big{]}$. We first assume $j\leq d_0$. We denote by $(S^x_k,k\in\Z_+)$ the unidimensional discrete time simple random walk starting at $x\in\Z$, and by $\IP_x^\Z$ its associated measure. We have
\begin{equation}
\IP_{\!  w_0^j}^j\big{[} \tau^j>t    \big{]}\leq\IP_{\!  w_0^j}^j\big{[} \tau^j_{\{0\}}>t    \big{]}\leq\sum_{t_0} \IP_{\!  w_0^j}^\Z \big{[} \min_{0\leq k \leq t_0} S_k^{w_0^j}>0    \big{]}\IP_{\!  w_0^j}^c\big{[} J_t^j=t_0   \big{]}.
\end{equation}
Using ($\ref{poissondev}$), the reflection principle for the unidimensional simple random walk, and the central limit theorem, we can bound the above expression by
\begin{align}
\label{tau1}
\lefteqn{\sum_{\substack{t_0\in\Z_+,\\t_0\in\big(\frac{(1-\delta)t}{d},\frac{(1+\delta)t}{d}\big)}}
\IP_{\!  w_0^j}^c\big{[} J_t^j=t_0   \big{]}
\Big(1-2 \IP_{\!  w_0^j}^\Z \big{[}  S_{t_0}^{w_0^j}<0    \big{]}  \Big) +e^{-c(\delta)t}
 }  \phantom{***}\\ \nonumber
& = \sum_{\substack{t_0\in\big(\frac{(1-\delta)t}{d},\frac{(1+\delta)t}{d}\big)}}
\IP_{\!  w_0^j}^c\big{[} J_t^j=t_0   \big{]}
\Big(1-2 \IP_{0}^\Z \big{[}  S_{t_0}^{0}>w_0^j   \big{]}  \Big) +e^{-c(\delta)t}\\ \nonumber
& \leq  c\sum_{\substack{t_0\in\Z_+,\\t_0\in\big(\frac{(1-\delta)t}{d},\frac{(1+\delta)t}{d}\big)}}
\IP_{\!  w_0^j}^c\big{[} J_t^j=t_0   \big{]}\Big(1-\frac{2}{\sqrt{2\pi}}\int_0^\infty e^{-\frac{v^2}{2}}\d v \\
&\quad\quad\quad +\quad\frac{2}{\sqrt{2\pi}}\int_0^{\frac{w_0^j}{\sqrt{t_0}}} e^{-\frac{v^2}{2}}\d v+O(\sqrt{t_0})^{-1}\Big)+e^{-c(\delta)t}\\
& \leq \nonumber c\sum_{\substack{t_0\in\Z_+,\\t_0\in\big(\frac{(1-\delta)t}{d},\frac{(1+\delta)t}{d}\big)}}
\IP_{\!  w_0^j}^c\big{[} J_t^j=t_0   \big{]}\frac{w_0^j}{\sqrt{t_0}}\\
& \leq c\frac{w_0^j}{\sqrt{t}}. \nonumber
\end{align}
In a analogous way, we show for $j>d_0$
\begin{equation}
\label{tau2}
\IP_{\!  w_0^j}^j\big{[} \tau^j>t    \big{]}\leq c\frac{h}{\sqrt{t}}.
\end{equation}

We now bound the probability that the walk exits the sphere $ \partial \tilde B_{w_0} $ through the first direction, without ever hitting $\partial A_2$.
\begin{align}
\label{condcontwalk}
\lefteqn{\IP_{w_0}^c\big{[} X_{\tau^1}^1=w_0^1+h,\tau^j>\tau^1\text{ for all } j\neq 1    \big{]}}\phantom{*******}\\ 
\nonumber & = \IP_{\!  w_0^1}^1\big{[} X_{\tau^1}^1=w_0^1+h\big{]}\int\prod_{j\neq 1}\IP_{\!  w_0^j}^j\big{[} \tau^j>t    \big{]}\IP_{w_0}^c\big{[}\tau^1=t+\d t\mid X_{\tau^1}^1=w_0^1+h\big{]},
\end{align}
where $\IP_{w_0}^c\big{[}\tau^1=t+\d t\mid X_{\tau^1}^1=w_0^1+h\big{]}$ is the distribution of $\tau^1$ conditioned on the event $\{X_{\tau^1}^1=w_0^1+h\}$. Then, using $(\ref{tau1})$, $(\ref{tau2})$ and the gambler's ruin estimate (see Section~$5.1$ of \cite{LawlerLimic}), we are able to bound the above expression by
\begin{equation}
c\frac{w_0^1}{h}\int\prod_{1<j\leq d_0 } \frac{w_0^j}{\sqrt{t}} \prod_{d_0<j\leq d}\frac{h}{\sqrt{t}}\IP_{w_0}^c\big{[}\tau^1=t+\d t\mid X_{\tau^1}^1=w_0^1+h\big{]}.
\end{equation}

We define the continuous time simple random walk
\begin{equation}
X_t^{1,\frac{1}{2}}:=X^1_{t+\tau^1 (\{0,\lceil 2^{-1}(w_0^1+h)\rceil\})},
\end{equation}
so that, on $\{X_{\tau^1}^1=w_0^1+h\}$, $X_t^{1,\frac{1}{2}}$ is distributed  as a continuous time one-dimensional walk starting at a halfway point between $0$ and $w_0^1+h$. We also define the hitting time
\begin{equation}
\tau^{1,\frac{1}{2}}:=\inf\{t\geq 0;X_t^{1,\frac{1}{2}}\in\{0,w_0^1+h\}\}.
\end{equation}

On the event $\{X_{\tau^1}^1=w_0^1+h\}$, $\tau^1$ is distributed as $\tau^1(\{0,\lceil 2^{-1}(w_0^1+h)\rceil\})+\tau^{1,\frac{1}{2}}$, so that $\{\tau^1<t\}$ implies $\{\tau^{1,\frac{1}{2}}<t\}$.

We then have, for $\alpha<1$
\begin{equation}
\IP_{w_0}^c\Big{[}\tau^1<\alpha h^2 d\mid X_{\tau^1}^1=w_0^1+h\Big{]}\leq
\IP_{w_0}^c\Big{[}\tau^{1,\frac{1}{2}}<\alpha h^2 d\mid X_{\tau^1}^1=w_0^1+h\Big{]}.
\end{equation}
Since $X_t^{1,\frac{1}{2}}$ starts at a halfway point between $0$ and $w_0^1+h$, we have
\begin{equation}
\IP_{w_0}^c\Big{[}\tau^{1,\frac{1}{2}}<\alpha h^2 d\mid X_{\tau^1}^1=w_0^1+h\Big{]}\leq
c\IP_{2^{-1}(w_0^1+h)}^1\Big{[}\tau^{1,\frac{1}{2}}<\alpha h^2 d\Big{]}.
\end{equation}
Using ($\ref{poissondev}$) together with a large deviation estimate (see Lemma~$1.5.1$ of \cite{LawlerI}), we obtain
\begin{equation}
\label{halfwaytrick}
\IP_{w_0}^c\big{[}\tau^1<\alpha h^2 d\mid X_{\tau^1}^1=w_0^1+h\big{]}\leq
e^{c\alpha^{-1}}.
\end{equation}
We define
\begin{equation}
\psi_{w_0}(t):=c\frac{w_0^1}{h}\prod_{1<j\leq d_0 } \frac{w_0^j}{\sqrt{t}} \prod_{d_0<j\leq d}\frac{h}{\sqrt{t}}.
\end{equation}
Then
\begin{align*}
\lefteqn{\IP_{w_0}^c\big{[} X_{\tau^1}^1=w_0^1+h,\tau^j>\tau^1\text{ for all } j\neq 1    \big{]}}\phantom{*****}\\
& \leq  c\frac{w_0^1}{h}\int\prod_{1<j\leq d_0 } \frac{w_0^j}{\sqrt{t}} \prod_{d_0<j\leq d}\frac{h}{\sqrt{t}}\IP_{w_0}^c\big{[}\tau^1=t+\d t\mid X_{\tau^1}^1=w_0^1+h\big{]} \\
& = \sum_{k\geq 1}c\frac{w_0^1}{h}\int_{t=h^2 d (k+1)^{-1}}^{t=h^2 d k^{-1}}\prod_{1<j\leq d_0 } \frac{w_0^j}{\sqrt{t}} \prod_{d_0<j\leq d}\frac{h}{\sqrt{t}}\IP_{w_0}^c\big{[}\tau^1=t+\d t\mid X_{\tau^1}^1=w_0^1+h\big{]} \\
&\quad + c\frac{w_0^1}{h}\int_{t\geq h^2 d }\prod_{1<j\leq d_0 } \frac{w_0^j}{\sqrt{t}} \prod_{d_0<j\leq d}\frac{h}{\sqrt{t}}\IP_{w_0}^c\big{[}\tau^1=t+\d t\mid X_{\tau^1}^1=w_0^1+h\big{]}  \\
& \leq \psi_{w_0}(h^2 d)+\sum_{k\geq 1}\psi_{w_0}(h^2 dk^{-1})e^{-ck}.
\end{align*}

Since $\psi_{w_0}(h^2 dk^{-1})$ grows polynomially in~$k$ as $k\rightarrow\infty$, we have
\begin{equation}
\IP_{w_0}^c\big{[} X_{\tau^1}^1=w_0^1+h,\tau^j>\tau^1\text{ for all } j\neq 1    \big{]}\leq c \psi_{w_0}(h^2 d)\leq c'\prod_{1\leq j\leq d_0 }\frac{w_0^j}{h}.
\end{equation}

The proof is analogous for every $j=1,\dots,d$. When $j>d_0$ the calculations are in fact easier because, since $\max\{0,w_0^j-h\}=w_0^j-h$, there is no preferential direction in which the random walk $(X^j_t,t\geq 0)$ has to exit the ball $\partial \tilde B_{w_0}\cap A_2^C$, so that the required conditioning in $(\ref{condcontwalk})$ is simpler. We then have, for $j$ such that $d_0<j\leq d$,
\begin{equation}
\IP_{w_0}^c\big{[}\tau^n>\tau^j\text{ for all } n\neq j    \big{]}\leq c\prod_{1\leq j\leq d_0 }\frac{w_0^j}{h},
\end{equation}
so that
\begin{align}
\label{5arrow1}
\IP\big{[}  w_0\xrightarrow{5}\partial\tilde B_{w_0}  \big{]} 
&\leq \sum_{1\leq k \leq d_0}\IP_{w_0}^c\big{[} X_{\tau^k}^k=w_0^k+h,\tau^n>\tau^k\text{ for all } n\neq j    \big{]} \\
\nonumber &\quad+  \sum_{d_0<k\leq d}\IP_{w_0}^c\big{[} \tau^n>\tau^k\text{ for all } n\neq j    \big{]} \\
\nonumber &\leq c\prod_{1\leq i \leq d_0}\frac{w_0^i}{h}.
\end{align}

We will change the notation so that we are able to express the inequality above in a way that does not uses the fact that $\{0\}^d$ is the vertex of $\{0,r+2s\}^d$ which is closest to~$w_0$. Let $\mathfrak{H}^{d-1}_i$; $i=1,\dots,2d$; denote the $(d-1)$-dimensional~hyperfaces of~$\mathfrak{H}_{r+2s}$, and let $l^{w_0}_i:=\min \{\dist(w_0,\mathfrak{H}^{d-1}_i),h\}$, and $l^{y_0}_i:=\min \{\dist(y_0,\mathfrak{H}^{d-1}_i),h\}$. Then, ($\ref{5arrow1}$) implies
\begin{equation*}
\IP\big{[}  w_0\xrightarrow{5}\partial\tilde B_{w_0}  \big{]}\leq
c\frac{l^{w_0}_1\dots l^{w_0}_{2d}}{h^{2d}},
\end{equation*}
and using the same arguments used above, we can see that
\begin{equation*}
\IP\big{[}  y_0\xrightarrow{5}\partial\tilde B_{y_0}  \big{]}\leq
c\frac{l^{y_0}_1\dots l^{y_0}_{2d}}{h^{2d}}.
\end{equation*}
Together with ($\ref{5arrow3}$), this shows
\begin{equation}
\label{2final}
\IP_{w_0}\big{[} w_0\xrightarrow{2'}y_0 \big{]}\leq 
ch^{-(d-2)}\frac{l^{w_0}_1\dots l^{w_0}_{2d}}{h^{2d}}
\frac{l^{y_0}_1\dots l^{y_0}_{2d}}{h^{2d}}.
\end{equation}

We will also need a matching lower bound. We will continue to use the same notations and conventions. Again we assume $h>100s$, since otherwise the lower bound follows immediately from using a Green's function estimate. We define
\begin{equation}
w_5:=\Big(w_0^1+\frac{h}{4\sqrt{d}},\dots,w_0^{d_0}+\frac{h}{4\sqrt{d}},w_0^{d_0+1},\dots,w_0^d\Big),
\end{equation}
We analogously define $y_5$: Let $e_{i_{d_1}},\dots,e_{i_{d_k}}$ be the vectors in the orthonormal basis of $\R^d$ corresponding to the directions in which the ball 
$B_{\infty}\big(y_0,\frac{h}{4\sqrt{d}}\big)$ 
passes the limits of the hypercube
 $\mathfrak{H}_{r+2s}$. $y_5$ is defined to be
 the point in $A_2^C$ such that 
$$l=d_1,\dots,d_k\implies |\langle y_5-y_0,e_{i_{l}}\rangle | = \frac{h}{4\sqrt{d}},$$
$$n\neq d_1,\dots,d_k\implies|\langle y_5-y_0,e_{i_n}\rangle | = 0$$
$$\text{and}$$ 
$$B_{\infty}\Big(y_5,\frac{h}{4\sqrt{d}}\Big)\subseteq \mathfrak{H}_{r+2s}.$$
Our plan is to describe an event contained in $w_0\xrightarrow{2'}y_0$ with probability matching that of the right side of $(\ref{2final})$. We let
\begin{equation*}
B_{w_5} := B_{\infty}\Big(w_5,\frac{h}{16\sqrt{d}}\Big),
\end{equation*}
and
\begin{equation*}
B_{y_5} := B_{\infty}\Big(y_5,\frac{h}{16\sqrt{d}}\Big),
\end{equation*}
For $w_6\in \partial B_{w_5}$ and $y_6\in \partial B_{y_5}$, 
we define the events
\begin{itemize}
\item$w_0\xrightarrow{8}w_6$: The event where the random 
walk started at $w_0$ hits $\partial B_{w_5}$ before hitting
 $\partial A_2$ and its entrance point in $\partial B_{w_5}$ is $w_6$.
\\
\item$w_6\xrightarrow{9}y_6$: The event where the random walk started at $w_6$ visits $y_6\in\partial B_{y_5}$ before reaching $\partial A_2$.
\\
\item$y_6\xrightarrow{10}y_0$: The event where the simple random walk started at $y_6$ hits $y_0$ before returning to $\partial B_{y_5}$.
\end{itemize}
\begin{figure}[ht]
\centering
\includegraphics[scale = 1]{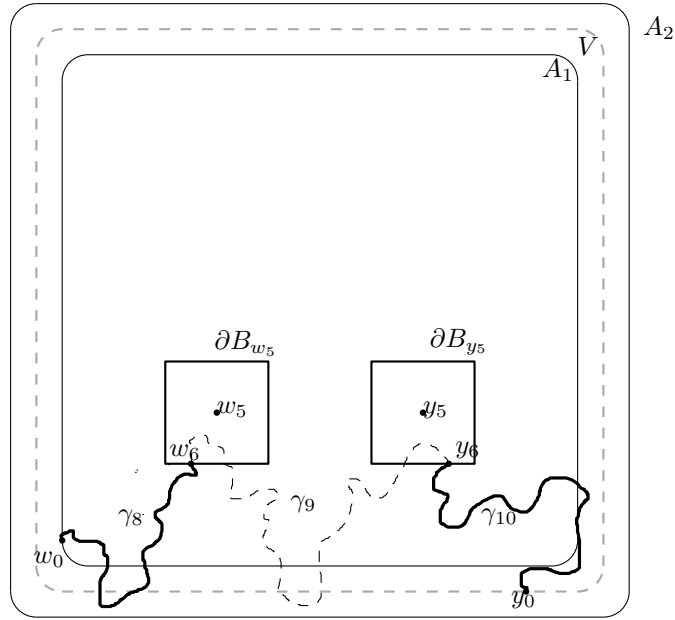}
\caption{Definition of the paths $\gamma_8$, $\gamma_9$ and $\gamma_{10}$.}
\label{w02y0fig5}
\end{figure}

And we denote by $w_0\xrightarrow{8}w_6\xrightarrow{9}y_6\xrightarrow{10} y_0$ the ``concatenation'' of these three events, that is, 
the path $\gamma$ belongs to the event $w_0\xrightarrow{8}w_6\xrightarrow{9}y_6\xrightarrow{10} y_0$ 
if and only if $\gamma$ is the concatenation of three paths: $\gamma_{8}\in w_0\xrightarrow{8}w_6$, $\gamma_{9}\in w_6\xrightarrow{9}y_6$ and $\gamma_{10}\in y_6\xrightarrow{10}y_0$. It is then clear that
\begin{equation*}
\bigcup_{w_6}\bigcup_{y_6}w_0\xrightarrow{8}w_6\xrightarrow{9}y_6\xrightarrow{10} y_0\subset w_0\xrightarrow{2'}y_0,
\end{equation*}
so that; summing over 
$\gamma_{8}\in w_0\xrightarrow{8}w_6$, $\gamma_{9}\in w_6\xrightarrow{9}y_6$ and $\gamma_{10}\in y_6\xrightarrow{10}y_0$;
 we have
\begin{equation}
\begin{array}{e}
\IP_{w_0}\big{[} w_0\xrightarrow{2'}y_0 \big{]}
& \geq &
\sum_{w_6}\sum_{\gamma_8}\frac{1}{2^{|\gamma_8|}}
\sum_{y_6}\sum_{\gamma_9}\frac{1}{2^{|\gamma_9|}}\sum_{\gamma_{10}}\frac{1}{2^{|\gamma_{10}|}} \\[.5 cm] 
& =  & \sum_{w_6}\IP_{w_0}\big{[} w_0\xrightarrow{8}w_6 \big{]}\sum_{y_6}\IP_{w_6}\big{[} w_6\xrightarrow{9}y_6 \big{]}\IP_{y_6}\big{[} y_6\xrightarrow{10}y_0 \big{]}\\[.5 cm] 
& \geq & \frac{c}{h^{d-2}}\sum_{w_6}\IP_{w_0}\big{[} w_0\xrightarrow{8}w_6 \big{]}\sum_{y_6}\IP_{y_6}\big{[} y_6\xrightarrow{10}y_0 \big{]},
\end{array}
\end{equation}
where we bounded  $\IP_{w_6}\big{[} w_6\xrightarrow{9}y_6 \big{]}$ from below by $ch^{d-2}$ using the Green's function estimate \eqref{greenestimate} and the fact that the distance of 
both $w_6$ and $y_6$ from $\partial A_2$ has order $h$.

We define the events
\begin{itemize}
\item$w_0 \xrightarrow{8}\partial B_{w_5}$: 
The event where the simple random walk started at $w_0$ reaches $\partial B_{w_5}$ before reaching $\partial A_2$.
\\
\item$y_0 \xrightarrow{8}\partial B_{y_5}$: The event where the simple random walk started at $y_0$ reaches $\partial B_{y_5}$ before reaching $\partial A_2$.
\end{itemize}
Due to the simple random walk's reversibility, we get that 
\begin{equation}
\label{2boxes}
\IP_{w_0}\big{[} w_0\xrightarrow{2'}y_0 \big{]} \geq \frac{c}{h^{d-2}}
\IP_{w_0}\big{[} w_0\xrightarrow{8}\partial B_{w_5} \big{]}
\IP_{y_0}\big{[} y_0\xrightarrow{10}\partial B_{y_5} \big{]}.
\end{equation}
We will prove a bound for $\IP_{w_0}\big{[} w_0\xrightarrow{8}\partial B_{w_5} \big{]}$, since the bound for $\IP_{y_0}\big{[} y_0\xrightarrow{10}\partial B_{y_5} \big{]}$ follows from analogous arguments. We will use the same continuous time
 random walk projections to study this probability. 
The notation used will be the same as the one used in the proof of the upper bound. For $1\leq j \leq d_0$, we define
\begin{equation*}
\tau^j := \inf \Bigg{\{} t\geq 0,X^j_t\in\Big\{0,w_0^j+\frac{h}{4\sqrt{d}} \Big\}  \Bigg{\}},
\end{equation*}
and on the event $\big\{X^j_{\tau^j}=w_0^j+\frac{h}{4\sqrt{d}}\big\}$, we define $\tau^j_\infty$ as
\begin{equation*}
\tau^j_\infty := \inf \Bigg{\{} t\geq 0,X^j_{t+\tau^j}\in\Big\{w_0^j+\frac{3h}{16\sqrt{d}},w_0^j+\frac{5h}{16\sqrt{d}} \Big\}  \Bigg{\}},
\end{equation*}
that is, the first time after $\tau^j$ when the projection of the continuous time simple random walk on the $j$-th direction hits
 the projected boundary of the ball $B_{w_5}$. For $d_0<m\leq d$, we also define
\begin{equation*}
\tau^m_\infty := \inf \Bigg{\{} t\geq 0,X^m_{t}\in\Big\{w_0^m-\frac{h}{16\sqrt{d}},w_0^m+\frac{h}{16\sqrt{d}} \Big\}  \Bigg{\}},
\end{equation*}
the first time the walk projected in the $m$-th direction hits the projected boundary of the ball $B_{w_5}$. Finally, we define
\begin{equation*}
T := \max_{1\leq i \leq d_0} \tau^i.
\end{equation*}
We then have, for any $l>0$,
\begin{align*}
\lefteqn{\IP_{w_0}\Big{[} w_0\xrightarrow{8}\partial B_{w_5}\Big{]} }
 \phantom{****}\\ 
& \geq  \IP_{w_0}^c\left[\begin{array}{c}  X^j_{\tau^j}=w_0^j+\frac{h}{4\sqrt{d}},\tau^j_\infty>T \text{ for all } j=1,\dots , d_0\phantom{*};\\ \tau^m_\infty>T\text{ for all } m=d_0+1,\dots , d  \end{array}\right] \\
& \geq \IP_{w_0}^c\Big{[} T < l              \Big{]}
\prod_{1\leq j \leq d_0}   \IP_{\!  w_0^j}^j\Big{[}  X^j_{\tau^j}=w_0^j+\frac{h}{4\sqrt{d}},\tau^j_\infty>l \Big{]} \prod_{d_0< m \leq d}\IP_{\!  w_0^m}^m\Big{[} \tau^m_\infty>l \Big{]}.
\end{align*}
Now, let $c_1,c_2>0$ be such that $c_1>c_2$. We have that
\begin{equation*}
\IP_{w_0}^c\big{[} T < c_1 h^2              \big{]} \geq c >0.
\end{equation*}
For each $j=1,\dots,d_0$, we have, by the strong Markov property,
\begin{align*}
\IP_{\!  w_0^j}^j\Big{[}  X^j_{\tau^j}=w_0^j+\frac{h}{4\sqrt{d}},\tau^j_\infty>c_1 h^2  \Big{]}
&\geq
\IP_{\!  w_0^j}^j\Big{[}  X^j_{\tau^j}=w_0^j+\frac{h}{4\sqrt{d}},\tau^j_\infty>c_1 h^2 , \tau^j>c_2 h^2 \Big{]} 
\\
&\geq
\IP_{\!  w_0^j}^j\Big{[}  X^j_{\tau^j}=w_0^j+\frac{h}{4\sqrt{d}}, \tau^j>c_2 h^2 \Big{]}
\\
&\quad\times
\IP_{\!  w_0^j+\frac{h}{4\sqrt{d}}}^j\Big{[} H_{\partial B_{w_5}} > (c_1-c_2) h^2\Big{]}
\\
&\geq
c\IP_{\!  w_0^j}^j\Big{[}  X^j_{\tau^j}=w_0^j+\frac{h}{4\sqrt{d}}, \tau^j>c_2 h^2 \Big{]}
\\
&\geq
c\IP_{\!  w_0^j}^j\Big{[}  X^j_{\tau^j}=w_0^j+\frac{h}{4\sqrt{d}}\Big{]} \IP_{\!  w_0^j}^j\Big{[}\tau^j>c_2 h^2\Big\vert X^j_{\tau^j}=w_0^j+\frac{h}{4\sqrt{d}} \Big{]}.
\end{align*}
Using ($\ref{halfwaytrick}$), we can see that
\begin{equation*}
\IP_{\!  w_0^j}^j\Big{[}\tau^j>c_2 h^2\Big\vert X^j_{\tau^j}=w_0^j+\frac{h}{4\sqrt{d}}\Big{]}>c>0,
\end{equation*}
so that
\begin{align*}
\IP_{\!  w_0^j}^j\Big{[}  X^j_{\tau^j}=w_0^j+\frac{h}{4\sqrt{d}},\tau^j_\infty>c_1 h^2  \Big{]}
&\geq
c\IP_{\!  w_0^j}^j\Big{[}  X^j_{\tau^j}=w_0^j+\frac{h}{4\sqrt{d}}\Big{]}
\\
&\geq c\frac{w_0^j}{h}. 
\end{align*}
For each $m=d_0+1,\dots,d$, it is elementary to see that
\begin{equation*}
\IP_{\!  w_0^m}^m\big{[} \tau^m_\infty>c_1 h^2 \big{]}\geq c >0.
\end{equation*}
Collecting the above equations, we obtain that
\begin{equation}
\label{3final}
\IP_{w_0}\big{[} w_0\xrightarrow{8}\partial B_{w_5} \big{]}\geq \prod_{1\leq i \leq d_0}\frac{w_0^i}{h}.
\end{equation}
Together with ($\ref{2boxes}$) and using the new notation, we have established the bounds:
\begin{equation}
\label{5final}
ch^{-(d-2)}\frac{l^{w_0}_1\dots l^{w_0}_{2d}}{h^{2d}}
\cdot\frac{l^{y_0}_1\dots l^{y_0}_{2d}}{h^{2d}}\leq\IP_{w_0}\big{[} w_0\xrightarrow{2'}y_0 \big{]}\leq 
c'h^{-(d-2)}\frac{l^{w_0}_1\dots l^{w_0}_{2d}}{h^{2d}}
\cdot\frac{l^{y_0}_1\dots l^{y_0}_{2d}}{h^{2d}}.
\end{equation}

$\boldsymbol{w\xrightarrow{4}y}$: Again we let $h_4:=\dist(w,y)$, and again we suppose $h_4>100s$, since a elementary application of the estimate for the Green's function proves the case when $h_4<100s$. 
We will start with the lower bound. Let $C_3$ be a discrete ball of radius $s$ contained in~$A_2^C$ such that $\partial A_2\cap C_3=\{y\}$. Let $C_4$ be a discrete ball of radius $\frac{s}{4}$ concentric with $C_3$. 
Then, the probability that the walk started at $y$ hits $C_4$ before returning to $\partial A_2$ is bigger than the probability that it hits $C_4$ before returning to $C_3$, and has order $s^{-1}$. 
Now, for every point $\tilde{y}\in\partial C_4$, we bound $\IP_{w}\big{[} w\xrightarrow{2'}\tilde{y} \big{]}$ from below in exactly the same way as we bounded~$\IP_{w_0}\big{[} w_0\xrightarrow{2'} y_0 \big{]}$. So that, using the walk's reversibility, the fact that $h_4> 100s$, and the same notation introduced above, we have
\begin{align*}
\IP_{w}\big{[} w\xrightarrow{4} y \big{]}
&\geq 
\sum_{\tilde{y}\in\partial C_4}\IP_{w}\big{[} w\xrightarrow{2'} \tilde{y} \big{]}\IP_{y}\big{[} H_{\partial C_4}<H_{\partial A_2},X_{H_{\partial C_4}}=\tilde{y} \big{]}\\
&\geq
cs^{-1}\inf_{\tilde{y}\in\partial C_4}\IP_{w}\big{[} w\xrightarrow{2'} \tilde{y} \big{]}\\
&\geq
cs^{-1}h_{4}^{-(d-2)}\inf_{\tilde{y}\in\partial C_4}
\frac{l^{w_0}_1\dots l^{w_0}_{2d}}{h_{4}^{2d}}
\cdot\frac{l^{\tilde{y}}_1\dots l^{\tilde{y}}_{2d}}{h_{4}^{2d}}.
\end{align*}

For the upper bound, let $C_3'$ be a discrete ball of radius $s$ contained in $A_2\cup\partial A_2$ such that $\partial A_2\cap C_3'=\{y\}$. Let $C_4'$ be a discrete ball of radius $2s$ concentric with $C_3'$. Then
\begin{align*}
\IP_{w}\big{[} w\xrightarrow{4} y \big{]}
&\leq 
\sum_{\hat{y}\in\partial C_4'}\IP_{w}\big{[} w\xrightarrow{2'} \hat{y} \big{]}\IP_{y}\big{[} H_{\partial C_4'}<H_{\partial A_2},X_{H_{\partial C_4'}}=\hat{y} \big{]}\\
&\leq
cs^{-1}\inf_{\hat{y}\in\partial C_4'}\IP_{w}\big{[} w\xrightarrow{2'} \hat{y} \big{]}\\
&\leq
cs^{-1}h_{4}^{-(d-2)}\sup_{\hat{y}\in\partial C_4'}
\frac{l^{w_0}_1\dots l^{w_0}_{2d}}{h_{4}^{2d}}
\frac{l^{\hat{y}}_1\dots l^{\hat{y}}_{2d}}{h_{4}^{2d}}.
\end{align*}

Using ($\ref{pathprobfsquare}$) and ($\ref{pathprobhsquare}$) we see that the supremum in ($\ref{supprob}$) is reached when~$h_1$ and~$h_3$ are of order $s$. This way,~$h$ should have the same order as~$h_4$. We have proved the following proposition:

\begin{proposition}
\label{p_boundprosquare}
Regarding the sets $A_1^{\tiny\Square},V^{\tiny\Square}$ and $A_2^{\tiny\Square}$, we have that, using the notation defined above, for some constants $c_1$, $c_2$, $c_3$, $c_4$, $c_5$, $c_6$, $c_7$, $c_8$, $c_9> 0$, the following bounds are valid:
\begin{equation*}\IP_w\big{[} w\xrightarrow{1}w_0 \big{]}\leq c_1 \exp\Big(\frac{-c_2 h_1}{s}\Big)s^{-(d-1)},\end{equation*}
\begin{equation*}\IP_w\big{[} y\xrightarrow{3}y_0 \big{]}\leq c_3 \exp\Big(\frac{-c_4 h_3}{s}\Big)s^{-(d-1)}s^{-1},\end{equation*}
\begin{equation*} c_5 h^{-(d-2)}\frac{l^{w_0}_1\dots l^{w_0}_{2d}}{h^{2d}}
\frac{l^{y_0}_1\dots l^{y_0}_{2d}}{h^{2d}}\leq\IP_{w_0}\big{[} w_0\xrightarrow{2'}y_0 \big{]}\leq 
c_6 h^{-(d-2)}\frac{l^{w_0}_1\dots l^{w_0}_{2d}}{h^{2d}}
\frac{l^{y_0}_1\dots l^{y_0}_{2d}}{h^{2d}},\end{equation*}
\begin{equation*}  c_7 s^{-1}h_{4}^{-(d-2)}\inf_{\tilde{y}\in\partial C_4}
\frac{l^{w_0}_1\dots l^{w_0}_{2d}}{h_{4}^{2d}}
\frac{l^{\tilde{y}}_1\dots l^{\tilde{y}}_{2d}}{h_{4}^{2d}}
\leq\IP_{w_0}\big{[} w\xrightarrow{4}y \big{]} \leq 
c_8s^{-1}h_{4}^{-(d-2)}\sup_{\hat{y}\in\partial C_4'}
\frac{l^{w_0}_1\dots l^{w_0}_{2d}}{h_{4}^{2d}}
\frac{l^{\hat{y}}_1\dots l^{\hat{y}}_{2d}}{h_{4}^{2d}}.
\end{equation*}
\begin{equation*}\sup_{\substack{w\in V \\ y\in \partial A_2}}\IP_w\big{[} w\xrightarrow{1}w_0\xrightarrow{2}y_0\xrightarrow{3}y \mid w\xrightarrow{4}y \big{]}\leq c_9 s ^{-2(d-1)} .\end{equation*}
\end{proposition}

\subsection{Proof of Lemma ~\ref{l_expecslt}}\label{s_t2}

Let $z\in\Sigma$ be such that $ \Xi(z)=(w_0,y_0)$, and again let~$h$ stand for the Euclidean distance between $w_0$ and $y_0$. We let $\pi(w_0,y_0)$ be defined in the same way as in $(\ref{l_expecslt})$.
 Given a simple random walk trajectory $\varrho$ started in a set $B$ containing~$V$, we define $\mathcal{C}_{w_0,y_0}^{B}(\varrho)$ to be the function that counts how many times the random walk trajectory $\varrho$ makes an excursion on $A_2^C$ that enters $A_1$ at $w_0$, and $y_0$ 
is the last point such excursion visits on $V$ before reaching $\partial A_2$. We let $\mathcal{C}_{w_0,y_0}^{B}$ denote the random variable $\mathcal{C}_{w_0,y_0}^{B}(\bar{\varrho})$ when $\bar{\varrho}$'s first point is chosen according to $\bar{e}_B$. Proposition $\ref{t_expslt}$ then implies
\begin{equation*}
\pi(w_0,y_0)= \IE(\mathcal{C}_{w_0,y_0}^{V}).
\end{equation*}

Define $\tilde{V}:=\partial B(0,3(r+s))$, the discrete sphere of radius $3(r+s)$.
We define 
\begin{equation*}
\tilde{\pi}(w_0,y_0):=\IE(\mathcal{C}_{w_0,y_0}^{\tilde{V}}).
\end{equation*}
From the compatibility of the laws defined in \eqref{interlacementsdef}, one can see that (see also the proof of Lemma~$6.2$ of~\cite{PopovTeixeira}):
\begin{equation*}
u\capacity(\tilde{V})\IE(\mathcal{C}_{w_0,y_0}^{\tilde{V}})=u\capacity(V)\IE(\mathcal{C}_{w_0,y_0}^{V}).
\end{equation*}
Since $\capacity(\tilde{V})\asymp\capacity(V)$, if we successfully estimate $\tilde{\pi}(w_0,y_0)$ we will automatically be provided with an estimate for $\pi(w_0,y_0)$. We changed the problem from estimating
 $\pi(w_0,y_0)$ to estimating  $\tilde{\pi}(w_0,y_0)$ so that the distance between the simple random walk's starting point and $w_0$ does not affect our calculations.

First we note that $\mathcal{C}_{w_0,y_0}^{\tilde{V}}$ is dominated by a Geometric $(c_1)$ random variable, for some $0<c_1<1$. This follows from the fact that every time the simple random walk exits $A_2^C$, with probability uniformly greater than some constant $1-c_1>0$,
 the walk never returns to $w_0$. This way, it will be sufficient to estimate the probability $\GP [\mathcal{C}_{w_0,y_0}^{\tilde{V}}\geq 1]$ for our purposes.

So, for a walk started at $\tilde{V}$ to reach $w_0$, it first has to hit a discrete sphere $\partial C_1$ of radius~$\frac{s}{2}$ centered on $w_0$. The probability of such event is of order $\frac{s^{d-2}}{r^{d-2}}$, by Proposition~$6.4.2$ of \cite{LawlerLimic}. 

Let $C_2$ be a discrete ball of radius $s$ contained in $A_1$ such that $C_2\cap A_1 =\{w_0\}$. We also let $C_3$ be a discrete ball of radius $2s$ lying outside $A_1$ such that  $C_3\cap A_1 =\{w_0\}$. Using Proposition 6.5.4 of \cite{LawlerLimic} we have, for any $x'\in\partial C_1\cap A_!^C$ and some constant $c_2>0$:
\begin{equation*}
\IP_{x'}\big{[}  X_{H_{A_1}}=w_0     \big{]} \leq \IP_{x'}\big{[}  X_{H_{C_2}}=w_0     \big{]}\leq c_2 s^{-(d-1)}.
\end{equation*}
Then, recalling the notation $f_{A_1}(w_0,y_0):=\IP_{w_0}\big{[} w_0\xrightarrow{2'}y_0 \big{]}$ and the fact that $\capacity (V)\asymp r^{(d-2)}$, and using the strong Markov property, we get, for constants $c,c_1>0$:
\begin{equation}
\pi(w_0,y_0)\leq c \GP [\mathcal{C}_{w_0,y_0}^{\tilde{V}}\geq 1] \leq c_1\capacity(V)^{-1}s^{-1}f_{A_1}(w_0,y_0).
\end{equation}

For the lower bound, we let $C_4$ be a discrete ball of radius $\frac{s}{4}$ contained in $A_2^C\setminus B(0,r+s)$ such that for every $x\in C_4$, $\dist(x,w_0)\leq 2s$. Using the strong Markov property, we get
\begin{equation*}
\GP\big{[}  \mathcal{C}_{w_0,y_0}^{\tilde{V}} \geq 1    \big{]} 
\geq \inf_{x\in\tilde{V}}\IP_x\big{[} H_{C_4}<\infty   \big{]}    
\inf_{x''\in C_4}\IP_{x}\big{[}  X_{C_3}=w_0   \big{]} 
f_{A_1}(w_0,y_0),
\end{equation*}
so that, using Proposition 6.4.2 of \cite{LawlerLimic} we have, for some constant $c_3>0$,
\begin{equation*}
\pi(w_0,y_0)\geq c_3 \capacity(V)^{-1}s^{-1}f_{A_1}(w_0,y_0).
\end{equation*}
The part~(ii) then follows from (i) and Proposition $\ref{t_2ndmoment}$.

\subsection{A lower bound for $\alpha$ }\label{s_t3} 

Let $z\in\Sigma$ be such that $\Xi(z)=(w_0,y_0)$, let $c_4 >0$ be some positive real number. For 
\begin{equation*}
\Gamma_{w_0,y_0} :=\{(w_0 ',y_0 ')\in V\times \partial A_2 ;\space\space \max \{||w_0 ' - w_0||,||y_0 ' - y_0||\} \leq c_4 s \}
\end{equation*} 
and
\begin{equation*} 
\alpha  := \inf\Big\{\frac{g_{(w,y)}(z')}{g_{(w,y)}({\hat z})}; 
  (w,y) \in V\times\partial A_2, z' \in \Gamma_{w_0,y_0},\hat z \in \mathcal{K}\Big\}.
\end{equation*}
We need to find a constant lower bound for $\alpha$. Such lower bound will be provided if we bound the ratios:
\begin{equation} 
\label{balpha}
\inf_{||w_0 ' - w_0||\leq c_4 s} \frac{\IP_w\big{[} w\xrightarrow{1}w_0 ' \big{]}}{\IP_w\big{[} w\xrightarrow{1}w_0  \big{]}}  ,\space\space\space \inf_{||y_0 ' - y_0||\leq c_4 s} \frac{\IP_y\big{[} y\xrightarrow{3'}y_0 ' \big{]}}{\IP_y\big{[} y\xrightarrow{3'}y_0 \big{]}}
\end{equation}
as the other terms of the product
\begin{equation*}
\IP_w\big{[} w\xrightarrow{1}w_0 \big{]} \IP_{w_0}\big{[} w_0\xrightarrow{2'}y_0 \big{]} \IP_y\big{[} y\xrightarrow{3'}y_0 \big{]} \IP_{w}\big{[} w\xrightarrow{4}y \big{]}^{-1}=g_{(w,y)}(z)
\end{equation*}
already have matching lower and upper bounds. Since the ratios in $(\ref{balpha}) $ are very similarly defined, we will only give a lower bound to the first one. We define:
\begin{equation*}
 D = \Big\{x \in \Z^d \setminus A_1 : 
  \dist(x,A_1)\leq 
\frac{s}{8} \text{ and }
    \max\{\dist(x,w_0)\dist(x,w_0 ')\}\leq c_4 s\Big\},
\end{equation*}
and
\begin{equation*}
 \label{e:Ghat}
 \hat D = \{ x \in D: \text{ there exists }v \in \Z^d \setminus 
   (A_1 \cup D) \text{ such that }x\leftrightarrow v\}.
\end{equation*}
One can think of $\hat D$ as the part of the internal boundary of $D$ that is not adjacent to $A_1$.

Proposition $8.7$ of \cite{PopovTeixeira} then says
\begin{equation}
\label{e_popovteixeira2}
\inf_{\substack{x\in\hat D \\ \IP_x[X_{H_{A_1}}=w_0 ']>0}} \frac{\IP_x[X_{H_{A_1}}=w_0 ]}{\IP_x[X_{H_{A_1}}=w_0 ']}>c_4>0.
\end{equation}
Informally the above inequality says that if a random walk is sufficiently away from the points $w_0$ and $w_0 '$,
 but somewhat close to $\partial A_1$, then the probabilities that such walk hits either $w_0$ or $w_0 '$ are comparable.

Changing the notation, we have
\begin{equation}
 \frac{\IP_w\big{[} w\xrightarrow{1}w_0 \big{]}}{\IP_w\big{[} w\xrightarrow{1}w_0 ' \big{]}}=\frac{\IP_w[X_{H_{A_1\cup\partial A_2}}=w_0 ]}{\IP_w[X_{H_{A_1\cup\partial A_2}}=w_0 ']}.
\end{equation}
Using the strong Markov property we can rewrite the above ratio between probabilities as the ratio between the sums:
\begin{equation}
\frac{\IP_w[X_{H_{A_1\cup\partial A_2}}=w_0 ]}{\IP_w[X_{H_{A_1\cup\partial A_2}}=w_0 ']}=\frac{\sum_{x\in\hat D}\IP_w[X_{H_{\hat D\cup\partial A_2\cup A_1}}=x]\IP_x[X_{H_{A_1\cup\partial A_2}}=w_0]}{\sum_{x\in\hat D}\IP_w[X_{H_{\hat D\cup\partial A_2\cup A_1}}=x]\IP_x[X_{H_{A_1\cup\partial A_2}}=w_0 ']}.
\end{equation}
But at the same time
\begin{equation*}
\IP_x[X_{H_{A_1\cup\partial A_2}}= w_0 ] 
\geq \IP_x[X_{H_{A_1}}= w_0 ] - \IP_x[H_{\partial A_2}<H_{A_1} ]\sup_{x'\in \partial A_2}\IP_{x'}[X_{H_{A_1}}= w_0 ].
\end{equation*}
By the usual trick of considering the probabilities 
of hitting and escaping certain well placed discrete balls, we are able to see that both terms in the right side of the inequality have order $ \dist(x,\partial A_1) s^{-1}s^{-(d-1)}$. 
We can then fine-tune the constant $c_4$ in the definition of $\Gamma_{w_0,y_0}$ in such a way that 
\begin{equation*}
\IP_x[X_{H_{A_1\cup\partial A_2}}= w_0 ] 
\geq c\IP_x[X_{H_{A_1}}= w_0 ],
\end{equation*}
for some constant $c>0$. The same is valid for $w_0'$, so that 
\begin{equation*}
\frac{\sum_{x\in\hat D}\IP_w[X_{H_{\hat D\cup\partial A_2\cup A_1}}=x]\IP_x[X_{H_{A_1\cup\partial A_2}}=w_0]}{\sum_{x\in\hat D}\IP_w[X_{H_{\hat D\cup\partial A_2\cup A_1}}=x]\IP_x[X_{H_{A_1\cup\partial A_2}}=w_0 ']}
\geq
c\frac{\sum_{x\in\hat D}\IP_w[X_{H_{\hat D\cup\partial A_2\cup A_1}}=x]\IP_x[X_{H_{A_1}}=w_0]}{\sum_{x\in\hat D}\IP_w[X_{H_{\hat D\cup\partial A_2\cup A_1}}=x]\IP_x[X_{H_{A_1}}=w_0 ']}.
\end{equation*}

Using $(\ref{e_popovteixeira2})$ again we obtain
\begin{equation}
\inf_{w_0 ':||w_0' - w_0||}\frac{\IP_w\big{[} w\xrightarrow{1}w_0 \big{]}}{\IP_w\big{[} w\xrightarrow{1}w_0 ' \big{]}}\geq c_2 > 0.
\end{equation}
This fact together with the arguments 
presented above show the existence of a constant~$c>0$ such that $\alpha\geq c$, which concludes the proof of the uniform lower bound.
\section*{Acknowledgments}
Caio Alves was supported by FAPESP (grant 2013/24928-2). 
Serguei Popov was supported by CNPq (grant 300886/2008–0) and FAPESP (grant 2009/52379–8).

\bibliographystyle{plain}
\bibliography{./all}

\end{document}